\documentclass[11pt,reqno]{amsart}

\usepackage{amsmath,amssymb,amsthm}
\usepackage{bm}
\usepackage{natbib}
\usepackage{graphicx}
\usepackage{fullpage}

\usepackage{multirow}

\newtheorem{theorem}{Theorem}[section]

\newtheorem{lemma}{Lemma}[section]

\newtheorem{assumption}{Assumption}[section]
\theoremstyle{definition}

\newtheorem{remark}{Remark}[section]
\newtheorem{example}{Example}[section]

\newcommand{\R}{\mathbb{R}}
\newcommand{\mF}{\mathcal{F}}
\newcommand{\mG}{\mathcal{G}}

\newcommand{\mC}{\mathcal{C}}
\newcommand{\mD}{\mathcal{D}}
\newcommand{\mS}{\mathcal{S}}

\newcommand{\U}{U}
\newcommand{\Ep}{\mathrm{E}}
\renewcommand{\Pr}{\mathrm{P}}
\renewcommand{\tilde}{\widetilde}
\renewcommand{\hat}{\widehat}

\DeclareMathOperator{\Var}{Var}
\DeclareMathOperator{\Cov}{Cov}

\DeclareMathOperator{\E}{E}
\DeclareMathOperator{\Prob}{P}

\begin{document}

\title[Confidence bands for EIV regression]{Uniform confidence bands for nonparametric errors-in-variables regression}

\author[K. Kato]{Kengo Kato}
\author[Y. Sasaki]{Yuya Sasaki}

\date{First arXiv version: February 11, 2017. This version: \today}

\address[K. Kato]{
Department of Statistics and Data Science, Cornell University \\
1194 Comstock Hall, Ithaca, NY 14853, USA.
}
\email{kk976@cornell.edu}

\address[Y. Sasaki]{
Department of Economics, Vanderbilt University \\
VU Station B, Box \#351819, 2301 Vanderbilt Place, Nashville, TN 37235, USA. 
}
\email{yuya.sasaki@vanderbilt.edu}

\begin{abstract}
This paper develops a method to construct uniform confidence bands for a nonparametric regression function where a predictor variable is subject to a measurement error. We allow for the distribution of the measurement error to be unknown, but assume the availability of validation data or repeated measurements on the latent predictor variable. The proposed confidence band builds on the deconvolution kernel estimation and a novel application of the multiplier bootstrap method. We establish asymptotic validity of the proposed confidence band. To our knowledge, this is the first paper to derive asymptotically valid uniform confidence bands for nonparametric errors-in-variables regression. 
\end{abstract}

\maketitle

\section{Introduction}

Consider the nonparametric errors-in-variables (EIV) regression model with classical measurement error
\begin{equation}
Y=g(X) +\epsilon, \ W=X+\U, \ \Ep[\epsilon \mid X,\U] = 0, \label{eq: EIV}
\end{equation}
where each of $Y,X,\epsilon,W$, and $\U$ is a univariate random variable, and $\U$ is independent from $X$.
We observe $(Y,W)$, but observe neither $X$ nor $\U$.\footnote{Although we state the independence and mean independence restrictions as a concise sufficient condition for the identification, we remark that they may be relaxed as in \cite{Sc13}. Specifically, it suffices to assume that $\Ep[e^{it(X+U)}]=\Ep[e^{itX}]\Ep[e^{itU}]$ and $\Ep[g(X)e^{it(X+U)}]=\Ep[g(X)e^{itX}]\Ep[e^{itU}]$ for all $t \in \mathbb{R}$.}
Furthermore, we assume that the distribution of $\U$ is unknown. The variable $X$ is a latent predictor variable, while $\U$ is a measurement error. 
Of interest are estimation of and inference on the regression function $g(x)=\E[Y \mid X=x]$.  
In particular, we are interested in constructing uniform confidence bands for $g$.
Confidence bands provide a simple graphical description of the extent to which a nonparametric estimator varies at design points, thereby quantifying uncertainties of the nonparametric estimator. 
However, construction of confidence bands tends to be challenging, especially for complex nonparametric models (we refer to \cite{Wa06,HaHo13,GiNi16} as general references on confidence bands in nonparametric statistical models). 

Indeed, despite the rich literature on consistent estimation of nonparametric EIV regression, the literature on pointwise or uniform confidence bands for nonparametric EIV regression is  limited -- see below for a literature review -- likely because of its complexity. 
Even pointwise inference on $g$ under the assumption that the measurement error distribution is known is considered by experts to be difficult.\footnote{\citet{DeHaJa15}, who study pointwise confidence bands for nonparametric EIV regression under the assumption that the measurement error distribution is known, state that ``despite their practical importance, to our knowledge confidence bands in nonparametric EIV regression have largely been ignored so far. We show that the problem is particularly complex, much more so than in the standard error-free setting.'' \citep[][p.149]{DeHaJa15}} This is because, as discussed in \cite{DeHaJa15}: 1) the asymptotic variance of the ``deconvolution kernel'' estimator of $g$ is non-trivial to estimate and so inference based on limiting distributions is difficult to implement; and 2) it is not straightforward to devise a way to implement bootstrap for inference on $g$ due to the unobservability of $X$ in data.

With all these challenges recognized in the literature, the present paper attempts to solve an even more challenging problem of constructing \textit{uniform} confidence bands for the regression function $g$ \textit{without} assuming that the measurement error distribution is known. 
To deal with unknown measurement error distribution, we assume that there is an independent sample $\{ \eta_{1},\dots,\eta_{m} \}$ from the measurement error distribution, in addition to an independent sample $\{ (Y_{1},W_{1}),\dots,(Y_{n},W_{n}) \}$ from the distribution of $(Y,W)$,  where $m=m_{n} \to \infty$ as $n \to \infty$. (The auxiliary sample $\{ \eta_{1},\dots,\eta_{m} \}$ need not be independent from $\{ (Y_{1},W_{1}),\dots,(Y_{n},W_{n}) \}$.) 
Real data scenarios of such additional data availability that are plausible in economics, social sciences, and biomedical sciences include:
the case where validation data is available for data combination; and
the case where repeated measurements on $X$ with errors one of which is symmetrically distributed are available.
These patterns of data requirements are often considered in the existing literature with measurement errors that we review below.

Under this setup, we develop a method to construct confidence bands for the regression function $g$. Our method builds on the deconvolution kernel estimation \citep{FaTr93}, and a novel application of the multiplier (or wild) bootstrap method. 
Bootstrap methods for confidence bands for nonparametric functions have been established previously -- see, for example, \cite{HaHo13,ChChKa14b} that also use Gausssian process approximations.
Our construction of the multiplier process differs from the standard approach in the error-free case \citep[cf.][]{NePo98}, and is tailored to EIV regression; see Remark \ref{rem: multiplier process} ahead.
Building on non-trivial applications of the probabilistic techniques developed in \cite{ChChKa14a,ChChKa14b,ChChKa16}, we establish asymptotic validity of the proposed confidence band, i.e., the proposed confidence band contains the true regression function with probability approaching the nominal coverage probability.  In the present paper, as in \cite{BiDuHoMu07, ScMuDu13,DeHaJa15} that study inference in deconvolution and EIV regression, we focus for a technical reason on the case where the measurement error density is \textit{ordinary smooth}, i.e., the characteristic function of the measurement error distribution decays at most polynomially fast in the tail \citep[cf.][]{Fa91a,FaTr93}. 

In addition to these contributions, we also provide a practical guideline by adapting the regularization method of \cite{DeHaMe08} and the bandwidth selection procedures of \cite{DeHa08,DeHaJa15} to our framework.
We conduct simulation studies to verify the finite sample performance of the proposed confidence band. The simulation studies show that the proposed confidence band, combined with the proposed bandwidth selection rule, works well. 
Applying our method to a combination of the two data sets, the National Health and Nutrition Examination Survey (NHANES) and the Panel Survey of Income Dynamics (PSID), we draw confidence bands for nonparametric regressions of medical costs on the body mass index (BMI), accounting for measurement errors in BMI.
Finally, we discuss extensions of our results to specification testing, cases with additional error-free regressors, and confidence bands for conditional distribution functions.

\subsection{Related Literature}

In order to locate the present paper in the context of the relevant literature, it is useful to first review measurement error models and deconvolution.
We refer to books by \cite{CaRuStCr06,Me09,Ho09} and surveys by \cite{ChHoNe11} and \cite{Sc16} for general references. 
The genesis of this literature features the deconvolution kernel density estimation with known error distributions \citep{CaHa88,StCa90,Fa91a,Fa91b}, followed by that with unknown error distributions \citep{DiHa93,HoMa96,Ne97,Ef97,LiVu98,DeHaMe08,Jo09,CoLa11}. \citet{DiHa93, Ne97,Ef97,Jo09,CoLa11} assume the availability of an auxiliary  sample from the measurement error distribution, while \citet{HoMa96,DeHaMe08} assume repeated measurements with symmetrically and identically distributed errors. For repeated measurements without symmetry of error distributions, \citet{LiVu98} propose an alternative density estimator based on Kotlarski's lemma \citep[cf.][]{Ko67} that does not require known error distribution; see also \cite{BoRo10,CoKa15} for further developments.
More recently, \citet{DeHa16}  propose a new approach of non-parametric deconvolution when the error distribution is unknown and there is no additional data available.
Methods to construct confidence bands in deconvolution are developed by \cite{BiDuHoMu07,BiHo08,EsGu08,LoNi11,ScMuDu13} for the case of known error distribution, and more recently by \cite{KaSa16} for the case of unknown error distribution. See also \cite{AdOtWh16} for uniform confidence bands for the distribution function in a deconvolution problem. 

Similarly to the density estimation, the literature on nonparametric EIV regression estimation often takes the deconvolution kernel approach.
\citet{FaTr93} propose to substitute the deconvolution kernel in the Nadaraya-Watson estimator -- also see \cite{FaMa92} for pointwise asymptotic normality, \cite{DeMe07} for extensions to heteroscedastic measurement errors, \cite{DeFaCa09} for local polynomial extensions, and \cite{DeHaJa15} for pointwise inference.
\citet{DeHaMe08} estimate the error characteristic function using repeated measurements on $X$ with symmetrically and identically distributed errors, and substitute the estimated error characteristic function into the deconvolution kernel. 
\citet{Li02} and \citet{Sc04} also work with cases with repeated measurements but without assuming symmetry of error distributions, and propose alternative approaches to estimate the regression function based on Kotlarski's lemma. See also \cite{CaMaRu99,ScWhCh12,ScHu13,HuSa15}. 

Our method of inference is based on the deconvolution kernel estimation.
We mainly focus on 
(i) the case where a sample drawn from the error distribution is available;
(ii) the case where validation data is available for data combination; and
(iii) the case where repeated measurements with errors one of which is symmetrically distributed are available.
For (ii) data combination with validation data, our model shares similarities albeit different assumptions to that of the nonparametric instrumental variables (NPIV) regression, for which \citet{HoLe12, ChCh15}, and \citet{Ba16} develop methods to construct confidence bands as we do for nonparametric EIV regression.

We note the following two references as particularly relevant benchmarks for identifying our contributions. 
One reference is \cite{Sc04} that derives pointwise asymptotic normality for the nonparametric EIV regression estimator different from ours, under unknown error distribution.
To this existing result, our contributions are four-fold.
First, we provide a method of uniform inference as opposed to a pointwise one.
Second, we propose a method of bandwidth selection for valid inference.
Third, while the existing result left aside the issue of variance estimation and thus are not readily applicable in practice, we provide a bootstrap method for ease of practical implementation.
Fourth, we devise lower-level assumptions which are easier to verify with concrete examples of distribution and conditional moment functions.
The other reference is \cite{DeHaJa15} that suggests a method of pointwise inference via bootstrap for nonparametric EIV regression with known error distribution.
To this existing result, our contributions are three-fold.
First, our method allows for unknown error distribution.
Second, we provide a method of uniform inference as opposed to a pointwise one.
Third, we provide formal theories to support the asymptotic validity of our bootstrap method. \citet{DeHaJa15} mention how to modify their methodology to the case where the measurement error distribution is unknown, and to construct uniform confidence bands. However, their theoretical results do not formally cover those cases. 

Finally, \citet{BiBiHo10} and \citet{PrBiDe15} obtain confidence bands for inverse regression with fixed equidistant designs (the fixed equidistant design assumption is substantial in their setups and analyses); the inverse regression is related to but different from our EIV regression, and our setup does not allow fixed equidistant designs because of measurement errors. The methodologies and the proof strategies are also different; e.g. both of those papers rely on Gumbel approximations for validity of the confidence bands, which we do not. 

Importantly, to the best of our knowledge, none of the existing results covers uniform confidence bands for EIV regression (\ref{eq: EIV}), even under the simpler setting that the measurement error distribution is known. The present paper fills this important void. 

\subsection{Notations and Organization}
For a non-empty set $T$ and a (complex-valued) function $f$ on $T$, we use the notation $\| f \|_{T} = \sup_{t \in T} |f(t)|$. Let $\ell^{\infty}(T)$ denote the Banach space of all bounded real-valued functions on $T$ with norm $\| \cdot \|_{T}$.
The convolution of $f_X$ and $f_U$ is denoted and defined by $(f_X \ast f_U)(y) = \int_{\mathbb{R}} f_X(x) f_U(y-x) dx$.
The Fourier transform of an integrable function $f$ on $\R$ is defined by $\varphi_{f}(t) = \int_{\R} e^{itx} f(x) dx, \ t \in \R$,
where $i=\sqrt{-1}$ denotes the imaginary unit throughout the paper. We refer to \cite{Fo99} as a basic reference on Fourier analysis. 
For any positive sequences $a_{n}$ and $b_{n}$, we write $a_{n} \sim b_{n}$ if $a_{n}/b_{n}$ is bounded and bounded away from zero. 
For any $a,b \in \R$, let $a \wedge b = \min \{ a,b \}$ and $a \vee b = \max \{ a,b \}$. 
For $a \in \R, b > 0$, we use the shorthand notation $[a \pm b] = [a-b,a+b]$. 
Let $\stackrel{d}{=}$ denote the equality in distribution. 

The rest of the paper is organized as follows. In Section \ref{sec: methodology}, we informally present our methodology to construct uniform confidence bands for $g$. In Section \ref{sec: main results}, we present asymptotic validity of the proposed confidence band under suitable regularity conditions. In Section \ref{sec: practical considerations}, we propose a practical method to choose the bandwidth. In Section \ref{sec: simulation}, we conduct simulation studies to verify the finite sample performance of the proposed confidence band. In Section \ref{sec: real data analysis}, we apply the proposed method to a combination of two empirical data sets. In Section \ref{sec: extensions}, we discuss extensions of our results  to specification testing of the conditional mean function,  cases with additional regressors without measurement errors, and construction of confidence bands for the conditional distribution function. 
Section \ref{sec: conclusion} concludes. 
Appendix contains all the proofs, additional simulation results, and additional details of real data analysis.


\section{Methodology}
\label{sec: methodology}

In this section, we informally present our methodology to construct confidence bands for $g$. The formal analysis of our confidence bands will be carried out in the next section. We will also discuss some examples of situations where an auxiliary sample from the measurement error distribution is available. 

\subsection{Deconvolution kernel estimation}
\label{sec: deconvolution kernel estimation}

We first introduce a deconvolution kernel method to estimate $f_{X}$ and $g$ under the assumption that the distribution of $\U$ is known. 
Let $\{ (Y_{1},W_{1}),\dots,(Y_{n},W_{n}) \}$ be an independent sample from the distribution of $(Y,W)$.
Throughout the paper, we assume that the densities of $X$ and $\U$ exist and are denoted by $f_{X}$ and $f_{\U}$, respectively. Let $\varphi_{W}, \varphi_{X}$, and $\varphi_{\U}$ denote the characteristic functions of $W,X$, and $\U$, respectively. 
Let $K: \R \to \R$ be a kernel function such that $K$ is integrable on $\R$, $\int_{\R} K(x) dx=1$, and its Fourier transform $\varphi_{K}$ is supported in $[-1,1]$ (i.e., $\varphi_{K}(t) = 0$ for all $|t| > 1$). 

When $f_{\U}$ is known, the deconvolution kernel density estimator of $f_{X}$ reads
\begin{equation}
\hat{f}_{X}^{*}(x)  =\frac{1}{2\pi} \int_{\R} e^{-itx} \hat{\varphi}_{W}(t) \frac{\varphi_{K}(th_{n})}{\varphi_{\U}(t)} dt = \frac{1}{nh_{n}} \sum_{j=1}^{n} K_{n}((x-W_{j})/h_{n}),
\label{eq: second expression}
\end{equation}
where the function $K_{n}$, called the \textit{deconvolution kernel}, is defined by 
\[
K_{n}(x) = \frac{1}{2\pi} \int_{\R} e^{-itx} \frac{\varphi_{K}(t)}{\varphi_{\U}(t/h_{n})} dt.
\]
See \cite{CaHa88,StCa90,Fa91a,Fa91b} for asymptotic properties.


Analogously, \citet{FaTr93} propose to estimate the regression function $g(x)$ by $\hat{g}^{*}(x) = \hat{\mu}^{*}(x)/\hat{f}_{X}^{*}(x)$, where
\begin{align*}
\hat{\mu}^{*}(x)&=\frac{1}{2\pi} \int_{\R} e^{-itx} \left (n^{-1} {\textstyle \sum}_{j=1}^{n} Y_{j} e^{itW_{j}} \right ) \frac{\varphi_{K}(th_{n})}{\varphi_{\U}(t)} dt \\
&= \frac{1}{nh_{n}} \sum_{j=1}^{n}Y_{j}K_{n}((x-W_{j})/h_{n}).
\end{align*}
For asymptotic properties, see \cite{FaTr93,FaMa92}, among others.\footnote{It is worth pointing out that estimation of $f_{X}$ and $gf_{X}$ corresponds to solving certain Fredholm integral equations of the first kind, and therefore estimation of $f_{X}$ and $gf_{X}$ (or $g$) is a statistical ill-posed inverse problem. 
In fact, $f_{X}$ and $gf_{X}$ satisfy $f_{X} * f_{\U} = f_{W}$ and $(gf_{X})*f_{\U} =\Ep[Y \mid W=\cdot \,]f_{W}$; these are Fredholm integral equations of the first kind where the right hand side functions are directly estimable (see, e.g., \cite{Ca08} and \cite{Ho09} for an overview of statistical ill-posed inverse problems).}

The discussion so far has presumed that the distribution of $\U$ is known. 
In the present paper, we consider the case of unknown distribution of $\U$, and assume that there is an independent sample $\{ \eta_{1},\dots,\eta_{m} \}$ from the distribution of $\U$, i.e.,
$
\eta_{1},\dots,\eta_{m} \sim f_{\U} \ \text{i.i.d.}
$
where $m=m_{n} \to \infty$ as $n \to \infty$. We do not assume that $\eta_{1},\dots,\eta_{m}$ are independent from $\{ (Y_{1},W_{1}),\dots,(Y_{n},W_{n}) \}$. In Section \ref{section examples}, we will discuss examples where such an auxiliary sample from the measurement error distribution is available. 
Given $\{ \eta_{1},\dots,\eta_{m} \}$, we may estimate $\varphi_{\U}$ by the empirical characteristic function, namely, 
\[
\hat{\varphi}_{\U}(t) = \frac{1}{m} \sum_{j=1}^{m} e^{it\eta_{j}},
\]
and estimate the deconvolution kernel $K_{n}$ by the plug-in method: 
\[
\hat{K}_{n}(x) = \frac{1}{2\pi} \int_{\R} e^{-itx} \frac{\varphi_{K}(t)}{\hat{\varphi}_{\U}(t/h_{n})} dt.
\]
We can thus estimate $g(x)$ by $\hat{g}(x) = \hat{\mu}(x)/\hat{f}_{X}(x)$, where 
\[
\hat{\mu}(x)  = \frac{1}{nh_{n}} \sum_{j=1}^{n}Y_{j}\hat{K}_{n}((x-W_{j})/h_{n}), \ \  \hat{f}_{X}(x) = \frac{1}{nh_{n}} \sum_{j=1}^{n}\hat{K}_{n}((x-W_{j})/h_{n}).
\]
Density estimators of the form $\hat{f}_{X}$ are studied in \cite{DiHa93,Ne97,Ef97}, among others, and nonparametric regression estimators of the form $\hat{g}$ are studied in \cite{DeHaMe08}, among others. 

\subsection{Motivation for our approach}

In this section, we motivate our approach based on the multiplier bootstrap and Gaussian approximation.
To this end, we first introduce the multiplier bootstrap in a simple setting, and present intuition for how and why it works.
Suppose that an estimator $\hat{g}(x)$ of $g(x)$ has the influence function representation
$
\sqrt{n}\left(\hat{g}(x)-g(x)\right) \approx n^{-1/2} \sum_{j=1}^n \phi_j(x),
$
and the right-hand side is approximately normal $N(0,\sigma^2(x))$.
In this simple setting, we can construct a confidence interval for $g(x)$ by
$
\left[\hat g(x) \pm \hat{\sigma}(x) \Phi^{-1}\left((1-\tau)/2\right) \right],
$
where $\Phi$ denotes the distribution function of $N(0,1)$, and $\hat{\sigma}(x)$ is an estimator of $\sigma(x)$.
The multiplier bootstrap can be used to obtain an estimator of the approximating Gaussian distribution $N(0,\sigma^2(x))$.
Specifically, generate multiplier variables $\xi_1, \dots, \xi_n \sim N(0,1)$ i.i.d. independently of the data, and consider
$
Z_n^\xi(x) =n^{-1/2}\sum_{j=1}^n \xi_j \phi_j(x).
$
One can generate multiplier variables many times to approximate the distribution of $Z_n^\xi(x)$ conditionally on the data, which estimates the $N(0,\sigma^{2}(x))$ distribution.
This method works intuitively because $n^{-1/2} \sum_{j=1}^n \xi_j \phi_j(x)$, as a sum of normal random variables, is normal conditionally on $\phi_1(x),\dots,\phi_n(x)$ with variance approximately $\sigma^2(x)$.

The above idea of the multiplier bootstrap extends to estimating the distribution of
$
\| \sqrt{n}\left(\hat{g}(\cdot)-g(\cdot) \right) \|_{I}
$
for a compact interval $I$ of $\mathbb{R}$, where $\|f\|_{I} = \sup_{x \in I} \left|f(x)\right|$.
Suppose that
$
\sqrt{n}\left(\hat{g}(x)-g(x)\right) \approx n^{-1/2} \sum_{j=1}^n \phi_j(x)
$
for all $x \in I$, and the process 
$
I \ni x \mapsto n^{-1/2} \sum_{j=1}^n \phi_j(x)
$
is approximately Gaussian.
One can generate $\xi_1,\dots,\xi_n \sim N(0,1)$ many times to approximate the distribution of $\| Z_n^\xi \|_{I}$ conditionally on the data, which can be used as an estimator for the distribution of 
$
\| \sqrt{n}\left(\hat{g}(\cdot)-g(\cdot)\right) \|_{I}.
$
Particularly, the $(1-\tau)$-th quantile of $\| Z_n^\xi(\cdot) \|_{I}$ conditionally on  the data, denoted by $\widehat c_n$, can be used to form the uniform confidence band of $g$ on $I$ as
$
\left[\hat g(x) \pm \widehat c_n \right], x \in I.
$\footnote{
In the current simplified setup for a concise motivation, the width of the band is constant, $\widehat c_n$, on $I$.
We later extend this framework to self-normalized processes, and thence we obtain non-constant widths of the band on $I$.
}
We shall recommend and thus present this multiplier bootstrap approach based on the Gaussian approximation, rather than more traditional methods, for the following two reasons.
First, the multiplier bootstrap can easily approximate $\widehat c_n$ by simulating standard normal multipliers, whereas one would otherwise need to compute the covariance function of 
$
I \ni x \mapsto n^{-1/2} \sum_{j=1}^n \phi_j(x),
$
which is an infinite-dimensional object, and compute the distribution of the corresponding Gaussian supremum, which is not easy. 
Second, the traditional approach to construct confidence bands is based on the Gumbel limit distribution for example \citep[cf.][]{BiDuHoMu07}.
Obtaining the Gumbel limit distribution requires additional regularity assumptions than our approach, as discussed in Section \ref{sec: main results} in more detail.
Furthermore, the convergence to this limit distribution occurs only at a slow $\log$ rate \cite{Ha91}.
For these reasons, we recommend our approach based on the multiplier bootstrap and Gaussian approximation, and present it in the following section.

\subsection{Construction of confidence bands}
\label{sec: construction of confidence bands}

We now describe our method to construct confidence bands for $g$ based on the estimator $\hat{g}$. 
Under the regularity conditions stated below, we will show that $\hat{g}(x)-g(x)$ can be approximated by 
$\frac{1}{f_{X}(x)nh_{n}} \sum_{j=1}^{n} [ \{ Y_{j}-g(x) \} K_{n}((x-W_{j})/h_{n}) - A_{n}(x) ]$
uniformly in $x \in I$, where $I$ is a compact interval in $\R$ on which $f_{X}$ is bounded away from zero, and $A_{n}(x)= \Ep[\{ Y-g(x) \} K_{n}((x-W)/h_{n})]$. Let 
\[
s_{n}^{2}(x) = \Var \left ( \{ Y-g(x) \} K_{n}((x-W)/h_{n}) \right ),
\]
and consider the process
\begin{equation}
\mathsf{Z}_{n}^{*}(x) = \frac{1}{s_{n}(x)\sqrt{n}} \sum_{j=1}^{n}[ \{ Y_{j}-g(x) \} K_{n}((x-W_{j})/h_{n})) -A_{n}(x) ], \ x \in I,
\label{eq: Zstar}
\end{equation}
where $s_{n}(x)=\sqrt{s_{n}^{2}(x)}$. Note that under the regularity conditions stated below, $\inf_{x \in I} s_{n}(x) > 0$ for sufficiently large $n$, so that $\mathsf{Z}_{n}^{*}$ is well-defined.
Furthermore, we will show that there exists a tight Gaussian process $\mathsf{Z}_{n}^{G}$ in $\ell^{\infty}(I)$ with mean zero and the same covariance function as $\mathsf{Z}_{n}^{*}$, and such that as $n \to \infty$, $\sup_{z \in \R} \left | \Pr \{ \| \mathsf{Z}_{n}^{*} \|_{I} \leq z \} - \Pr \{ \| \mathsf{Z}_{n}^{G} \|_{I} \leq z \} \right | \to 0$.
Recall that $\| f(x) \|_{I} = \sup_{x \in I} | f(x) |$ for a  function $f$ on $I$, and $\ell^{\infty}(I)$ is the Banach space of bounded real functions on $I$ equipped with norm $\| \cdot \|_{I}$. This  in turn yields that 
\[
\sup_{z \in \R} | \Pr  \{ \| \hat{\mathsf{Z}}_{n} \|_{I} \leq z \} - \Pr \{ \| \mathsf{Z}_{n}^{G} \|_{I} \leq z  \}  | \to 0,
\]
where $\{ \hat{\mathsf{Z}}_{n}(x) : x \in I \}$ is a process defined by\footnote{At first glance, this equation appears to suggest that the rate of convergence is $1/\sqrt{n}h_n$, but the this is not the case -- see discussions in Section \ref{sec: main results}. This equation entailing the scale of $\sqrt{n} h_n$ is due to our definition of $Z_n^\ast$ for convenience of our proofs.}
\begin{equation}
\hat{\mathsf{Z}}_{n} (x) =f_{X}(x) \sqrt{n}h_{n} (\hat{g}(x)-g(x))/s_{n}(x), \ x \in I.
\label{eq: deviation process}
\end{equation}
Therefore, if we denote by $c_{n}^{G}(1-\tau)$ the $(1-\tau)$-quantile of $\| \mathsf{Z}_{n}^{G} \|_{I}$
for $\tau \in (0,1)$, then a band of the form 
\[
\hat{\mathcal{C}}_{1-\tau}^{*}(x) = \left [ \hat{g}(x) \pm \frac{s_{n}(x)}{f_{X}(x)\sqrt{n} h_{n}} c_{n}^{G}(1-\tau) \right ], \ x \in I
\]
will contain $g(x), x \in I$ with probability at least $1-\tau+o(1)$ as $n \to \infty$. In fact, it holds that
$\Pr\{ g(x) \in \hat{\mathcal{C}}_{1-\tau}^{*}(x) \ \forall x \in I \} = \Pr \{ \| \hat{\mathsf{Z}}_{n} \|_{I} \leq c_{n}^{G}(1-\tau)  \} = \Pr \{ \| \mathsf{Z}_{n}^{G} \|_{I} \leq c_{n}^{G}(1-\tau)  \} + o(1) \geq 1-\tau+o(1)$.

In practice, $f_{X}(x),s_{n}^{2}(x)$, and $c_{n}^{G}(1-\tau)$ are all unknown, and we have to estimate them. We estimate $f_{X}(x)$ and $s_{n}^{2}(x)$ by $\hat{f}_{X}(x)$ and
\[
\hat{s}_{n}^{2}(x) = \frac{1}{n} \sum_{j=1}^{n} \{ Y_{j} - \hat{g}(x) \}^{2}\hat{K}_{n}^{2}((x-W_{j})/h_{n}),
\]
respectively. Note that $(\Ep[ A_{n}(x) ])^{2}$ is negligible relative to $s_{n}^{2}(x)$ so that we have ignored $(\Ep[ A_{n}(x) ])^{2}$ in estimation of $s_{n}^{2}(x)$. Note also that $\sum_{j=1}^{n} (Y_{j}-\hat{g}(x)) \hat{K}_{n}((x-W_{j})/h_{n})) = 0$.

Next, we estimate  the quantile $c_{n}^{G}(1-\tau)$ by the Gaussian multiplier bootstrap. Generate $\xi_{1},\dots,\xi_{n} \sim N(0,1) \ \text{i.i.d.}$,
independently of the data $\mathcal{D}_{n} = \{ Y_{1},\dots,Y_{n},W_{1},\dots,W_{n},\eta_{1},\dots,\eta_{m} \}$, and consider 
\begin{equation}
\hat{\mathsf{Z}}_{n}^{\xi} (x) = \frac{1}{\hat{s}_{n}(x)\sqrt{n}} \sum_{j=1}^{n} \xi_{j}  \{ Y_{j}-\hat{g}(x) \} \hat{K}_{n}((x-W_{j})/h_{n}),
\label{eq: multiplier process}
\end{equation}
where $\hat{s}_{n}(x) = \sqrt{\hat{s}_{n}^{2}(x)}$. Note that under the regularity conditions stated below, $\inf_{x \in I} \hat{s}_{n}(x) > 0$ with probability approaching one. 
Conditionally on the data $\mathcal{D}_{n}$, $\hat{\mathsf{Z}}_{n}^{\xi}$ is a Gaussian process with mean zero and covariance function (presumably) ``close'' to or ``estimates'' that of $\mathsf{Z}_{n}^{*}$. 
 Hence, we estimate $c_{n}^{G}(1-\tau)$ by $\hat{c}_{n}(1-\tau)$, defined as the conditional $(1-\tau)$-quantile of $\| \hat{\mathsf{Z}}_{n}^{\xi} \|_{I}$ given $\mathcal{D}_{n}$,
which can be computed via simulations.
Now, the resulting confidence band is defined by 
\begin{equation}
\hat{\mathcal{C}}_{1-\tau}(x) = \left [ \hat{g}(x) \pm \frac{\hat{s}_{n}(x)}{\hat{f}_{X}(x)\sqrt{n} h_{n}} \hat{c}_{n}(1-\tau) \right ], \ x \in I.
\label{eq: proposed band}
\end{equation}
Note that, except for the choice of the bandwidth, this confidence band is completely data-driven. We will discuss practical choice of the bandwidth in Section \ref{sec: practical considerations}. 

\begin{remark}
\label{rem: multiplier process}
In the error-free case where we observe $\{ (Y_{j},X_{j}) \}_{j=1}^{n}$, the deviation of a traditional kernel regression estimator with kernel $K$ from its mean is approximated by $\frac{1}{f_{X}(x) nh_{n} }\sum_{j=1}^{n} \epsilon_{j} K((x-X_{j})/h_{n})$ under suitable regularity conditions. 
On the other hand, in the EIV case, the deviation of $\hat{g}(x)$ from its mean is approximated by 
$\frac{1}{f_{X}(x)nh_{n}} \sum_{j=1}^{n} [ \{ Y_{j}-g(x) \} K_{n}((x-W_{j})/h_{n})]$.
Note that $\epsilon_{j} = Y_{j} - g(X_{j})$ is used in the former error-free case, whereas $Y_{j} - g(x)$ is used in the latter EIV case.
Since the conditional variance of $\epsilon_{j} = (Y_{j} - g(x)) - (g(x) - g(X_j))$ given $W_j=x$ is not necessarily the same as the the conditional variance of $Y_{j} - g(x)$ given $W_j=x$, it is essential to use $Y_{j}-g(x)$ in the EIV case rather than $\epsilon_{j}$ (even if $\epsilon_{j}$ or its estimate were available).
\end{remark}

\begin{remark}[Empirical bootstrap]
Instead of the multiplier bootstrap, one can use (a version of) the empirical bootstrap to estimate the quantile $c_{n}^{G}(1-\tau)$, as in \cite{AdOtWh16} (note: \cite{AdOtWh16} consider inference on the marginal distribution function of $X$ and not consider infererence on $g(x)$). Namely, let $(Y_{1}^{*},W_{1}^{*}),\dots,(Y_{n}^{*},W_{n}^{*})$ be i.i.d. draws from the empirical distribution of $\{ (Y_{1},W_{1}),\dots,(Y_{n},W_{n}) \}$, and consider the bootstrap process
\[
\hat{\mathsf{Z}}_{n}^{EB}(x) = \frac{1}{\hat{s}_{n}(x)\sqrt{n}} \sum_{i=1}^{n} \{ Y_{j}^{*} - \hat{g}(x) \} \hat{K}_{n}((x-W_{j}^{*})/h_{n}), \ x \in I. 
\]
We note that the expectation of $\{ Y_{j}^{*} - \hat{g}(x) \} \hat{K}_{n}((x-W_{j}^{*})/h_{n})$ with respect to the bootstrap distribution is $n^{-1} \sum_{j'=1}^{n} \{ Y_{j'} -  \hat{g}(x) \} \hat{K}_{n}((x-W_{j'})/h_{n}) = 0$ (by the definition of $\hat{g}(x)$), and so conditionally the right hand side is the sum of independent random variables with mean zero. 
Then we can estimate $c_{n}^{G}(1-\tau)$ by the conditional $(1-\tau)$-quantile of $\| \hat{\mathsf{Z}}_{n}^{EB} \|_{I}$ given $\mathcal{D}_{n}$, and so the resulting confidence band is 
\begin{equation}
\hat{\mathcal{C}}^{EB}_{1-\tau}(x) = \left [ \hat{g}(x) \pm \frac{\hat{s}_{n}(x)}{\hat{f}_{X}(x)\sqrt{n} h_{n}} \hat{c}_{n}^{EB}(1-\tau) \right ], \ x \in I,
\label{eq: EB band}
\end{equation}
which we call the empirical bootstrap confidence band. This empirical bootstrap confidence band is also asymptotically valid under the same conditions as the multiplier bootstrap confidence band. See Appendix \ref{sec: validity of MB} for details. 
\end{remark}

\subsection{Examples}\label{section examples}

In this section, we present a couple of examples where an auxiliary sample from the measurement error distribution is available.

\begin{example}[Repeated measurements, \cite{CaRuStCr06}, p.298]
\label{ex: panel}
Suppose that we observe repeated measurements $(W^{(1)},W^{(2)})$ on $X$ with measurement errors $(U^{(1)},U^{(2)})$, where $W^{(k)} = X + \U^{(k)}$ for $k=1,2$, $X$ and $(\U^{(1)},\U^{(2)})$ are independent, and the conditional distribution of $\U^{(2)}$ given $\U^{(1)}$ is symmetric. Note that the unobserved regressor $X$ is common between $k=1,2$. The distribution of $\U^{(1)}$ need not be symmetric (in particular, the distributions of $\U^{(1)}$ and $\U^{(2)}$ may be different), and independence between $\U^{(1)}$ and $\U^{(2)}$ is  not necessary. If we define $W=(W^{(1)}+W^{(2)})/2, \U = (\U^{(1)}+\U^{(2)})/2$, and $\eta = (W^{(1)}-W^{(2)})/2 = (\U^{(1)}-\U^{(2)})/2$, then we have that $W = X+\U$ and $\U \stackrel{d}{=} \eta$,
where $\eta$ is observable. 

Note that the constructed sample $(\eta_1,\dots,\eta_n) = \left((W^{(1)}_1-W^{(2)}_1)/2,\dots,(W^{(1)}_n-W^{(2)}_n)/2\right)$ behaves as an auxiliary sample of size $m=n$ due to $U \stackrel{d}{=} \eta$.
In this way, we transform the repeated measurement framework into the auxiliary sample framework, and therefore our proposed method outlined in Sections \ref{sec: deconvolution kernel estimation}--\ref{sec: construction of confidence bands} for the case of an auxiliary sample directly applies to the case of repeated measurements.
This approach differs from the way in which some of preceding papers use repeated measurements \citep[e.g.,][]{DeHaMe08}.
One implication of this difference is that we can allow for a different set of assumption from that of \citet{DeHaMe08}.
First, they require conditions on the relative smoothness of $g$ with respect to the measurement error density, while we do not.
Second, we only require the conditional distribution of $\U^{(2)}$ given $\U^{(1)}$ to be symmetric, while they require both $\U^{(1)}$ and $\U^{(2)}$ to be symmetrically distributed.

For a repeated measurements setup, \cite{Sc04} proposes an alternative estimator of $g$ based on Kotlarski's lemma which does not require the symmetry assumption. The form of Schennach's estimator is more complex than ours, and to the best of our knowledge, there is no existing result on asymptotically valid uniform confidence bands for Schennach's estimator. 

It is worth noting that while Schennach's approach can drop the symmetry assumption, it requires another technical assumption that the characteristic function $\varphi_{X}(t) = \Ep[e^{itX}]$ of $X$ does not vanish on the entire real line $\R$. Both \cite{Sc04} and we (and in fact most of papers on deconvolution and EIV regression) assume that the characteristic functions of the error variables do not vanish on $\R$, but our approach does allow $\varphi_{X}$ to take zeros. The assumption that $\varphi_{X}$ does not vanish on $\R$ is not innocuous; it is non-trivial to find densities that are compactly supported and have non-vanishing characteristic functions (though these properties are not mutually exclusive; see, e.g., \cite{Sc16}, Footnote 4), and the assumption excludes densities convolved with distributions whose characteristic functions take zeros, and so on (e.g., convolutions of $k$ uniform densities on $[a,b]$ are piecewise polynomials with degrees $k-1$, and convex combinations of such piecewise polynomials form a rich family of densities, but their characteristic functions take zeros). 
So, we believe that Schennach's approach and ours are complementary to each other. 
\end{example}

\begin{example}[Data combination]
Suppose that we have access to data on $(Y,W)$ and $(W,X)$, separately, but do not have access to data on $(Y,X)$.
In this case, we can construct a sample $\{ Y_1,\dots,Y_n,W_1,\dots,W_n,\eta_1,\dots,\eta_m \}$ from a sample of size $n$ on $(Y,W)$ and a sample of size $m$ on $(W,X)$.

This case is often faced by empirical researchers, and various techniques are proposed to combine the two separate samples -- see a survey by \cite{RiMo07}.
To fix ideas, consider the demand model $Y= g(X) + \epsilon$, where $Y$ denotes the quantity purchased of a product and $X$ denotes the logarithm of its price.
Marketing scientists and economists often use Nielsen Homescan data for quantities and prices to analyze this demand model, but the home-scanned prices in this data are subject to imputation errors $\U = W - X$.
To overcome this issue, \cite{EiLeNe10} collect data on $(W,X)$ from a large grocery retailer by matching transaction prices $X$ that were recorded by the retailer (at the store) to the prices $W$ recorded by the Homescan panelists.
Together with Nielsen Homescan data on $(Y,W)$, \citeauthor{EiLeNe10} suggest to combine the two separate data sets to analyze the demand model.

In the literature, validation data are used as a way to relax the classical measurement error assumption that $X$ and $\U$ are independent; see, e.g., \cite{ChHoTa05}. While they allow for non-classical measurement errors, \cite{ChHoTa05} focus on the case where the parameter of interest is finite dimensional. 

It is worth noting that, when validation data on $(X,W)$ are available, the problem of estimation of $g$ can be considered as a nonparametric instrumental variable (NPIV) problem treating $X$ as an endogenous variable and $W$ as an instrumental variable \citep[see,e.g.,][for NPIV models]{NePo03,HaHo05,BlChKr07,ChRe11,Ho11}. In fact, observe that $\Ep[ Y \mid W ] = \Ep[ g(X) \mid W]$. For NPIV models, \cite{HoLe12} and the more recent paper by \cite{ChCh15} develop methods to construct confidence bands for the structural function using series methods, although these papers do not formally consider cases where samples on $(Y,W)$ and $(X,W)$ are different (\cite{Ba16} also develop methods to construct confidence bands for Tikhonov regularized estimators in NPIV models, but his confidence bands are asymptotically conservative in the sense that the coverage probabilities are in general strictly larger than the nominal level even asymptotically). 
 However, we would like to point out that there are difference in underlying assumptions between series estimation of NPIV models and deconvolution kernel estimation in EIV regression. For example, in series estimation of NPIV models, it is often assumed that the distribution of $W$ is compactly supported and the density of $W$ is bounded away from zero on its support \cite[cf.][]{BlChKr07,ChCh15}. 
On the other hand, in EIV regression, it is commonly assumed that the characteristic function of the measurement error $\U$ is non-vanishing on $\R$ (which leads to identification of the function $g$), and in many cases the measurement error $\U$ then has unbounded support, which in turn implies that $W$ has unbounded support. 
Further, while both NPIV and EIV regressions are statistical ill-posed inverse problems, the ways in which the ``ill-posedness'' is defined are different; in series estimation of NPIV models, the ill-posedness is defined for given basis functions, while in EIV regression, the ill-posedness is defined via how fast the characteristic function of the measurement error distribution decays. 
Finally, the EIV approach captures all the restrictions implied by the classical measurement error assumption and therefore enjoys more efficiency, while the NPIV approach requires less restrictive assumptions on the measurement error. 
For these differences and tradeoffs, we believe that our inference results cover different situations from those developed in the NPIV literature. 
\end{example}

\section{Main results}
\label{sec: main results}

In this section, we study asymptotic validity of the proposed confidence band (\ref{eq: proposed band}).
To this end, we make the following assumption. 
For any given constants $\beta, B > 0$, let $\Sigma (\beta,B)$ denote the class of functions $f: \R \to \R$ such that $f$ is $k$-times differentiable and
\[
|f^{(k)}(x)-f^{(k)}(y)| \leq B|x-y|^{\beta-k}, \ \forall x,y \in \R,
\]
where $k$ is the integer such that $k < \beta \leq k+1$, and $f^{(k)}$ denotes the $k$-th derivative of $f$ ($f^{(0)}=f$). Let $I$ be a compact interval in $\R$.

\begin{assumption}
\label{as: mean}
We assume the following conditions. 
\begin{enumerate}
\item[(i)] $\Ep[Y^{4}] < \infty$, the function $w \mapsto \Ep[Y^{2} \mid W=w]f_{W}(w)$ is bounded and continuous, and for each $\ell=1,2$, the function $w \mapsto \Ep[ |Y|^{2+\ell} \mid W=w] f_{W}(w)$ is bounded.  
\item[(ii)] The functions $\varphi_{X}(t) = \Ep[ e^{itX}]$ and $\psi_{X}(t) = \Ep[ g(X) e^{itX} ]$ for $t \in \R$ are integrable on $\R$.
\item[(iii)] The measurement error $\U$ has finite mean, $\Ep[| \U |]<\infty$, and its characteristic function,  $\varphi_{\U}(t) = \Ep[e^{it\U}], t \in \R$, does not vanish on $\R$. Furthermore, 
there exist constants $C_{1} > 1$ and $\alpha > 0$ such that $C_{1}^{-1} |t|^{-\alpha} \leq | \varphi_{\U} (t) | \leq C_{1} |t|^{-\alpha}$ and $| \varphi_{\U}'(t) | \leq C_{1}|t|^{-\alpha-1}$ for all $|t| \geq 1$. 
\item[(iv)] The functions $f_{X}$ and $gf_{X}$  belong to $\Sigma (\beta,B)$ for some $\beta > 1/2$ and $B > 0$. Let $k$ denote the integer such that $k <\beta\leq k+1$. 
\item[(v)] Let $K$ be a real-valued integrable function (kernel) on $\R$, not necessarily non-negative, such that  $\int_{\R} K(x) dx = 1$, and its Fourier transform $\varphi_{K}$ is continuously differentiable and supported in $[-1,1]$. Furthermore, the function $K$ is a $(k+1)$-th order kernel in the sense that $\int_{\R} |x|^{k+1} |K(x)| dx < \infty$ and $\int_{\R} x^{\ell} K(x) dx = 0$ for $\ell=1,\dots,k$
\item[(vi)] For all $x \in I$, $f_{X}(x) > 0$ and $\Ep[ \{ Y-g(x) \}^{2} \mid W=x] f_{W}(x) > 0$.
\item[(vii)] As $n \to \infty$, 
\begin{equation}
\frac{(\log (1/h_{n}))^{2}}{(n \wedge m)h_{n}^{2\alpha+2}}   \to 0, \ \frac{nh_{n} \log (1/h_{n})}{m} \to 0,  \ h_{n}^{\alpha+\beta} \sqrt{nh_{n} \log (1/h_{n})} \to 0. 
\label{eq: bandwidth}
\end{equation}
\end{enumerate}
\end{assumption}

Condition (i) is a moment condition on  $Y$, which we believe is not restrictive. Note that, for each $\ell=0,1,2$, if $\Ep[ |Y|^{2+\ell} \mid X,\U] = \Ep[ |Y|^{2+\ell} \mid X]$, then 
by comparing the Fourier transforms of both sides, we arrive at the identity
\[
\Ep[|Y|^{2+\ell} \mid W=w]f_{W}(w) =\left ( (\Upsilon_{\ell} f_{X})*f_{\U} \right )(w),
\]
where $\Upsilon_{\ell} (x) = \Ep[ |Y|^{2+\ell} \mid X=x]$ (cf. the proof of Lemma \ref{lem: moment bound} Case (ii)), and the right hand side is bounded and continuous if $\Upsilon_{\ell} f_{X}$ is bounded (recall that if functions $f_{1}$ and $f_{2}$ on $\R$ are integrable and bounded, respectively, then their convolution $f_{1} * f_{2}$ is bounded and continuous). Note that this does not require $\Upsilon_{\ell}$ to be bounded globally. For Condition (ii), we first note that  $\psi_{X}$ is the Fourier transform of $gf_{X}$ (which is integrable by $\Ep[|Y|]<\infty$). Condition (ii) implies that 
\[
\text{$f_{X}$ and $gf_{X}$ are bounded (and continuous)},
\]
which in turn implies that 
\[
\text{$f_{W} (w)= \int_{\R} f_{X}(w-x) f_{\U}(x)dx$ is bounded and continuous}.
\]
Condition (ii) is satisfied if, e.g., $f_{X}$ and $gf_{X}$ are twice continuously differentiable with integrable derivatives up to the second order; in fact, under such conditions, $|\varphi_{X}(t)| = o(|t|^{-2})$ and $|\psi_{X}(t)| =o(|t|^{-2})$ as $|t| \to \infty$. However, differentiability of $f_{X}$ and $gf_{X}$ is not strictly necessary for Condition (ii) to hold; e.g., a Laplace density is not differentiable but its Fourier transform is  integrable.

Condition (iii) is concerned with the characteristic function of the measurement error. Note that finiteness of the first moment of $\U$ ensures that $\varphi_{\U}$ is continuously differentiable.  In the present paper, as in \cite{BiDuHoMu07}, \cite{ScMuDu13}, and \cite{DeHaJa15}, we assume that the measurement error density is \textit{ordinary smooth}, namely, $|\varphi_{\U}(t)|$ decays at most polynomially fast as $|t| \to \infty$ \citep[cf.][]{Fa91a}.
Informally, the smoother $f_{\U}$ is, the faster $|\varphi_{\U} (t)|$ decays as $|t| \to \infty$, so Condition (iii) restricts smoothness of $f_{\U}$. 
Laplace and Gamma distributions, together with their convolutions, (suitable) mixtures, and symmetrizations,
are typical examples of distributions satisfying Condition (iii), but normal and Cauchy distributions do not satisfy Condition (iii). Normal and Cauchy densities are examples of \textit{supersmooth} densities, i.e., their characteristic functions decay exponentially fast as $|t| \to \infty$ (but mixtures of ordinary- and supersmooth densities are ordinary smooth).
Furthermore, Theorem 1 in \cite{HuRi10} provides sufficient conditions in terms of $f_{U}$ under which the characteristic function of $U$ decays at most polynomially fast as $|t| \to \infty$. 
Specifically, if $f_{U}$ is range restricted of order $k$ in the sense that $f_{U}$ is supported on $[\underline{u},\overline{u}]$ with either $\underline{u}$ or $\overline{u}$ finite, and has $k+2$ for $k \ge 0$ absolutely integrable derivatives $f^{(j)}_{U}$ with $|f^{(k)}_{U}(\underline{u})| \ne |f^{(k)}_{U}(\overline{u})|$ and $f^{(j-1)}_{U}(\underline{u}) = f^{(j-1)}_{U}(\overline{u})$ for $j=1,\dots,k$, then $|\varphi_{U}(t)|$ decays like $|t|^{-(k+1)}$ in the tail \citep[][Theorem 1]{HuRi10}. 

Condition (iv) is concerned with smoothness of the functions $f_{X}$ and $g$.
Condition (v) is about a kernel function. 
Condition (vi) ensures that $\inf_{x \in I} f_{X}(x) > 0$ (since $f_{X}$ is continuous) and $\inf_{x \in I}\Ep[ \{ Y-g(x) \}^{2} \mid W=x] f_{W}(x) > 0$ (see the proof of Lemma \ref{lem: moment bound}-(ii)). 
Note, that since $gf_{X}$ is bounded, we have that 
\[
\| g \|_{I} \leq \| gf_{X} \|_{I}/\inf_{x \in I} f_{X}(x) < \infty.
\]

It is worth mentioning that under these conditions, we have that 
\[
s_{n}^{2}(x) = \Var ( \{ Y - g(x) \} K_{n}((x-W)/h_{n}) ) \sim h_{n}^{-2\alpha+1}
\]
uniformly in $x \in I$ (see Lemma \ref{lem: moment bound}), and the right hand side is larger by factor $h_{n}^{-2\alpha}$ than the corresponding term  in the error-free case (recall that, in standard kernel regression without measurement errors, the variance of $\epsilon K((x-X)/h_{n})$ is $\sim h_{n}$). 
This results in slower rates of convergence of kernel regression estimators in presence of measurement errors than those in the error-free case, and the value of $\alpha$ is a key parameter that controls the difficulty of estimating $g$. Namely, the larger the value of $\alpha$ is, the more difficult estimation of $g$ will be. In other words, the value of $\alpha$ quantifies the degree of ill-posedness of estimation of $g$. 

Condition (vii) restricts the bandwidth $h_{n}$ and the sample size $m$ from the measurement error distribution. 
The second condition in (\ref{eq: bandwidth}) restricts the growth rates of $m$, and is crucial for ensuring the negligibility of the effects of estimating the error characteristic function on the Gaussian approximation, which in turn establishes the bootstrap validity.
This condition allows $m$ to be of smaller order than $n$, which in particular covers the repeated measurement setup  discussed in Example \ref{ex: panel}. 
The last condition in (\ref{eq: bandwidth}) means that we are choosing undersmoothing bandwidths, that is, choosing bandwidths that are of smaller order than optimal rates for estimation of $g$ under the $\sup$-norm loss. 
Inspection of the proof of Theorem \ref{thm: Gaussian approximation} shows that without the last condition in (\ref{eq: bandwidth}), we have that
\[
\| \hat{g} - g \|_{I}  = O_{\Pr} \{ h_{n}^{-\alpha}(nh_{n})^{-1/2}\sqrt{\log (1/h_{n})}\} + O(h_{n}^{\beta}),
\]
where the $O(h_{n}^{\beta})$ term comes from the deterministic bias. So, choosing $h_{n} \sim (n/\log n)^{-1/(2\alpha+2\beta+1)}$ optimizes the rate on the right hand side, and the resulting rate of convergence of $\| \hat{g} - g \|_{I}$ is $O_{\Pr} \{ (n/\log n)^{-\beta/(2\alpha+2\beta+1)} \}$.
The last condition in (\ref{eq: bandwidth}) requires to choose $h_{n}$ of smaller order than $(n/\log n)^{-1/(2\alpha+2\beta+1)}$ (by $\log n$ factors), so that the ``variance'' term dominates the bias term. See Remark \ref{rem: bias} for further discussions on the bias. For Condition (vii) to be non-void, we require $\beta > 1/2$.


We first state a theorem that establishes that, under Assumption \ref{as: mean},  the distribution of $\| \hat{\mathsf{Z}}_{n} \|_{I} = \sup_{x \in I} | \hat{\mathsf{Z}}_{n} (x)|$, where $\{ \hat{\mathsf{Z}}_{n}(x) : x \in I \}$ is defined in (\ref{eq: deviation process}),  can be approximated by that of the supremum of a certain Gaussian process, which is a building block for proving validity of the proposed confidence band. Recall that a Gaussian process $\{ Z(x) : x \in I \}$ indexed by $I$ is a tight random variable in $\ell^{\infty}(I)$ if and only if $I$ is totally bounded
for the intrinsic $L^{2}$ pseudo-metric $\rho_{2} (x,y) = \sqrt{\Ep[ \{ Z(x)-Z(y) \}^{2}]}$ for $x,y \in I$, and $Z$ has sample
paths almost surely uniformly $\rho_{2}$-continuous; see \citet[][p.41]{vaWe96}. 
Recall the process $\mathsf{Z}_{n}^{*}$ defined in (\ref{eq: Zstar}). 

\begin{theorem}[Gaussian approximation]
\label{thm: Gaussian approximation}
Under Assumption \ref{as: mean}, for each sufficiently large $n$, there exists a tight Gaussian process  $\mathsf{Z}_{n}^{G}$ in $\ell^{\infty}(I)$ with mean zero and the same covariance function as $\mathsf{Z}_{n}^{*}$, and such that as $n \to \infty$, 
\begin{equation}
\sup_{z \in \R} \left | \Pr \left \{ \| \hat{\mathsf{Z}}_{n} \|_{I} \leq z \right \} - \Pr \left \{ \| \mathsf{Z}_{n}^{G} \|_{I} \leq z \right \} \right | \to 0.
\label{eq: Gaussian approximation}
\end{equation}
\end{theorem}

Theorem \ref{thm: Gaussian approximation} derives an ``intermediate'' Gaussian approximation to the  process $\hat{\mathsf{Z}}_{n}$, in the sense that the distribution of the approximating Gaussian process $\mathsf{Z}_{n}^{G}$ depends on the sample size $n$. It could be possible to further show that, if $I$ is not a singleton, under additional conditions, for some sequences $a_{n} > 0$ and $b_{n} \in \R$, $a_{n}(\| \hat{\mathsf{Z}}_{n}^{G} \|_{I}-b_{n})$ converges in distribution to a Gumbel distribution. However, while it is mathematically intriguing, we avoid using the Gumbel approximation, since 1) the Gumbel approximation is slow and the coverage error of the resulting confidence band is of order $1/\log n$ \citep[see][]{Ha91}, and 2) deriving the Gumbel approximation would require additional restrictive conditions on the measurement error distribution. For example, in a problem constructing confidence bands in deconvolution with known error distribution, \citet{BiDuHoMu07} derive a Gumbel approximation to the supremum deviation of the deconvolution kernel density estimator, thereby establishing a Smirnov-Bickel-Rosenblatt type theorem \citep{Sm50,BiRo73} for the deconvolution kernel density estimator. But to do so, they require more restrictive conditions on the measurement error distribution than those in the present paper (see their Assumption 2). 

The following theorem shows asymptotic validity of the proposed  confidence band. 
\begin{theorem}[Validity of multiplier bootstrap confidence band]
\label{thm: validity of MB}
Under Assumption \ref{as: mean}, as $n \to \infty$, 
\begin{equation}
\sup_{z \in \R}  \left | \Pr \left \{ \| \hat{\mathsf{Z}}^{\xi}_{n} \|_{I} \leq z \mid \mathcal{D}_{n}  \right \} - \Pr \left  \{ \| \mathsf{Z}_{n}^{G} \|_{I} \leq z \right \} \right | \stackrel{\Pr}{\to} 0, \label{eq: validity of MB}
\end{equation}
where $\mathsf{Z}_{n}^{G}$ is the Gaussian process in $\ell^{\infty}(I)$ given in Theorem \ref{thm: Gaussian approximation}. Therefore, for the confidence band $\hat{\mathcal{C}}_{1-\tau}$ defined in (\ref{eq: proposed band}),  we have as $n \to \infty$, 
\begin{equation}
\Pr \left \{ g(x) \in \hat{\mathcal{C}}_{1-\tau}(x) \ \forall x \in I \right \} = 1-\tau+o(1).
\label{eq: validity of MB 2}
\end{equation}
Finally, the supremum width of the band $\hat{\mathcal{C}}_{1-\tau}$ is 
\[
O_{\Pr}\{ h_{n}^{-\alpha} (nh_{n})^{-1/2} \sqrt{\log (1/h_{n})} \}.
\]
\end{theorem}

\begin{remark}
Inspection of the proof shows that the result (\ref{eq: validity of MB 2}) holds even when $\tau = \tau_{n} \downarrow 0$ as $n \to \infty$. Furthermore, the supremum width of the band is 
$O_{\Pr}\{ h_{n}^{-\alpha} (nh_{n})^{-1/2} \sqrt{\log (1/h_{n}) \vee \log (1/\tau_{n})} \}$.
\end{remark}

\begin{remark}
For $v_{n} \sim (\log n)^{-1}$, take $h_{n} = v_{n} (n/\log n)^{-1/(2\alpha+2\beta+1)}$; then the supremum width of the band $\hat{\mC}_{1-\tau}$ is $(n/\log n)^{-\beta/(2\alpha+2\beta+1)}(\log n)^{\alpha+1/2}$. 
\end{remark}

\begin{remark}[Bias]
\label{rem: bias}
For any nonparametric inference problem, how to deal with the deterministic bias is a delicate and difficult problem. See Section 5.7 in \cite{Wa06} for related discussions. In the present paper, we employ undersmoothing bandwidths so that the bias is negligible relative to the ``variance'' part.
An alternative approach is to estimate the bias at each point, and construct a bias corrected confidence band. See, for example, \cite{EuSp93} and \cite{Xi98} for the error-free case (more recent discussions regarding the problem of bias in nonparametric inference problems include \cite{HaHo13}, \cite{ChChKa14b}, \cite{ArKo14}, \cite{CaCaFa15}, and \cite{Sc15}; these papers do not cover EIV regression). However, in EIV regression, estimation of the bias is not quite attractive for a couple of reasons. First, the bias consists of higher order derivatives of $g$ and $f_{X}$, and estimation of these higher order derivatives is difficult, especially in the EIV case. This is because estimation of $g$ and $f_{X}$ is an ill-posed inverse problem and rates of convergence of the derivative estimators of $g$ and $f_{X}$ are even slower than those in the error-free case. As such, using a bias-corrected estimator would suffer from slower convergence rates and more complicated inference results. Second, one of popular kernels used in EIV regression and deconvolution is a flat top kernel \citep{McPo04} which is an infinite order kernel, and if we use a flat top kernel, then the bias is  not calculated in a closed form (e.g., \citet{Sc04,BiDuHoMu07} use flat top kernels in their simulation studies). See Remark 1 in \cite{BiDuHoMu07} for a related issue in the deconvolution case. 
\end{remark}

\begin{remark}[supersmooth case]
\label{rem: super smooth}
In the present paper, we focus on the case where the measurement error density is ordinary smooth, similarly to \cite{BiDuHoMu07}, \cite{ScMuDu13}, and \cite{DeHaJa15} that study inference in deconvolution and nonparametric EIV regression. If the measurement error density is supersmooth, i.e., its characteristic function decays exponentially fast as $|t| \to \infty$, then 1) in view of the pointwise asymptotic normality result in \cite{FaMa92}, the asymptotic behavior of  the variance function $s_{n}^{2}(x)$ is much more complex. Indeed, as in \cite{KaSa16}, assume that the distribution of $U$ is supersmooth in the sense that 
$\varphi_{U}(t) \sim |t|^{\gamma_{0}}e^{-\nu |t|^{\gamma}}$ as $|t| \to \infty$ for some $\gamma > 1, \gamma_{0} \in \R, \nu > 0$, and the Fourier transform of the kernel function $\varphi_{K}$ is even, supported in  $[-1,1]$, and  $\varphi_{K}(1-t) \sim t^{\lambda}$ as $t \downarrow 0$ for some $\lambda \ge 0$; then it is expected that 
\[
s_{n}^{2}(x) \sim h_{n}^{2\gamma (1+\lambda) + 2\gamma_{0}} e^{2\nu h_{n}^{-\gamma}}.
\]
See also \cite{EsUh05}. Importantly, the ratio of $1/\inf_{|t| \leq h_{n}^{-1}} | \varphi_{U}(t)|$ over $\inf_{x \in I} \sigma_{n}(x)$ would be then larger in the supersmooth case than that in the ordinary smooth case; the ratio is $O(h_{n}^{-\gamma (1+\lambda)})$ in the supersmooth case, while it is $O(h_{n}^{-1/2})$ in the ordinary smooth case. This would require $m$ to be of larger order than $n$ (i.e., $m/n \to \infty$) to formally show validity of the confidence band in the supersmooth case. Analogous discussion can be found in \cite{KaSa16} in the density deconvolution case. 
2) Minimax rates of convergence for estimation of $g$ under the sup-norm loss are logarithmically slow (i.e., of the form $(\log n)^{-c}$ for some constant $c > 0$), even when the measurement error distribution is assumed to be known \citep{FaTr93}. These difficulties prevent us from directly extending our analysis to the supersmooth case. Hence the supersmooth case is left for future research. 
\end{remark}

The proofs of Theorems \ref{thm: Gaussian approximation} and \ref{thm: validity of MB} build on non-trivial applications of the intermediate Gaussian and multiplier bootstrap approximation theorems developed in \cite{ChChKa14a, ChChKa14b, ChChKa16}. However, we stress that Theorems \ref{thm: Gaussian approximation} and \ref{thm: validity of MB} do not follow directly from the general theorems in \cite{ChChKa14a, ChChKa14b,ChChKa16} and require substantial work. This is because 1) first of all, how to devise a multiplier bootstrap in EIV regression is not apparent, and as discussed in Remark \ref{rem: multiplier process} our construction of the multiplier process appears to be novel; 2) the ``population'' deconvolution kernel $K_{n}$ is implicitly defined via the Fourier inversion and substantially different from standard kernels in the error-free case; and 3) the deconvolution kernel $K_{n}$ is in fact unknown and estimated, so that its estimation error has to be taken into account. 

An alternative standard technique to derive Gaussian approximations similar to (\ref{eq: Gaussian approximation}) is to apply the Koml\'{o}s-Major-Tusn\'{a}dy (KMT) strong approximation \citep{KoMaTu75}. In a problem of constructing confidence bands in deconvolution with known error distribution, \citet{BiDuHoMu07} (and \citet{ScMuDu13}) use the KMT approximation to derive Gaussian approximations to the deconvolution kernel density estimator. However, the KMT approximation is tailored to empirical processes indexed by univariate functions and hence is not applicable to our problem. Alternatively, we can use Rio's coupling \citep[see][]{Ri94}, but to apply Rio's coupling, we would have to assume (at least) that $Y$ is bounded (rather than finite fourth moment) and $K_{n}$ has total variation of order $h_{n}^{-\alpha}$ (which requires additional conditions on the measurement error distribution). By employing the techniques developed in \cite{ChChKa14a, ChChKa14b,ChChKa16}, we are able to avoid such restrictive conditions.

\section{Practical considerations}
\label{sec: practical considerations}

\subsection{Additional regularization}
\label{sec: additional regularization}

Estimation of the characteristic function, $\varphi_\U$, is difficult in practice.
Therefore, an additional regularization including a ridge parameter, for example, is useful for robustness.
Following \citet[][Sec. 4]{DeHaMe08}, we employ 
$$
\tilde\varphi_\U(t)
=
\widehat\varphi_\U(t) \cdot 1_{\{ t \in \widehat T \}}
+
\left(1 + \widehat a_\U t^2 \right)^{-\widehat b_\U} \cdot 1_{\{ t \not\in \widehat T \}},
$$
where $\widehat T$ is the largest interval around 0 in which $\widehat\varphi_{\U}$ is non-decreasing to the left of 0 and non-increasing to the right of 0,
and $(a_\U, b_\U) = (\widehat a_\U, \widehat b_\U)$ fits the empirical second and fourth moments of $\eta$ with the characteristic function $t \mapsto \left(1 + a_\U t^2 \right)^{- b_\U}$, i.e., $(\widehat a_\U, \widehat b_\U)$ are chosen to satisfy 
$
2 \widehat a_\U \widehat b_\U = \frac{1}{m} \sum_{j=1}^m \eta_j^2
$
and
$
12 \widehat a_\U^2 \ \widehat b_\U \left(\widehat b_\U + 1\right) = \frac{1}{m} \sum_{j=1}^m \eta_j^4.
$
In case where these coefficients are negative, we set $\widehat b_\U = 1$ following \citet[][Sec. 4]{DeHaMe08}.
That is, we set
\begin{align*}
\widehat a_\U &= \frac{1}{2 \widehat b_\U} \cdot \frac{1}{m} \sum_{j=1}^m \eta_j^2
\quad\text{ and }
\\
\widehat b_\U &=
\begin{cases}
1
&\text{if } \frac{1}{m} \sum_{j=1}^m \eta_j^4 \leq 3\left(\frac{1}{m} \sum_{j=1}^m \eta_j^2\right)^2
\\
\frac{3\left(\frac{1}{m} \sum_{j=1}^m \eta_j^2\right)^2}{\frac{1}{m} \sum_{j=1}^m \eta_j^4-3\left(\frac{1}{m} \sum_{j=1}^m \eta_j^2\right)^2}
&\text{if } \frac{1}{m} \sum_{j=1}^m \eta_j^4 > 3\left(\frac{1}{m} \sum_{j=1}^m \eta_j^2\right)^2
\end{cases}.
\end{align*}
With this approach, the robustness is achieved by a fully data-driven manner.
In simulation studies and real data analysis below, we use the estimated deconvolution kernel with this additional regularization:
$$
\tilde K_n(x) = \frac{1}{2\pi} \int_{\mathbb{R}} e^{-itx} \frac{\varphi_K(t)}{\tilde \varphi_\U(t/h)} dt.
$$

\subsection{Bandwidth Selection}
\label{sec: bandwidth selection}

For another practical consideration, we present a bandwidth selection rule based on the SIMEX methods \citep{DeHa08,DeHaJa15}.
The selection procedure consists of two steps.
In the first step, we choose an optimal bandwidth $\widetilde h_n$ in terms of balancing the supremum squared bias and the supremum variance (Section \ref{sec: optimal bandwidth selection}).
In the second step, we choose a scaling factor $\widetilde \chi_n$ such that $\widetilde \chi_n \widetilde h_n$ optimizes the coverage probability with respect to the nominal one (Section \ref{sec: undersmoothing bandwidth selection}).

\subsubsection{Optimal bandwidth selection}
\label{sec: optimal bandwidth selection}

To clarify the dependence on a candidate bandwidth $h$, write $K(x;h) = \frac{1}{2\pi} \int_\mathbb{R} e^{-itx} \frac{\varphi_K(t)}{\varphi_\U(t/h)} dt$, $s^2(x;h) = \Var (\{Y - g(x)\} K(x-W)/h; h))$, and
 $A(x;h) = \Ep [\{Y - g(x)\} K((x-W)/h; h)]$.
An optimal choice $h_n$ (with the $\log n$ factor ignored) minimizes the criterion function $\Gamma_n$ defined by
$
\Gamma_n(h) =
\vert\vert A^2(\cdot;h) \vert\vert_I  + \vert\vert s^2(\cdot;h) / n \vert\vert_I .
$
If the true function $g$ \textit{were} known, then a natural estimator of the optimal bandwidth would be 
$
\tilde h_n = \arg\min_{h>0} \vert\vert \tilde A_{n}^{2}(\cdot;h) \vert\vert_I  + \vert\vert \tilde s_{n}^{2}(\cdot;h) / n \vert\vert_I,
$
where
\begin{align*}
\tilde s_{n}^{2}(x;h) &= \frac{1}{n}\sum_{j=1}^n \{Y_j - g(x)\}^2 \tilde K^2_{n}((x-W_j)/h; h) - \tilde A_{n}^{2}(x;h)
\text{ and}\\
\tilde A_{n}(x;h) &= \frac{1}{n}\sum_{j=1}^n \{Y_j - g(x)\} \tilde K_{n}((x-W_j)/h; h).
\end{align*}
However, this is an infeasible procedure because we do not know the true function $g$ nor the distribution of $(Y,X)$ in practice.
Therefore, adapting \cite{DeHa08} to our framework, we use the SIMEX method outlined below to estimate the optimal $h_n$.

Let $\{\U^{\dagger}_1,\dots,\U^{\dagger}_n\}$ and $\{\U^{\dagger\dagger}_1,\dots,\U^{\dagger\dagger}_n\}$ be independent samples drawn with replacement from $\{\eta_1,\dots,\eta_m\}$, independently of the data $\{(Y_1,W_1),\dots,(Y_n,W_n)\}$.
Write $W^{\dagger}_j = W_j + \U^{\dagger}_j$ and $W^{\dagger\dagger}_j = W_j + \U^{\dagger}_j + \U^{\dagger\dagger}_j$ for each $j \in \{1,\dots,n\}$.
As counterparts of $s^2(x;h)$ and $A(x;h)$, consider
$s^{\dagger 2}(x;h) = \Var (\{Y - g_1(x)\} K((x-W^{\dagger})/h; h))$,
$A^{\dagger}(x;h) = \Ep [\{Y - g_1(x)\} K((x-W^{\dagger})/h; h)]$,
$s^{\dagger\dagger 2}(x;h) = \Var (\{Y - g_2(x)\} K((x-W^{\dagger\dagger})/h; h))$, and
$A^{\dagger\dagger}(x;h) = \Ep [\{Y - g_2(x)\} K((x-W^{\dagger\dagger})/h; h)]$,
where $g_1(x) = \Ep[Y \mid W=x]$ and $g_2(x) = \Ep [Y \mid W^\dagger=x]$.
Substituting estimates and the additional regularization suggested at the end of Section \ref{sec: additional regularization}, we can compute
\begin{align*}
\tilde s_{n}^{\dagger 2}(x;h) &= \frac{1}{n}\sum_{j=1}^n \{Y_j - \tilde g_1(x)\}^2 \tilde K^2_{n}((x-W_j^{\dagger})/h; h) - \tilde A_{n}^{\dagger 2}(x;h)
\text{ where}\\
\tilde A_{n}^{\dagger}(x;h) &= \frac{1}{n}\sum_{j=1}^n \{Y_j - \tilde g_1(x)\} \tilde K_{n}((x-W_j^{\dagger})/h; h),
\text{ and}\\
\tilde s_{n}^{\dagger\dagger 2}(x;h) &= \frac{1}{n}\sum_{j=1}^n \{Y_j - \tilde g_2(x)\}^2 \tilde K^2_{n}((x-W_j^{\dagger\dagger})/h; h) - \tilde A_{n}^{\dagger\dagger 2}(x;h)
\text{ where}\\
\tilde A_{n}^{\dagger\dagger}(x;h) &= \frac{1}{n}\sum_{j=1}^n \{Y_j - \tilde g_2(x)\} \tilde K_{n}((x-W_j^{\dagger\dagger})/h; h),
\end{align*}
with $\tilde g_1$ and $\tilde g_2$ denoting error-free estimates of $g_1$ and $g_2$, respectively, which are feasible with the observed and simulated data $\{(Y_1,W_1,W^\dagger_1),\dots,(Y_n,W_n,W^\dagger_n)\}$.
Let
\begin{align*}
\tilde \Gamma^{\dagger}_n(h) &=
\vert\vert \tilde A_{n}^{\dagger 2}(\cdot;h) \vert\vert_I  + \vert\vert \tilde s_{n}^{\dagger 2}(\cdot;h) / n \vert\vert_I
\text{ and}\\
\tilde \Gamma^{\dagger\dagger}_n(h) &=
\vert\vert \tilde A_{n}^{\dagger\dagger 2}(\cdot;h) \vert\vert_I  + \vert\vert \tilde s_{n}^{\dagger\dagger 2}(\cdot;h) / n \vert\vert_I.
\end{align*}

Since $\tilde \Gamma^{\dagger}_n(h)$ and $\tilde \Gamma^{\dagger\dagger}_n(h)$ are affected by the simulated data, $\{\U^{\dagger}_1,\dots,\U^{\dagger}_n\}$ and $\{\U^{\dagger\dagger}_1,\dots,\U^{\dagger\dagger}_n\}$, we average over a large number $B$ of their versions, $\{\tilde \Gamma^{\dagger}_{n,b}(h)\}_{b=1}^B$ and $\{\tilde \Gamma^{\dagger\dagger}_{n,b}(h)\}_{b=1}^B$.
We then estimate the optimal bandwidths
\begin{align*}
\tilde h_{n,1} = \arg\min_{h>0} \frac{1}{B} \sum_{b=1}^B \tilde \Gamma^{\dagger}_{n,b}(h)
\quad\text{ and }\quad
\tilde h_{n,2} = \arg\min_{h>0} \frac{1}{B} \sum_{b=1}^B \tilde \Gamma^{\dagger\dagger}_{n,b}(h)
\end{align*}
for $g_1$ and $g_2$, respectively.
Note that ``$W^\ddagger$ measures $W^\dagger$ in the same way that $W^\dagger$ measures $W$ and $W$ measures $X$'' \citep[][Sec. 2.3]{DeHa08}.
We therefore expect that $\widetilde h_{n,2} / \widetilde h_{n,1}$ is similar to $\widetilde h_{n,1} / \widetilde h_{n}$.
This motivates the selection of $h_n$ by the linear backward extrapolation
$$
\tilde h_{n} = \tilde h_{n,1}^2 / \tilde h_{n,2} .
$$

\subsubsection{Coverage probability optimal bandwidth selection}
\label{sec: undersmoothing bandwidth selection}

The bandwidth chosen in Section \ref{sec: optimal bandwidth selection} is based on balancing the bias and variance, and therefore fails to provide an asymptotically valid inference result.
In this section, we discuss a fully data-driven procedure of choosing a coverage-probability optimal choice of bandwidth.
To clarify the dependence on a candidate bandwidth $h$, write the confidence band by
$
\hat{\mathcal{C}}_{1-\tau}(x;h), \ x \in I.
$
We define the ideal bandwidth to be $\chi_n h_n$ that minimizes the coverage error as
$
\chi_n = \arg\min_{\chi \geq 1} \left\vert \Prob\left(g(x) \in \widehat{\mathcal{C}}_{1-\tau}(x;\chi h_n) \ \forall x \in I \right) - (1-\tau)\right\vert.
$
In practice, we do not know $g$ nor the distribution of $(Y,X)$.
We therefore adapt the SIMEX method of \citet[][Sec. 3.2]{DeHaJa15} to our framework.

The same notations for the simulated samples from Section \ref{sec: optimal bandwidth selection} carry over to the current section. 
Let
$
\hat{\mathcal{C}}^{\dagger}_{1-\tau}(x;h)
\text{ and }
\hat{\mathcal{C}}^{\dagger\dagger}_{1-\tau}(x;h),
x \in I
$
denote the confidence bands for $g_1$ and $g_2$, respectively, computed using the samples $\{(Y_1,W^{\dagger}_1),\dots,(Y_n,W^{\dagger}_n)\}$ and $\{(Y_1,W^{\dagger\dagger}_1),\dots,(Y_n,W^{\dagger\dagger}_n)\}$, respectively.
We repeat the simulation for large number $B$ of times to construct versions, $\{\hat{\mathcal{C}}^{\dagger}_{1-\tau,b}(x;h)\}_{b=1}^B$ and $\{\hat{\mathcal{C}}^{\dagger\dagger}_{1-\tau,b}(x;h)\}_{b=1}^B$, and estimate the coverage probabilities for $g_1$ and $g_2$ by
\begin{align*}
\widehat{CP}_{1-\tau}^{\dagger} (h) &= \frac{1}{B} \sum_{b=1}^B {1}_{\left\{ \tilde g_{1}(x) \in \hat{\mathcal{C}}^{\dagger}_{1-\tau,b}(x;h) \ \forall x \in I\right\}}
\qquad\text{and}\\
\widehat{CP}_{1-\tau}^{\dagger\dagger} (h) &= \frac{1}{B} \sum_{b=1}^B {1}_{\left\{ \tilde g_{2,b}(x) \in \hat{\mathcal{C}}^{\dagger\dagger}_{1-\tau,b}(x;h)  \ \forall x \in I \right\}},
\end{align*}
where $\tilde g_1$ and $\tilde g_{2,b}$ denote error-free estimates of $g_1$ and $g_2$, respectively, which are feasible with the observed data $\{(Y_1,W_1),\dots,(Y_n,W_n)\}$ and $b$-th simulated data $\{(Y_1,W_1^{\dagger}),\dots,(Y_n,W_n^{\dagger})\}$, respectively.

We then estimate the coverage-error-optimal bandwidths for $g_1$ and $g_2$ by $\tilde \chi_{n,1} \tilde h_{n,1}$ and $\tilde \chi_{n,2} \tilde h_{n,2}$, respectively, where
\begin{align*}
\tilde \chi_{n,1} &= \arg\min_{\chi \geq 1} \left\vert \widehat{CP}_{1-\tau}^{\dagger} (\chi \tilde h_{n,1}) - (1-\tau) \right\vert
\qquad\text{and}\\
\tilde \chi_{n,2} &= \arg\min_{\chi \geq 1} \left\vert \widehat{CP}_{1-\tau}^{\dagger\dagger} (\chi \tilde h_{n,2}) - (1-\tau) \right\vert.
\end{align*}
Finally, the coverage-error-optimal bandwidth for $g$ may be estimated by
$
\tilde \chi_n \tilde h_n
$
where, following a similar reasoning to the one made at the end of Section \ref{sec: optimal bandwidth selection}, $\tilde \chi_n$ is given by the linear backward extrapolation
$
\tilde \chi_n = \tilde \chi_{n,1}^2 / \tilde \chi_{n,2}.
$

\begin{remark}\label{remark:simex}
The bandwidth selection procedure presented above consists of two steps: the first step chooses the optimal bandwidth $\widetilde h_n$ due to \cite{DeHa08}, and the second step chooses the scaling factor $\widetilde \chi_n$ to optimize the coverage probability due to \cite{ DeHaJa15}.
To our best knowledge, this data-driven selection procedure has not been formally proven to be consistent with theoretical requirements of inference methods even under the simpler settings covered in the existing literature.
In order to ensure that the theoretical results of the paper still hold, we require the existence of a deterministic bandwidth $h_n$ such that $\widetilde \chi_n \widetilde h_n / h_n = 1 + o_p(h_n^{-1/2})$.
Indeed, the two-step SIMEX procedure produces a complicated stochastic sequence, and we hence remark that a theoretical investigation of this practical issue by itself deserves a future research.
\end{remark}

\section{Simulation studies}
\label{sec: simulation}
\subsection{Simulation Framework}

We consider two data generating models, reflecting two common patterns of data availability.
For the first model, the observed  data $\mathcal{D}_n = \{ (Y_{j},W_{j},\eta_{j}) \}_{j=1}^{n}$ is constructed by 
\begin{equation*}
\text{Model 1}: \ Y_j = g(X_j) + \epsilon_j , \ W_j = X_j + \U_j, \ \eta_{j} \stackrel{d}{=} \U_{j}, \ j=1,\dots,n
\end{equation*}
where  $X_{j}, \epsilon_{j}, \U_{j}$, and $\eta_{j}$ are  independent with $X_j \sim N(0,\sigma_X^2), \epsilon_{j} \sim N(0,1)$ and $\U_j \stackrel{d}{=} \eta_j \sim \mathrm{Laplace}\,(0,2^{-1/2})$.
The characteristic function of $\U_j$ is $\varphi_{\U}(t) = (1 + t^2/2)^{-1}$, which is non-vanishing on $\R$ and ordinary smooth of order $\alpha=2$. 
The signal-to-noise ratio is $\sqrt{\Var(X)/\Var(\U)} =  \sigma_X$.

The second model considers the repeated measurement setup:
\begin{equation*}
\text{Model 2}: Y_{j} = g(X_{j}) + \epsilon_{j}, \ W^{(k)}_{j} = X_{j} + \U^{(k)}_{j}, \ j=1,\dots,n; k=1,2,
\end{equation*}
where $X_{j},\epsilon_{j},\U_{j}^{(1)}$, and $\U_{j}^{(2)}$ are  independent with $X_{j} \sim N(0,\sigma_{X}^{2}), \epsilon_{j} \sim N(0,1)$ and $\U^{(k)}_{j} \sim \mathrm{Laplace}\,(0,1)$. 
We observe $\{ (Y_{j},W_{j}^{(1)},W_{j}^{(2)}) \}_{j=1}^{n}$. 
By defining $W_{j} :=  ( W^{(1)}_{j} + W^{(2)}_{j} ) / 2$ and $\eta_j :=  ( W^{(1)}_{j} - W^{(2)}_{j} ) / 2$,
we obtain the generated data $\mathcal{D}_{n} = \{ (Y_{j},W_{j},\eta_{j}) \}_{j=1}^{n}$ such 
that $W_{j} = X_{j} + \U_{j}$ with $\U_{j} = (\U_{1}+\U_{2})/2 \stackrel{d}{=} \eta_{j}$.
For Model 2, the characteristic function of $\U_{j}$ is $\varphi_{\U}(t) = (1+t^{2}/4)^{-2}$, 
which is non-vanishing on $\R$ and ordinary smooth with order $\alpha=4$. 
The signal-to-noise ratio is $\sqrt{\Var(X)/\Var(\U)} = \sigma_X$. 

Simulations are run across five different specifications of $g$ and two alternative values of $\sigma_{X}$.
The five specifications of $g$ are
$g(x) = x$, 
$g(x) = x^2$,
$g(x) = x^3$,
$g(x) = \sin(x)$, and
$g(x) = \cos(x)$.
The two alternative values of $\sigma_{X}$ are $\{1/\sqrt{1/4},1/\sqrt{1/3}\}$ for both Model 1 and Model 2. 
The variances of the errors are one fourth (25\%) and one third (33\%) that of $X$ with these values of $\sigma_{X}$. 

We use Monte Carlo simulations to evaluate the coverage probabilities of our confidence bands for $g$ on the interval $I = [-\sigma_X, \sigma_X]$.
To assess how informative these confidence bands are, we also record and present the median statistics of the average band lengths on $I$, and also visually illustrate realized bands via figures.
We use the kernel function $K$ defined by its Fourier transform $\varphi_K$ given by
\[
\varphi_K(t) 
= 1_{\{ |t| \leq c \}} + \exp\left\{ \frac{-b \exp(-b/(\vert t \vert - c)^2)}{(\vert t \vert - 1)^2} \right\} 1_{\{ c < |t| < 1 \}},
\]
where $b=1$ and $c=0.05$ \citep[cf.][]{McPo04, BiDuHoMu07}.
The function $\varphi_K$ is infinitely differentiable with support $[-1,1]$, $\varphi^{(\ell)}(0)=0$ for any $\ell \geq 1$, 
and its inverse Fourier transform $K$ is real-valued and integrable with $\int_\mathbb{R} K(x) dx = 1$.
We follow the additional regularization discussed in Section \ref{sec: additional regularization} and the bandwidth selection rule discussed in Section \ref{sec: bandwidth selection}. 

\subsection{Simulation Results}
Since the simulation results are qualitatively similar for all specifications of the function $g$, we only discuss simulation results for $g(x) = x^{3}$ and $g(x) = \sin (x)$ in the present section. 
Results for the remaining functions, $g(x)=x$, $g(x)=x^2$, and $g(x)=\cos(x)$, can be found in Appendix \ref{sec: additional simulation}. 
Table \ref{tab:simulation_results} presents simulated uniform coverage probabilities for (A) $g(x) = x^3$ and (B) $g(x)=\sin(x)$ by estimated confidence bands in $I=[-\sigma_X,\sigma_X]$, and the median statistics of the average band lengths on $I$ based on 1,000 Monte Carlo iterations.
Observe that the simulated coverage probabilities approaches the nominal probabilities as the sample size increases, and they are reasonably close for the sample size of $n=400$.
Also observe that the median statistics of the average band length tend to decrease as the sample size increases.

\begin{table}
	\centering
		\scalebox{1}{
		\begin{tabular}{ccccccccc}
		\hline\hline
		\multicolumn{3}{l}{(A) Regression: $g(x) = x^3$} && \multicolumn{2}{c}{Error Variance} && \multicolumn{2}{c}{Error Variance}\\
			& Nominal & Sample && \multicolumn{2}{c}{=1/4 (25\%)} && \multicolumn{2}{c}{=1/3 (33\%)}\\
		\cline{5-6}\cline{8-9}
			Model & Probability & Size ($n$) && Coverage & Length && Coverage & Length\\
		\hline
			1 & 0.900         
			 &   100  && 0.886 &14.802 && 0.894 &10.294\\
			&&   200  && 0.931 &10.211 && 0.917 & 7.222\\
			&&   400  && 0.914 & 7.391 && 0.912 & 5.286\\
		\cline{2-9}
			1 & 0.950
			 &   100  && 0.903 &16.527 && 0.919 &11.449\\
			&&   200  && 0.963 &11.405 && 0.947 & 8.029\\
			&&   400  && 0.955 & 8.215 && 0.950 & 5.883\\
		\hline
			2 & 0.900           
			 &   100  && 0.930 &16.587 && 0.927 &11.554\\
			&&   200  && 0.949 &11.030 && 0.935 & 7.763\\
			&&   400  && 0.925 & 7.288 && 0.893 & 5.097\\
		\cline{2-9}
			2 & 0.950
			 &   100  && 0.951 &18.428 && 0.950 &12.870\\
			&&   200  && 0.974 &12.271 && 0.962 & 8.684\\
			&&   400  && 0.956 & 8.149 && 0.941 & 5.691\\
		\hline\hline
\\
		\hline\hline
		\multicolumn{3}{l}{(B) Regression: $g(x) = \sin(x)$} && \multicolumn{2}{c}{Error Variance} && \multicolumn{2}{c}{Error Variance}\\
			& Nominal & Sample && \multicolumn{2}{c}{=1/4 (25\%)} && \multicolumn{2}{c}{=1/3 (33\%)}\\
		\cline{5-6}\cline{8-9}
			Model & Probability & Size ($n$) && Coverage & Length && Coverage & Length\\
		\hline
			1 & 0.900
			 &   100  && 0.862 & 2.134 && 0.866 & 2.363\\
			&&   200  && 0.890 & 1.619 && 0.874 & 1.512\\
			&&   400  && 0.886 & 1.194 && 0.889 & 1.142\\
		\cline{2-9}
			1 & 0.950
			 &   100  && 0.904 & 2.369 && 0.906 & 2.621\\
			&&   200  && 0.932 & 1.796 && 0.915 & 1.682\\
			&&   400  && 0.940 & 1.321 && 0.934 & 1.271\\
		\hline
			2 & 0.900            
			 &   100  && 0.849 & 2.049 && 0.863 & 2.321\\
			&&   200  && 0.883 & 1.636 && 0.874 & 1.526\\
			&&   400  && 0.897 & 1.139 && 0.871 & 1.068\\
		\cline{2-9}
			2 & 0.950
			 &   100  && 0.900 & 2.273 && 0.901 & 2.572\\
			&&   200  && 0.931 & 1.820 && 0.923 & 1.697\\
			&&   400  && 0.944 & 1.267 && 0.918 & 1.190\\
		\hline\hline
		\end{tabular}
		}
\medskip
	\caption{{\small Simulated uniform coverage probabilities of (A) $g(x) = x^3$ and (B) $g(x) = \sin(x)$ by estimated confidence bands in $I=[-\sigma_X,\sigma_X]$ under normally distributed $X$ and Laplace distributed $\U$. Also reported are the medians of the average band lengths on $I$. The simulated probabilities and lengths are computed for each of the two nominal coverage probabilities, 90\% and 95\%, based on 1,000 Monte Carlo iterations.}}
	\label{tab:simulation_results}
\end{table}

Note that our method applies even when $I$ is a singleton.
In this case, a uniform confidence band boils down to a pointwise one.
We next run simulations for the pointwise cinfidence band with $I=\{x\}$ for each location $x \in [-\sigma_X,\sigma_X]$.
The second column group of Table \ref{tab:simulation_results_pointwise} summarizes simulation results regarding the pointwise band.
For comparisons, the first column group of Table \ref{tab:simulation_results_pointwise} also displays the simulation results regarding the uniform band copied from Table \ref{tab:simulation_results}.
For the pointwise band, in addition to the uniform coverage probabilities and the medians of the average band lengths, we also report the average fraction of the length of the subset of $I$ on which the pointwise band covers the true function $g$.
Notice that the uniform coverage probabilities for the pointwise band fall short of the nominal probabilities to large extents, and they deviate further as the sample size increases.
These results motivate the attractiveness of the uniform band for the purpose of assessing the global shape of the true function $g$.
With this said, we also observe that average covered portions by the poitwise band are fairly close to the nominal probabilities, which is consistent with a popular (but unclear) belief in the literature.
Theoretical discussions of the latter observation is out of the scope of this paper, but we would like to make this remark to suggest a directin of future research.

\begin{table}
	\centering
		\scalebox{1}{
		\begin{tabular}{cccccccccc}
		\hline\hline
		\multicolumn{3}{l}{(A) Regression: $g(x) = x^3$} && \multicolumn{2}{c}{Uniform Band} && \multicolumn{3}{c}{Pointwise Band}\\
		\cline{5-6}\cline{8-10}
			Error & Nominal & Sample && Uniform & Band && Uniform & Covered & Band\\
			Variance & Probability & Size ($n$) && Coverage & Length && Coverage & Portion & Length\\
		\hline
			1/4 & 0.900         
			 &   100  && 0.886 &14.802 && 0.618 & 0.935 & 9.727\\
(25\%)&&   200  && 0.931 &10.211 && 0.593 & 0.923 & 6.968\\
			&&   400  && 0.914 & 7.391 && 0.555 & 0.904 & 4.953\\
		\cline{2-10}
			1/4 & 0.950
			 &   100  && 0.903 &16.527 && 0.794 & 0.973 & 11.583\\
(25\%)&&   200  && 0.963 &11.405 && 0.780 & 0.966 & 8.295\\
			&&   400  && 0.955 & 8.215 && 0.733 & 0.954 & 5.894\\
		\hline
			1/3 & 0.900           
			 &   100  && 0.894 &10.294 && 0.628 & 0.933 & 6.797\\
(33\%)&&   200  && 0.917 & 7.222 && 0.577 & 0.914 & 5.012\\
			&&   400  && 0.912 & 5.286 && 0.542 & 0.899 & 3.577\\
		\cline{2-10}
			1/3 & 0.950
			 &   100  && 0.919 &11.449 && 0.802 & 0.970 & 8.131\\
(33\%)&&   200  && 0.947 & 8.029 && 0.768 & 0.962 & 5.989\\
			&&   400  && 0.950 & 5.883 && 0.728 & 0.951 & 4.259\\
		\hline\hline
\\
		\hline\hline
		\multicolumn{3}{l}{(B) Regression: $g(x) = \sin(x)$} && \multicolumn{2}{c}{Uniform Band} && \multicolumn{3}{c}{Pointwise Band}\\
		\cline{5-6}\cline{8-10}
			Error & Nominal & Sample && Uniform & Band && Uniform & Covered & Band\\
			Variance & Probability & Size ($n$) && Coverage & Length && Coverage & Portion & Length\\
		\hline
			1/4 & 0.900
			 &   100  && 0.862 & 2.134 && 0.544 & 0.908 & 1.421\\
(25\%)&&   200  && 0.890 & 1.619 && 0.538 & 0.912 & 1.053\\
			&&   400  && 0.886 & 1.194 && 0.536 & 0.908 & 0.790\\
		\cline{2-10}
			1/4 & 0.950
			 &   100  && 0.904 & 2.369 && 0.700 & 0.948 & 1.691\\
(25\%)&&   200  && 0.932 & 1.796 && 0.725 & 0.955 & 1.254\\
			&&   400  && 0.940 & 1.321 && 0.716 & 0.953 & 0.939\\
		\hline
			1/3 & 0.900            
			 &   100  && 0.866 & 2.363 && 0.541 & 0.901 & 1.392\\
(33\%)&&   200  && 0.874 & 1.512 && 0.550 & 0.903 & 1.050\\
			&&   400  && 0.889 & 1.142 && 0.521 & 0.897 & 0.750\\
		\cline{2-10}
			1/3 & 0.950
			 &   100  && 0.906 & 2.621 && 0.692 & 0.943 & 1.654\\
(33\%)&&   200  && 0.915 & 1.682 && 0.706 & 0.948 & 1.252\\
			&&   400  && 0.934 & 1.271 && 0.714 & 0.947 & 0.893\\
		\hline\hline
		\end{tabular}
		}
\medskip
	\caption{{\small Simulated uniform coverage probabilities of (A) $g(x) = x^3$ and (B) $g(x) = \sin(x)$ by estimated uniform and pointwise confidence bands in $[-\sigma_X,\sigma_X]$ under normally distributed $X$ and Laplace distributed $\U$. The middle column under the column group of pointwise band displays the fraction of the length of the subset of $I$ on which the pointwise band contains the true function $g$. Also reported are the medians of the average band lengths on $I$. The simulated probabilities and lengths are computed for each of the two nominal coverage probabilities, 90\% and 95\%, based on 1,000 Monte Carlo iterations.}}
	\label{tab:simulation_results_pointwise}
\end{table}

The informativeness of the uniform confidence bands are better assessed with visual presentations, besides the band lengths reported in the table.
Figure \ref{fig:sim_model1_x3_sinx} displays realizations of estimates confidence bands for functions, $g(x) = x^3$ and $g(x)=\sin(x)$, on $I$ for each of the two error variance ratios (1/4 and 1/3) under Model 1 (results for the remaining functions and/or model can be found in Appendix \ref{sec: additional simulation}).
Observe that the realized confidence bands are precise enough to provide us with ideas about the possible shapes of the true regression functions.
We may distinguish the five functions by the realized shapes of the confidence bands.
For these particular realizations, it is not possible to conclude that the band length shrink as the sample size increases, but the band lengths reported in Table \ref{tab:simulation_results} complement the figure to this end.

\begin{figure}
	\centering
		\begin{tabular}{ccc}
		&
		$n=200$&
		$n=400$\\
		\begin{minipage}{0.2\textwidth}$g(x)=x^3$\bigskip\\EV=1/4 (25\%)\end{minipage} &
		\begin{minipage}{0.3\textwidth}\includegraphics[width=1\textwidth]{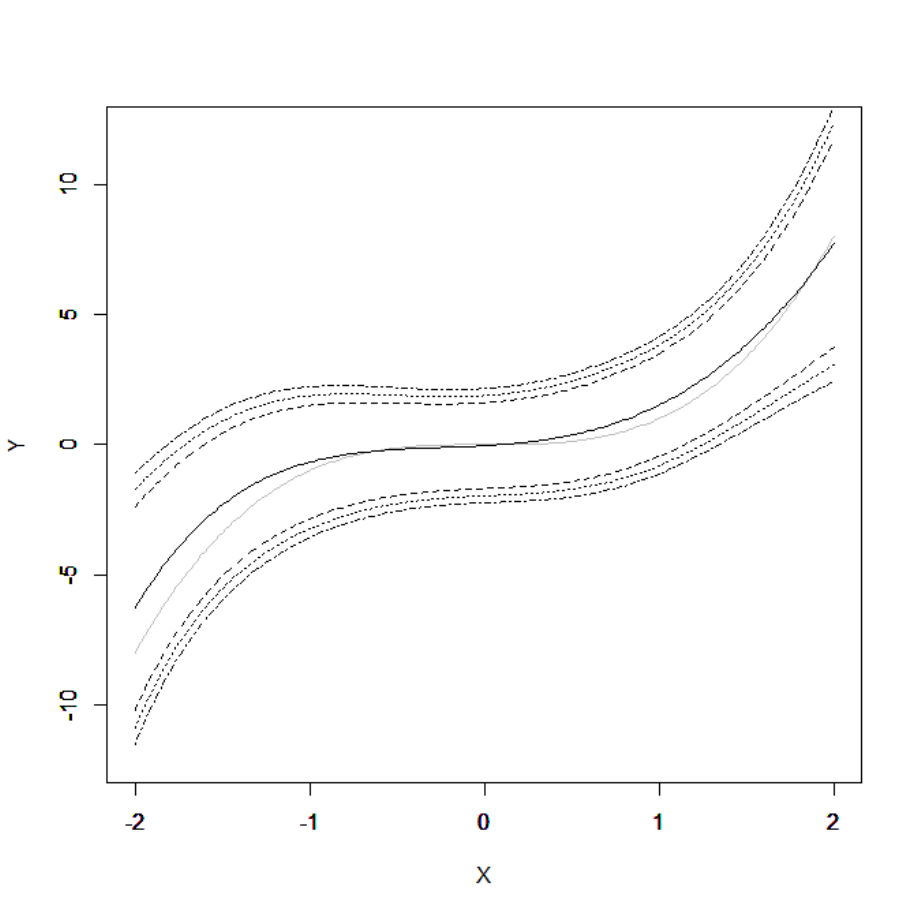}\end{minipage} &
		\begin{minipage}{0.3\textwidth}\includegraphics[width=1\textwidth]{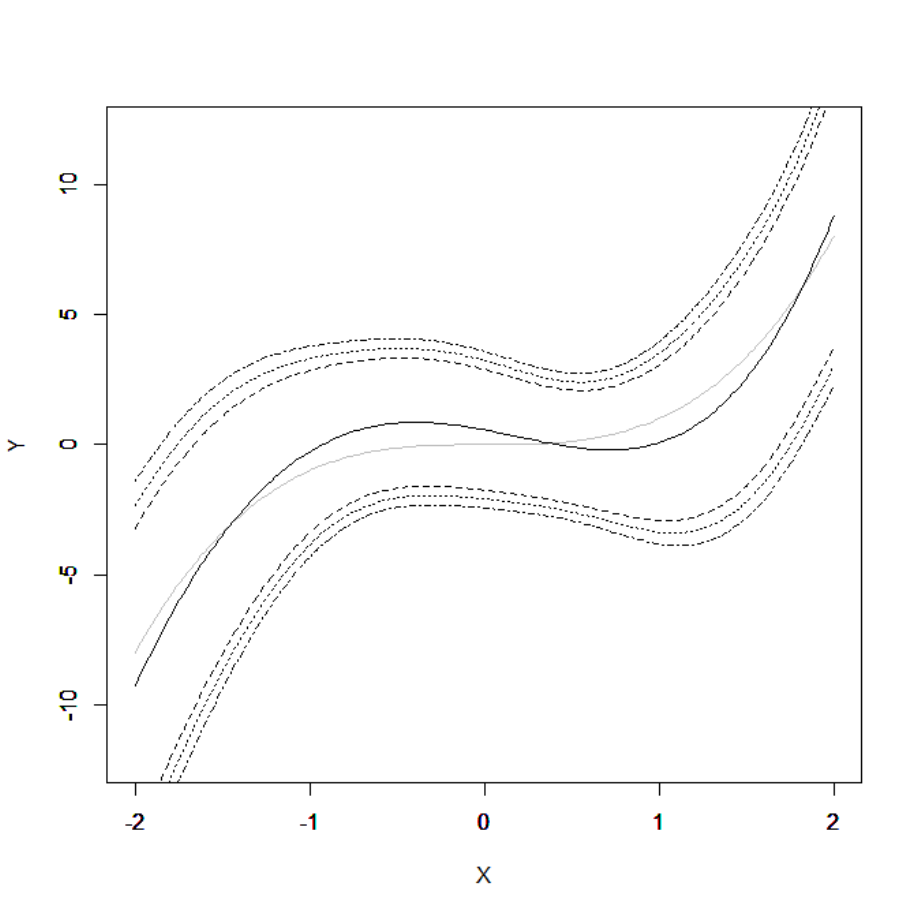}\end{minipage} \\
		\begin{minipage}{0.2\textwidth}$g(x)=x^3$\bigskip\\EV=1/3 (33\%)\end{minipage} &
		\begin{minipage}{0.3\textwidth}\includegraphics[width=1\textwidth]{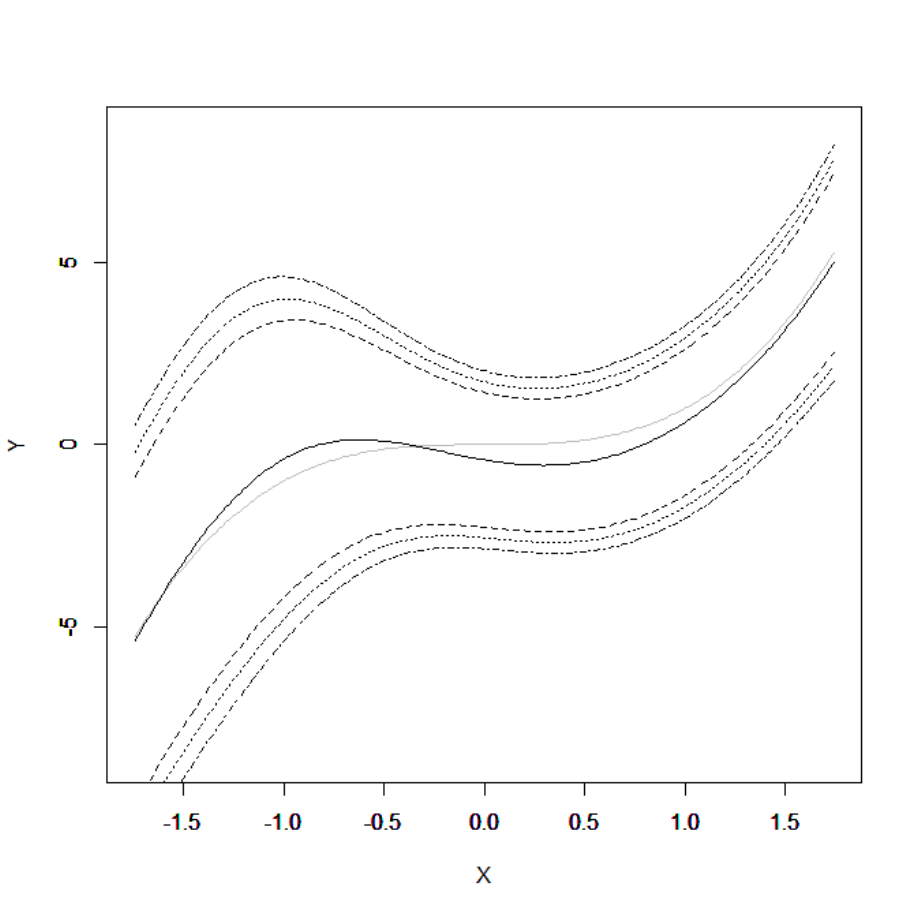}\end{minipage} &
		\begin{minipage}{0.3\textwidth}\includegraphics[width=1\textwidth]{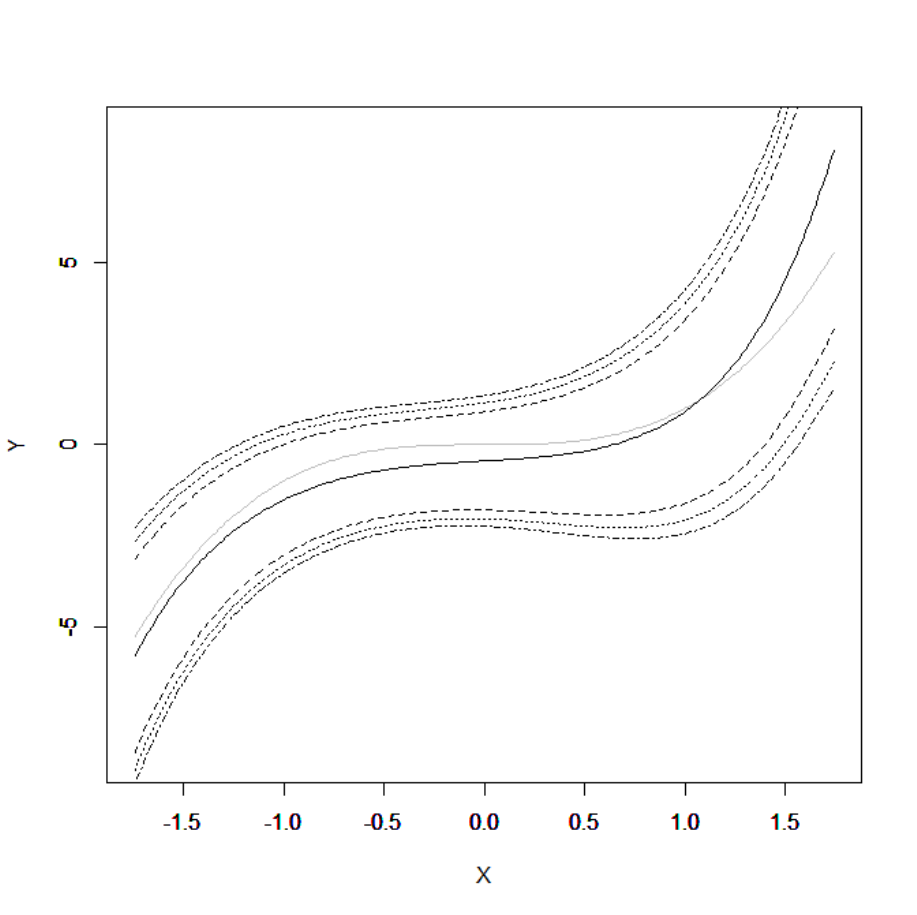}\end{minipage} \\
		\begin{minipage}{0.2\textwidth}$g(x)=\sin(x)$\bigskip\\EV=1/4 (25\%)\end{minipage} &
		\begin{minipage}{0.3\textwidth}\includegraphics[width=1\textwidth]{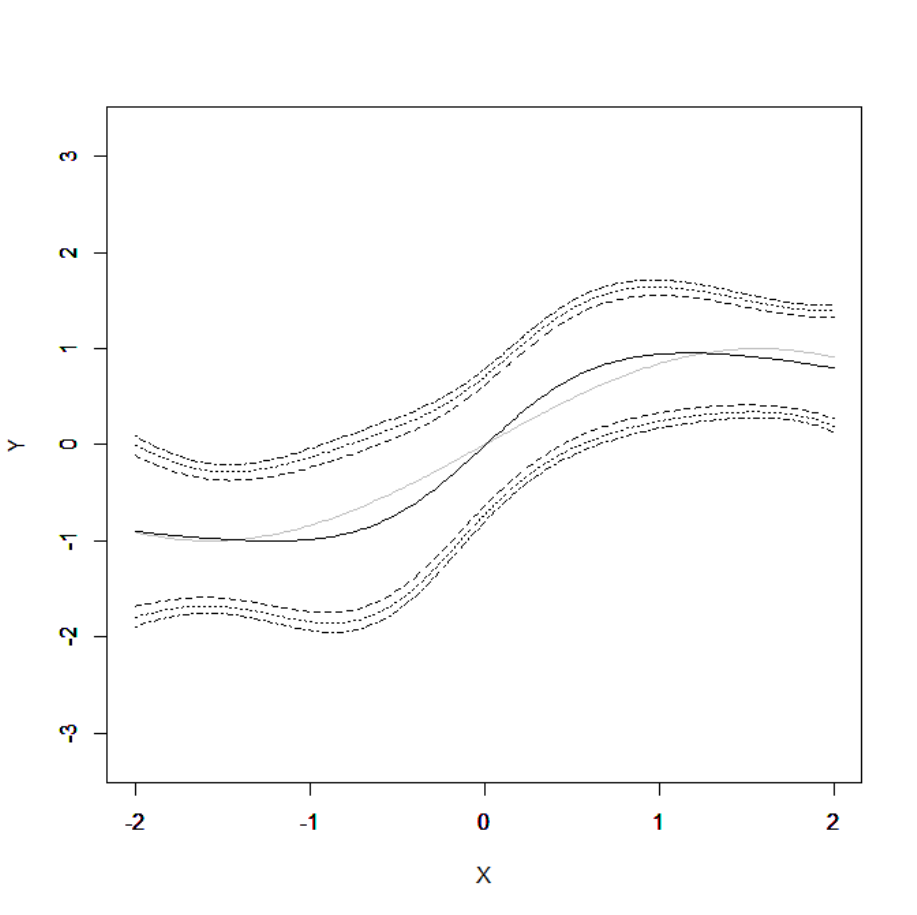}\end{minipage} &
		\begin{minipage}{0.3\textwidth}\includegraphics[width=1\textwidth]{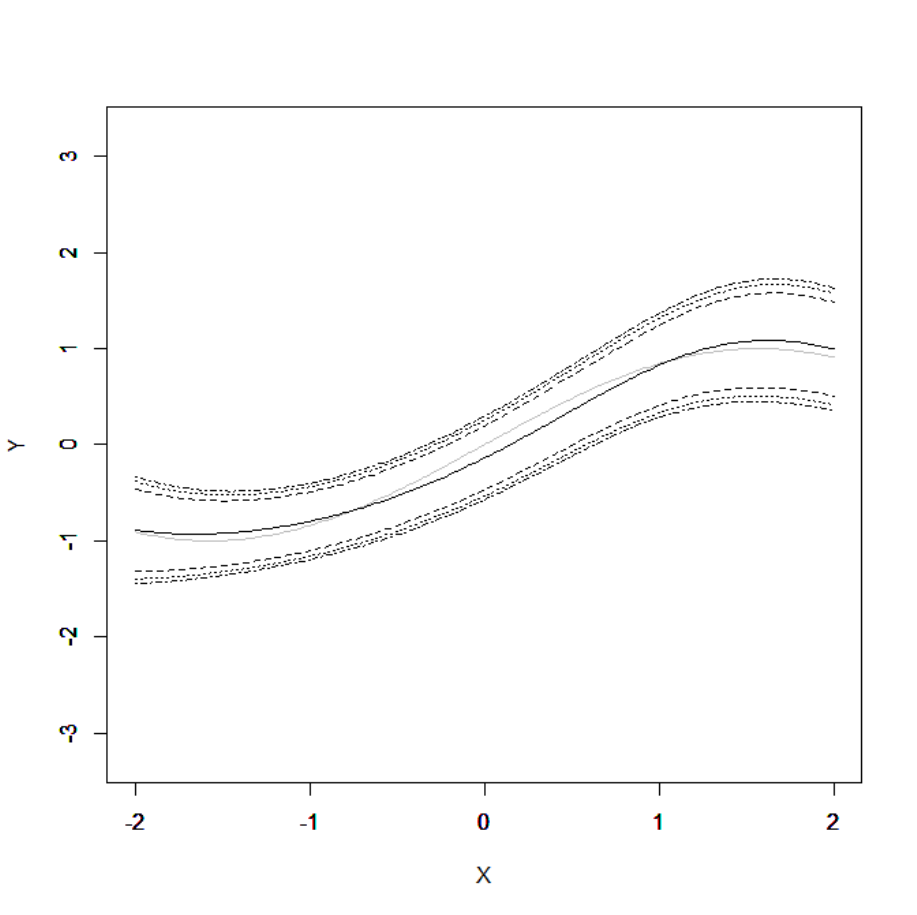}\end{minipage} \\
		\begin{minipage}{0.2\textwidth}$g(x)=\sin(x)$\bigskip\\EV=1/3 (33\%)\end{minipage} &
		\begin{minipage}{0.3\textwidth}\includegraphics[width=1\textwidth]{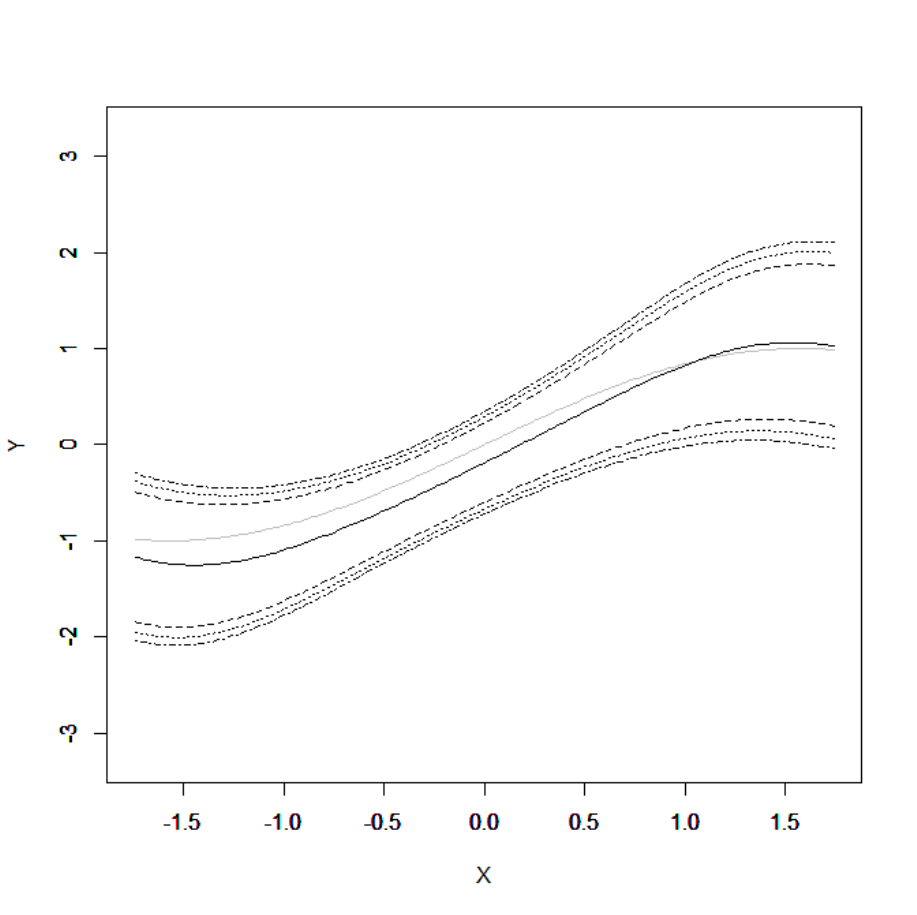}\end{minipage} &
		\begin{minipage}{0.3\textwidth}\includegraphics[width=1\textwidth]{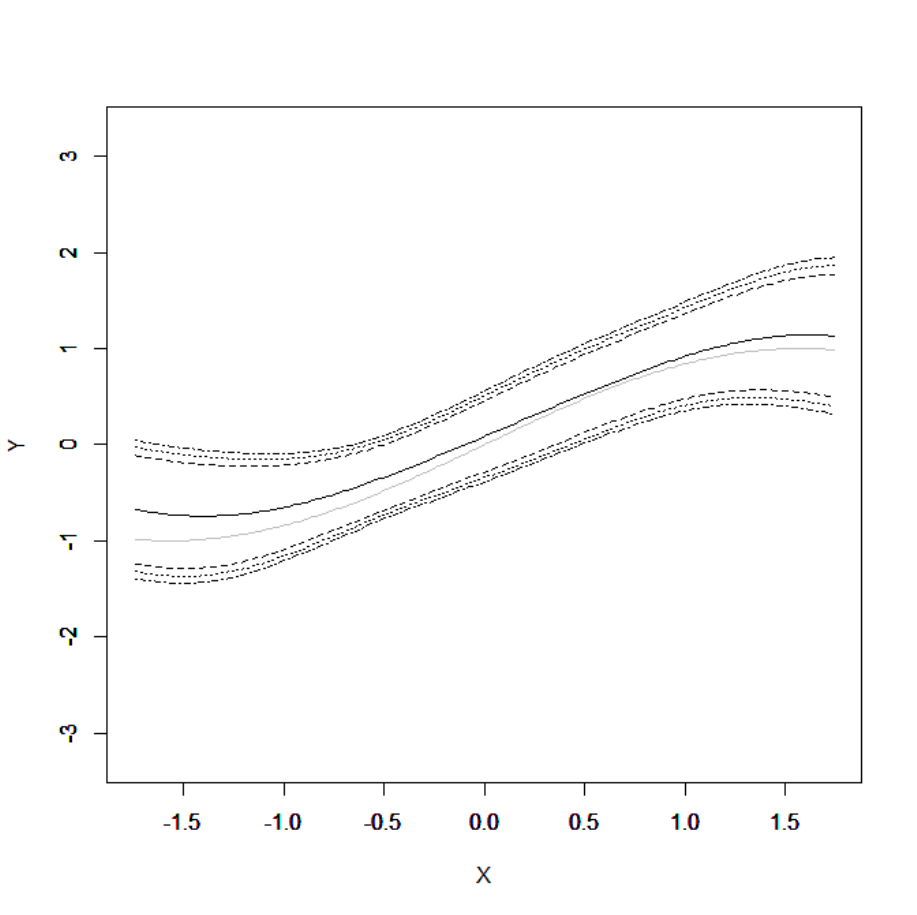}\end{minipage}
		\end{tabular}
	\caption{{\small Confidence bands for $g(x)=\sin(x)$ and $g(x)=x^3$ in Model 1 for error variance ratios of $1/4$ and $1/3$. Gray curves indicate the true function, black solid curves indicate estimates, and dashed curves indicate the 80\%, 90\%, and 95\% confidence bands.}}
	\label{fig:sim_model1_x3_sinx}
\end{figure}

\section{Real data analysis}
\label{sec: real data analysis}

In this section, we apply our method to real data, and draw confidence bands for nonparametric regressions of medical prescription expenditure on body mass index (BMI).
We combine the following two data sets.
The National Health and Nutrition Examination Survey (NHANES) provides self-report-based BMI (kg/m$^2$) and clinically measured BMI (kg/m$^2$).
We denote the former by $W_j$ and the latter by $X_j$.
From the NHANES as a validation data set of size $m$, we can compute $\eta_j = W_j - X_j$ for each $j=1,\dots,m$.
The Panel Survey of Income Dynamics (PSID) provides self-report-based BMI (kg/m$^2$) and prescription expenses.
We denote the former by $W_j$ and the latter by $Y_j$.
Additional details of these two data sets can be found in Appendix \ref{sec: additional details}. 

Combining the NHANES of size $m$ and the PSID of size $n$, we obtain the generated data $\mathcal{D}_n = \{Y_1,\ldots,Y_n,W_1,\ldots,W_n,\eta_1,\ldots,\eta_m\}$ to which we can apply our method to draw confidence bands for the regression function $g$ of the model $Y = g(X) + \epsilon$ with $\E[\epsilon \mid X,\U] = 0$.
We set $I = [15,35]$ as the interval on which we draw confidence bands.
The kernel function and the bandwidth rule carry over form our simulation studies.
For the sequence $\{c_n\}_{n=1}^\infty$ used for bandwidth choice, we follow the recommendation which we made from our simulation results.
To account for the different medical conditions across ages, we categorize the sample into the following subsamples: (a) male individuals aged 20--34, (b) male individuals aged 35--49, and (c) male individuals aged 50--64.
Details of data processing can be found in Appendix \ref{sec: additional details}. 

\begin{figure}
	\centering
	\begin{tabular}{ccc}
		(a) Men Aged 20--34 & (b) Men Aged 35--49 & (c) Men Aged 50--64\\
		\includegraphics[width=0.32\textwidth]{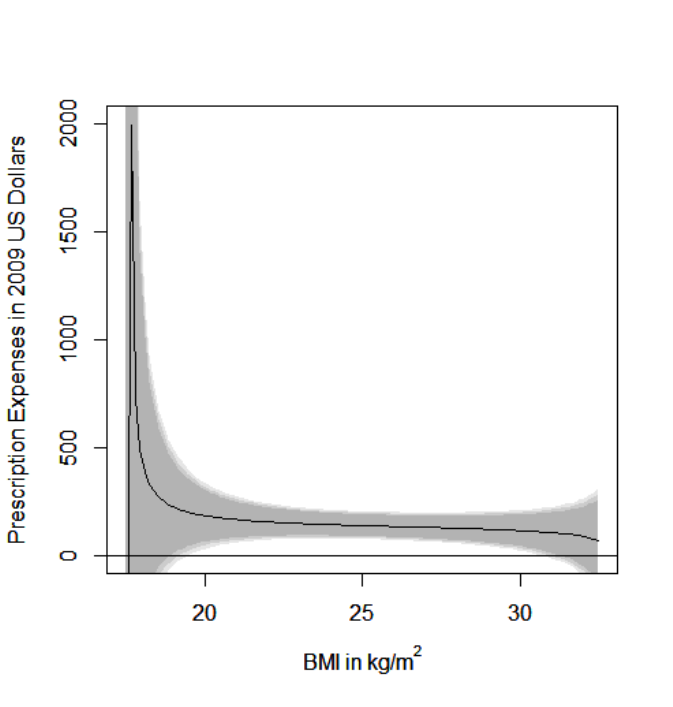} &
		\includegraphics[width=0.32\textwidth]{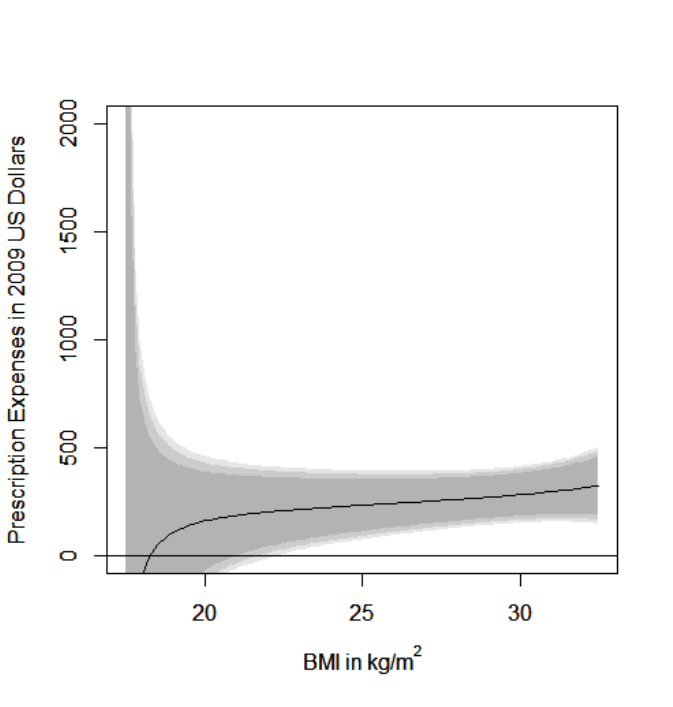} &
		\includegraphics[width=0.32\textwidth]{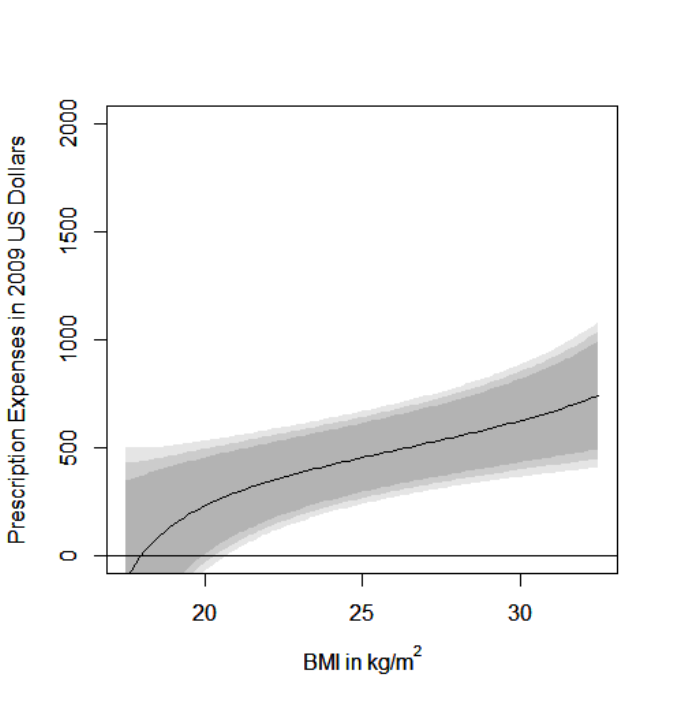}
	\end{tabular}
	\caption{{\small Estimates and confidence bands for the nonparametric regression of prescription expenses on BMI for (a) men aged from 20 to 34, (b) men aged from 35 to 49, and (c) men aged from 50 to 64. The horizontal axes measure the BMI in kg/m$^2$. The vertical axes measure the prescription expenses in 2009 US dollars. The estimates are indicated by solid black curves. The areas shaded by gray-scaled colors indicate 80\%, 90\%, and 95\% confidence bands.}}
	\label{fig:application_prescription}
\end{figure}

Figure \ref{fig:application_prescription} displays estimates and confidence bands for prescription expenses in 2009 US dollars as the dependent variable.
The estimates are indicated by solid black curves.
The areas shaded by gray-scaled colors indicate 80\%, 90\%, and 95\% confidence bands.
The three parts of the figure represent (a) men aged from 20 to 34, (b) men aged from 35 to 49, and (c) men aged from 50 to 64.
We see that the levels of prescription expenses tend to increase in age, as expected.
For the groups (a)--(b) of young men, prescription expenses exhibit little partial correlation with BMI.
For the group (c) of middle aged men, on the other hand, the relations turn into positive ones.
If we look at the 90\% confidence band for the group (c) of men aged from 50 to 64, annual average prescription expenses are
approximately \$0--\$495 if BMI $=20$,
approximately \$266--\$642 if BMI $=25$, and 
approximately \$395--\$849 if BMI $=30$.
These concrete numbers illustrate that confidence bands are useful to make interval predictions of incurred average costs, and this convenient feature has practical values added to the existing methods which only allow for reporting estimates with unknown extents of uncertainties.
We also present results on total medical expenses in Appendix \ref{sec: additional details}. 

Finally, we compare the uniform confidence bands and the pointwise confidence intervals in this application, as we did in the simulation studies.
Figure \ref{fig:application_prescription} displays estimates and pointwise confidence intervals for prescription expenses in 2009 US dollars as the dependent variable.
The estimates are indicated by solid black curves.
The areas shaded by gray-scaled colors indicate 80\%, 90\%, and 95\% confidence intervals.
The three parts of the figure represent (a) men aged from 20 to 34, (b) men aged from 35 to 49, and (c) men aged from 50 to 64.
Compared with the uniform confidence bands displayed in Figure \ref{fig:application_prescription}, notice that the pointwise confidence intervals displayed in Figure \ref{fig:application_prescription} exhibit somewhat shorter intervals.
Given the results of the simulation studies, these pointwise confidence intervals may yield misleading implications about the global shapes of the true regression functions.

\begin{figure}
	\centering
	\begin{tabular}{ccc}
		(a) Men Aged 20--34 & (b) Men Aged 35--49 & (c) Men Aged 50--64\\
		\includegraphics[width=0.32\textwidth]{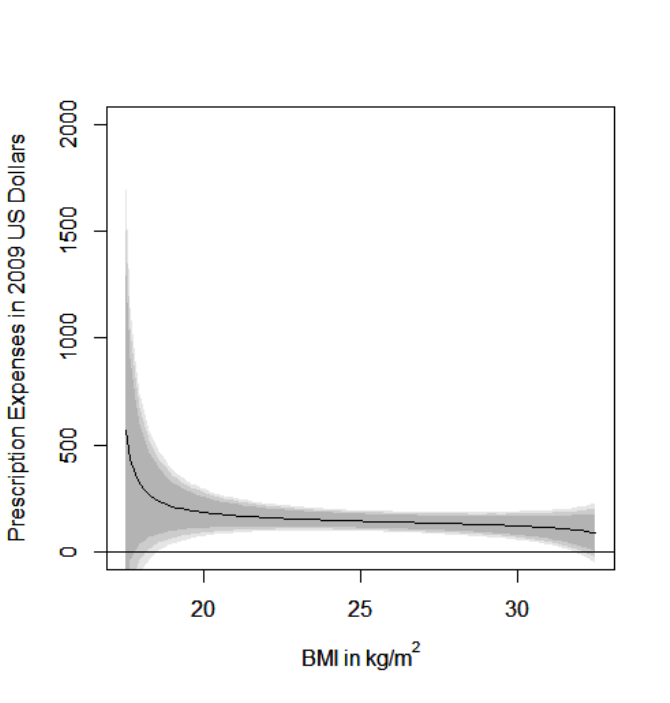} &
		\includegraphics[width=0.32\textwidth]{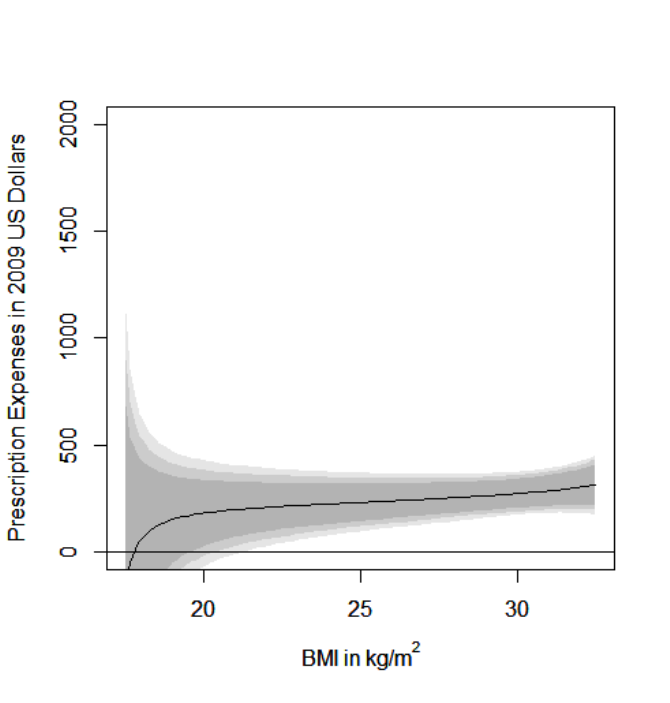} &
		\includegraphics[width=0.32\textwidth]{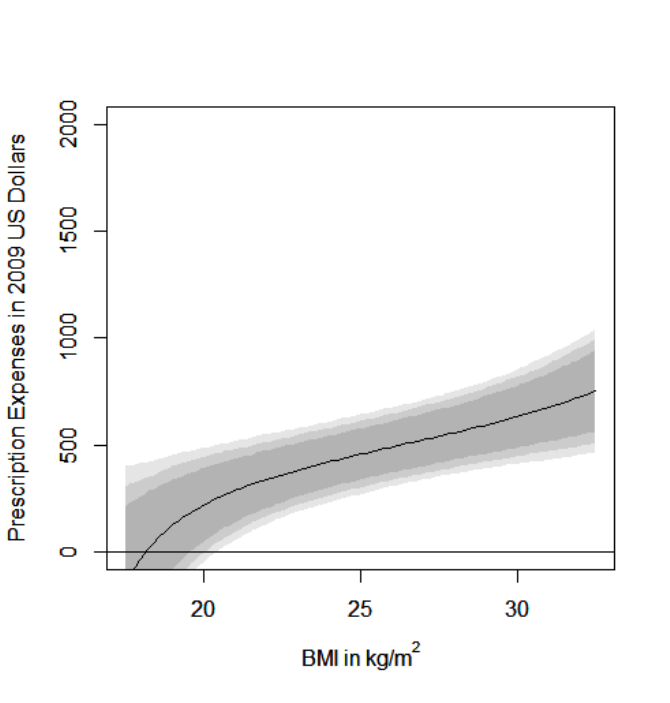}
	\end{tabular}
	\caption{{\small Estimates and pointwise confidence intervals for the nonparametric regression of prescription expenses on BMI for (a) men aged from 20 to 34, (b) men aged from 35 to 49, and (c) men aged from 50 to 64. The horizontal axes measure the BMI in kg/m$^2$. The vertical axes measure the prescription expenses in 2009 US dollars. The estimates are indicated by solid black curves. The areas shaded by gray-scaled colors indicate 80\%, 90\%, and 95\% confidence intervals.}}
	\label{fig:application_prescription_pointwise}
\end{figure}

\section{Extensions}
\label{sec: extensions}

\subsection{Application to specification testing}

The results of the present paper can be used for specification testing of the regression function $g$. 
Suppose that we want to test whether the regression function $g$ belongs to a parametric class $\{ g_{\theta} : \theta \in \Theta \}$ where $\Theta$ is a subset of a metric space (in most cases a Euclidean space). Popular specifications of $g$ are polynomials. In cases where $g$ is a polynomial, it is possible to estimate the coefficients with $\sqrt{n}$-rate under suitable regularity conditions \citep{ChMa85,HaNeIcPo91,ChSc98}. 
Suppose now that $g = g_{\theta}$ for some $\theta \in \Theta$ and $\theta$ can be estimated by $\hat{\theta}$ with a sufficiently fast rate, i.e., $\| g_{\hat{\theta}} - g_{\theta} \|_{I} = o_{\Pr} \{ h_{n}^{-\alpha}\{ nh_{n} \log (1/h_{n}) \}^{-1/2}  \}$, and that Assumption \ref{as: mean} is satisfied with $g = g_{\theta}$. Then it is not difficult to see from the proof of Theorem \ref{thm: validity of MB} that 
\begin{align*}
&\hat{f}_{X}(x) \sqrt{n}h_{n}(\hat{g}(x)-g_{\hat{\theta}}(x))/\hat{s}_{n}(x) \\
&\quad = \hat{f}_{X}(x) \sqrt{n}h_{n}(\hat{g}(x)-g_{\theta}(x))/\hat{s}_{n}(x) + \hat{f}_{X}(x) \sqrt{n}h_{n}(g_{\theta}(x)-g_{\hat{\theta}}(x))/\hat{s}_{n}(x) \\
&\quad =\hat{f}_{X}(x) \sqrt{n}h_{n}(\hat{g}(x)-g_{\theta}(x))/\hat{s}_{n}(x) + o_{\Pr} \{ (\log (1/h_{n}))^{-1/2} \},
\end{align*}
uniformly in $x \in I$, so that $\Pr \{ g_{\hat{\theta}}(x) \notin \hat{\mathcal{C}}_{1-\tau} (x) \ \text{for some} \ x \in I \} \to \tau$.
Therefore, the test that rejects the hypothesis that $g = g_{\theta}$ for some $\theta \in \Theta$ if $g_{\hat{\theta}} (x) \notin \hat{\mathcal{C}}_{1-\tau}(x)$ for some $x \in I$ is asymptotically of level $\tau$. 

%

The literature on specification testing for EIV regression is large. See \cite{HaMa07,So08,OtTa16}, and references therein. However, none of those papers considers $L^{\infty}$-based specification tests. 

\subsection{Additional regressors without measurement errors}

In practical applications, we may have additional regressors $Z$, possibly vector valued, without measurement errors. Suppose that we are interested in estimation and making inference on $g(x,z) = \Ep [ Y \mid X=x,Z=z]$.  In this section we consider two cases: one is fully nonparametric estimation of the conditional expectation $g(x,z)$, and the other is a partially linear model. 

\subsubsection{Fully nonparametric model}
Consider a noparametric regression model
\[
Y = g(X,Z) + \epsilon, \ W = X + U, \ \Ep[\epsilon \mid X,Z,U] = 0,
\]
where $Y,X,\epsilon,W$, and $U$ are univariate,  $Z$ is $d$-dimensional, and $U$ is independent of $(X,Z)$. 
We observe $(Y,W,Z)$, but do not observe $X$. Let $(Y_{1},W_{1},Z_{1}),\dots,(Y_{n},W_{n},Z_{n})$ be i.i.d. observations on $(Y,W,Z)$, and, as before, we assume that there is an independent sample from the measurement error distribution, $\eta_{1},\dots,\eta_{m} \sim f_{U}$ i.i.d. This additional sample may be dependent with $(Y_{1},W_{1},Z_{1}),\dots,(Y_{n},W_{n},Z_{n})$.

If $Z$ is finitely distributed, then $g(x,z)$, where $z$ is a mass point, can be estimated by using
only observations $j$ for which $Z_{j}=z$. So in what follows we focus on the case where $Z$ is continuous, or more precisely, $(X,Z)$ has a joint density $f_{XZ}(x,z)$. In this case, $g(x,z)$ can be estimated by using observations $j$ for which $Z_{j}$ is ``close''
to $z$, which can be implemented by using kernel weights. We will continue to use the same notations $K_{n}(x)$ and $\hat{K}_{n}(x)$ as defined in Section \ref{sec: methodology}. Let $L: \R^{d} \to \R$ be a kernel function (a function that integrates to one) for the $Z$ variable, and consider the estimator $\hat{g}(x,z) = \hat{\mu}(x,z)/\hat{f}_{XZ}(x,z)$, where
\[
\begin{split}
&\hat{\mu}(x,z) = \frac{1}{nh_{n}^{d+1}}\sum_{j=1}^{n}Y_{j}\hat{K}_{n}((x-W_{j})/h_{n})L((z-Z_{j})/h_{n}), \\
&\hat{f}_{XZ}(x,z) = \frac{1}{nh_{n}^{d+1}}\sum_{j=1}^{n} \hat{K}_{n}((x-W_{j})/h_{n})L((z-Z_{j})/h_{n}).
\end{split}
\]
Similarly to the previous case, $\hat{g}(x,z)-g(x,z)$ can be approximated as 
\[
\hat{g}(x,z)-g(x,z) \approx \frac{1}{f_{XZ}(x,z) nh_{n}^{d+1}} \sum_{j=1}^{n} \left [ \{ Y_{j} - g(x,z) \} K_{n}((x-W_{j})/h_{n}) L((z-Z_{j})/h_{n}) - A_{n}(x,z) \right]
\]
with $A_{n}(x,z) = \Ep[\{ Y - g(x,z) \} K_{n}((x-W)/h_{n}) L((z-Z)/h_{n})]$.
So we will modify the multiplier bootstrap process as follows: for $\xi_{1},\dots,\xi_{n} \sim N(0,1)$ i.i.d. independent of the data $\mD_{n}$, set
\[
\hat{\mathsf{Z}}_{n}^{\xi} (x,z)= \frac{1}{\hat{s}_{n}(x,z)\sqrt{n}} \sum_{j=1}^{n} \xi_{j} \{ Y_{j} - \hat{g}(x,z) \} \hat{K}_{n}((x-W_{j})/h_{n})L((z-Z_{j})/h_{n}),
\]
where $\hat{s}_{n}(x,z) = \sqrt{\hat{s}_{n}^{2}(x,z)}$ is defined by 
\[
\hat{s}_{n}^{2}(x,z) = \frac{1}{n}\sum_{j=1}^{n}  \{ Y_{j} - \hat{g}(x,z) \}^{2} \hat{K}_{n}((x-W_{j})/h_{n})L((z-Z_{j})/h_{n}).
\]
Let $I \times J$ be a compact rectangle in $\R \times \R^{d}$ on which the joint density $f_{XZ}$ is bounded away from zero, and compute the conditional $(1-\tau)$-quantile of $\| \hat{\mathsf{Z}}_{n}^{\xi} \|_{I \times J}$ given $\mD_{n}$, denoted by $\hat{c}_{n}(1-\tau)$. Then a confidence confidence band for $g(x,z)$ is given by 
\[
\hat{\mathcal{C}}_{1-\tau}(x,z) = \left [ \hat{g}(x,z) \pm \frac{\hat{s}_{n}(x,z)}{\hat{f}_{XZ}(x,z)\sqrt{n}h_{n}^{d+1}} \hat{c}_{n}(1-\tau) \right ], \ (x,z) \in I \times J.
\]
We make the following assumption, which is analogous to Assumption \ref{as: mean}. Denote by $f_{WZ}(w,z)$ the joint density of $(W,Z)$. 
For a multiindex $\alpha = (\alpha_{1},\alpha_{2},\dots,\alpha_{d+1})$ (a vector of nonnegative integers), define the differential operator 
\[
D^{\alpha} = \frac{\partial^{|\alpha|}}{\partial x^{\alpha_{1}} \partial z_{1}^{\alpha_{2}} \cdots \partial z_{d}^{\alpha_{d+1}}}
\]
with $| \alpha | = \sum_{j=1}^{d+1}\alpha_{j}$. For a given $\beta, B > 0$, let $\Sigma_{d+1} (\beta,B)$ denote the class of functions $f: (x,z) \mapsto f(x,z), \R \times \R^{d} \to \R$ such that $f$ is $k$-times continuously differentiable and 
\[
| D^{\alpha} f (x,z) - D^{\alpha}f(\overline{x},\overline{z})| \le B \| (x,z) - (\overline{x},\overline{z}) \|^{\beta - k}, \ \forall (x,z), (\overline{x},\overline{z}) \in \R \times \R^{d}
\]
for any multiindex $\alpha = (\alpha_{1},\alpha_{2},\dots,\alpha_{d+1})$ with $| \alpha | = k$, where $k$ is the integer such that $k < \beta \le k+1$. 
For $t \in \R^{d+1} = \R \times \R^{d}$, we decompose $t$ as $t = (t_{1},t_{-1}) \in \R \times \R^{d}$ (we use $(t_{1},t_{-1})$ instead of more precise $(t_{1},t_{-1}^{T})^{T}$ for the sake of notational simplicity). 

\begin{assumption}
\label{as: multivariate}
We assume the following conditions. 
\begin{enumerate}
\item[(i)] $\Ep[Y^{4}] < \infty$, the function $w \mapsto \Ep[Y^{2} \mid W=w,Z=z]f_{WZ}(w,z)$ is bounded and continuous, and for each $\ell=1,2$, and the function $w \mapsto \Ep[ |Y|^{2+\ell} \mid W=w,Z=z] f_{WZ}(w,z)$ is bounded.  
\item[(ii)] The functions $\varphi_{XZ}(t) = \Ep[ e^{i(t_{1}X+t_{-1}^{T}Z)}]$ and $\psi_{XZ}(t) = \Ep[ g(X,Z) e^{i(t_{1}X+t_{-1}^{T}Z)} ]$ for $t=(t_{1},t_{-1}) \in \R \times \R^{d}$ are integrable on $\R \times \R^{d}$.
\item[(iii)] Condition (iii) in Assumption \ref{as: mean} holds. 
\item[(iv)] The functions $f_{XZ}$ and $gf_{XZ}$  belong to $\Sigma_{d+1} (\beta,B)$ for some $\beta > (d+1)/2$ and $B > 0$. Let $k$ denote the integer such that $k <\beta\leq k+1$. 
\item[(v)] Condition (v) in Assumption \ref{as: mean} is satisfied for the kernel function $K: \R \to \R$. In addition, let $L: \R^{d} \to \R$ be a $(k+1)$-th order kernel function such that its Fourier transform $\varphi_{L}(t_{-1})$ for $t_{-1} \in \R^{d}$ is supported in $[-1,1]^{d}$, and the function class $\mathcal{L} = \{ z \mapsto L(az + b): a > 0, b \in \R^{d} \}$  is a VC type class in the sense that, for some positive constants $A$ and $v$, 
\[
\sup_{Q} N(\mathcal{L}, \| \cdot \|_{Q}, \delta \| L \|_{\R}) \le (A/\delta)^{v}, \ 0 < \forall \delta \le 1,
\]
where $\sup_{Q}$ is taken over all finitely discrete distributions on $\R^{d}$.\footnote{For a probability measure $Q$ on a measurable space $(S,\mS)$ and a  class of measurable functions $\mF$ on $S$ such that $\mF \subset L^{2}(Q)$, let $N(\mF,\| \cdot \|_{Q,2},\delta)$ denote the $\delta$-covering number for $\mF$ with respect to the $L^{2}(Q)$-seminorm $\| \cdot \|_{Q,2}$.}
\item[(vi)] For all $(x,z) \in I \times J$, $f_{XZ}(x,z) > 0$ and $\Ep[ \{ Y-g(x,z) \}^{2} \mid W=x,Z =z] f_{WZ}(x,z) > 0$.
\item[(vii)] As $n \to \infty$, 
\begin{equation}
\frac{(\log (1/h_{n}))^{2}}{(n \wedge m)h_{n}^{2\alpha+2d+2}}   \to 0, \ \frac{nh_{n}^{d+1} \log (1/h_{n})}{m} \to 0,  \text{ and } h_{n}^{\alpha+\beta} \sqrt{nh_{n}^{d+1} \log (1/h_{n})} \to 0. 
\label{eq: bandwidth}
\end{equation}
\end{enumerate}
\end{assumption}

\begin{theorem}[Validity of multiplier bootstrap confidence band]
\label{thm: validity of MB part 2}
Under Assumption \ref{as: multivariate}, we have 
\[
\Pr \left \{ g(x,z) \in \hat{\mathcal{C}}_{1-\tau}(x,z) \ \forall (x,z) \in I \times J \right \} = 1-\tau+o(1)
\]
as $n \to \infty$. 
The supremum width of the band $\hat{\mathcal{C}}_{1-\tau}$ is $O_{\Pr}\{ h_{n}^{-\alpha} (nh_{n}^{d+1})^{-1/2} \sqrt{\log (1/h_{n})} \}$. 
\end{theorem}

\subsubsection{Partially linear model}

Alternatively, we can consider a partially linear model of the form 
\[
Y = Z^{T}\beta + g(X) + \epsilon, \ W = X + U , \ \E[\epsilon \mid X,Z,U] = 0,
\]
where $Z \in \R^{d}$ is observed without measurement errors (observable is  $(Y,Z,W))$. If we can estimate $\beta$ at a sufficiently fast rate, say, the parametric (i.e., $\sqrt{n}$) rate, by some estimator $\hat{\beta}$, then
we can modify the multiplier bootstrap confidence band by replacing $Y_{j}$ with $Y_{j} - Z_{j}^{T}\hat{\beta}$, and  the asymptotic validity of the modified confidence band will go through under similar regularity conditions to those of Theorem \ref{thm: validity of MB}. 

However, obtaining a $\sqrt{n}$-rate estimator for $\beta$ is challenging in the presence of measurement error in $X$. 
One might be tempted to modify the \cite{Ro88} estimator by simply using the deconvolution kernel, but this naive modification does not work. To understand this, assume first that $X$ can be observed; the idea of \cite{Ro88}  to estimate $\beta$ is 1) to regress $Z$ on $X$ to obtain an estimate of the conditional expectation $m(x) = \E[Z \mid X = x]$ by e.g. a kernel method, and then 2) to regress $Y$ on $Z - m(X)$ with $m(\cdot)$ replaced by its estimate. This leads to a $\sqrt{n}$-rate estimator of $\beta$ since 
\[
Y = (Z-m(X))^{T}\beta + \{ m(X)^{T}\beta + g(X) \} +\epsilon,
\]
and by construction $Z-m(X)$ is orthogonal to $m(X)^{T}\beta + g(X)$. In the presence of measurement error in $X$, one can estimate the conditional mean $m(x) = \E[Z \mid X=x]$ by using the deconvolution kernel, i.e.,
\[
\hat{m}(x) = \frac{(nh_{n})^{-1}\sum_{j=1}^{n}Z_{j}\hat{K}_{n}((x-W_{j})/h_{n})}{(nh_{n})^{-1}\sum_{j=1}^{n}\hat{K}_{n}((x-W_{j})/h_{n})}.
\]
The problem happens in the second step,  because we can not observe $X$ and so there is no way to regress $Y$ on $Z - \hat{m}(X)$. One could naively use $W$ instead of $X$, but then $Z -\hat{m}(W) \approx Z - m(W) = Z - \E[Z \mid X=x]|_{x=W}$, and so regressing $Y$ on $Z - \hat{m}(W)$ leads to an inconsistent estimator of $\beta$. 

One way to circumvent this fundamental difficulty intrinsic to the measurement error model is to make inference on $g$ without estimating $\beta$. This is possible when $Z$ is finitely discrete. 
To gain the intuition, suppose first that $Z$ is binary, $Z \in \{ 0,1 \}$. Then the function $g$ can be identified by 
\[
g(x) = \E[Y \mid Z=0,X=x],
\]
and so we can estimate $g$ by applying the deonvolution kernel method to the subsample $\{ j : Z_{j} = 0 \}$, i.e., 
\[
\hat{g}(x) = \frac{(nh_{n})^{-1}\sum_{j=1}^{n}Y_{j}1(Z_{j} = 0) \hat{K}_{n}((x-W_{j})/h_{n})}{(nh_{n})^{-1}\sum_{j=1}^{n}1(Z_{j} = 0)\hat{K}_{n}((x-W_{j})/h_{n})}.
\]
The multiplier bootstrap confidence band can be constructed by applying the bootstrap process to the subsample $\{ j : Z_{j} = 0 \}$. 
It is not hard to see that asymptotic validity of the resulting confidence band follows if the assumption of Theorem \ref{thm: validity of MB} holds ``conditionally on $Z=0$''. 

Consider a more general situation where $Z$ is possibly multidimensional and finitely discrete, i.e., $Z \in \{ z_{1},\dots,z_{L} \}$. Assume that there exist weights $\lambda_{1},\dots,\lambda_{L} \in \R$ such that $\sum_{\ell=1}^{L} \lambda_{\ell} =1$ and $\sum_{\ell=1}^{L} \lambda_{\ell} z_{\ell} = 0$. For instance if $Z \in \{ 1,2 \}$, then we can choose $\lambda_{1} = 2$ and $\lambda_{2} = -1$. Such weights need not be unique. Then the function $g$ can be identified by 
\[
g(x) = \sum_{\ell=1}^{L} \lambda_{\ell} \E[Y \mid Z=z_{\ell},X=x].
\]
In this case, we can estimate $g(x)$ by
\[
\hat{g}(x) = \frac{(nh_{n})^{-1}\sum_{j=1}^{n}Y_{j}\mathbb{I}(Z_{j})  \hat{K}_{n}((x-W_{j})/h_{n})}{(nh_{n})^{-1}\sum_{j=1}^{n}\mathbb{I}(Z_{j})\hat{K}_{n}((x-W_{j})/h_{n})},
\]
where $\mathbb{I}(z) = \sum_{\ell=1}^{L} \lambda_{\ell} 1(z = z_{\ell})$. The multiplier bootstrap process should be modified accordingly, and asymptotic validity of the resulting confidence band follows if the assumption of Theorem \ref{thm: validity of MB} holds conditionally on $Z = z$ for each $z \in \{ z_{\ell}  : \lambda_{\ell} > 0 \}$. 

\subsection{Confidence bands for conditional distribution functions}

The techniques used to derive confidence bands for the conditional mean in EIV regression can be extended to the conditional distribution function. Suppose now that we are interested in constructing confidence bands for the conditional distribution function $g(y,x) = \Pr (Y \leq y \mid X=x)$
 on a compact rectangle $J \times I$ where $J$ and $I$ are compact intervals, and where we do not observe $X$ but instead observe $W=X+\U$ with $\U$ (measurement error) being independent of $(Y,X)$. 
As before, we assume that in addition to an independent sample $\{ (Y_{1},W_{1}),\dots,(Y_{n},W_{n}) \}$ on $(Y,W)$, there is an independent sample $\{ \eta_{1},\dots,\eta_{m} \}$ from the measurement error distribution. Since $g(y,x) = \Ep[ 1(Y \leq y) \mid X=x]$ where $1(\cdot)$ denotes the indicator function, we may estimate $g(y,x)$ by $\hat{g}(y,x) = \hat{\mu}(y,x)/\hat{f}_{X}(x)$, where  
$\hat{\mu}(y,x) =(nh_{n})^{-1} \sum_{j=1}^{n} 1(Y_{j} \leq y) \hat{K}_{n}((x-W_{j})/h_{n})$. 
To construct a confidence band for $g(y,x)$, we apply the methodology developed in Section \ref{sec: methodology} with $Y_{j}$ replaced by $1(Y_{j} \leq y)$ for each $y$. Let 
$\hat{s}_{n}^{2}(y,x) = \frac{1}{n} \sum_{j=1}^{n} \{ 1(Y_{j} \leq y) - \hat{g}(y,x) \}^{2} \hat{K}_{n}^{2}((x-W_{j})/h_{n})$,
and generate independent standard normal random variables $\xi_{1},\dots,\xi_{n}$ independent of the data $\mD_{n}$. Consider the multiplier process $\hat{\mathsf{Z}}_{n}^{\xi}(y,x) =n^{-1/2} \sum_{j=1}^{n} \xi_{j} \{ 1(Y_{j} \leq y) - \hat{g}(y,x) \} \hat{K}_{n}((x-W_{j})/h_{n})/\hat{s}_{n}(y,x)$, 
and for $\tau \in (0,1)$,  let 
\[
\hat{c}_{n}(1-\tau) = \text{conditional $(1-\tau)$-quantile of $\| \hat{\mathsf{Z}}_{n}^{\xi} \|_{J \times I}$ given $\mD_{n}$}.
\]
Then the resulting confidence band for $g(y,x)$ on $J \times I$ is given by 
\[
\hat{\mC}_{1-\tau}(y,x) = \left [ \hat{g}(y,x) \pm \frac{\hat{s}_{n}(y,x)}{\hat{f}_{X}(x)\sqrt{n} h_{n}} \hat{c}_{n}(1-\tau) \right ], \ (y,x) \in J \times I.
\]
We make the following assumption, which is analogous to Assumption \ref{as: mean}. 

\begin{assumption}
\label{as: cdf} Let $I, J$ be compact intervals in $\R$.
(i) The function $(y,w) \mapsto \Pr(Y \leq y \mid W=w)$ is continuous in $w$ uniformly in $y \in J$. (ii) The characteristic function of $X$, $\varphi_{X}(t) = \Ep[ e^{itX}], t \in \R$, is integrable on $\R$. Furthermore, $\sup_{y \in J} \int_{\R} |\Ep[ g(y,X) e^{itX} ] | dt < \infty$. (iii) Condition (iii) in Assumption \ref{as: mean}. (iv) The functions $f_{X}$ and $g(y,\cdot)f_{X}(\cdot)$ belong to $\Sigma (\beta,B)$ for some $\beta > 1/2$ and $B > 0$ for all $y \in J$. Let $k$ denote the integer such that $k < \beta \leq k+1$. (v) Condition (v) in Assumption \ref{as: mean}. (vi) For all $x \in I$, $f_{X}(x) > 0$, and $\inf_{(y,x) \in J \times I} \Ep[ \{ 1(Y \leq y) - g(y,x) \}^{2} \mid W=x] f_{W}(x) > 0$. (vii) Condition (vii) in Assumption \ref{as: mean}.
\end{assumption}

\begin{theorem}
\label{thm: cdf}
Under Assumption \ref{as: cdf}, $\Pr \{ g(y,x) \in \hat{\mC}_{1-\tau}(y,x) \ \forall (y,x) \in J \times I\} \to 1-\tau$ as $n \to \infty$. 
Furthermore, the supremum width of the band $\hat{\mC}_{1-\tau}$ is $O_{\Pr}\{ h_{n}^{-\alpha}(nh_{n})^{-1/2}\sqrt{\log(1/h_{n})}\}$.
\end{theorem}

Theorem \ref{thm: cdf} is also a new result. 
\section{Conclusion}
\label{sec: conclusion}

In this paper, we develop a method to construct uniform confidence bands for nonparametric EIV regression function $g$. 
We consider the practically relevant case where the distribution of the measurement error is unknown.
We assume that there is an independent sample from the measurement error distribution, where the sample from the measurement error distribution need not be independent from the sample on response and predictor variables. 
Such a sample from the measurement error distribution is available if there is, for example, either 1) validation data or 2) repeated measurements on the latent predictor variable with measurement errors, one of which is symmetrically distributed. 
We establish asymptotic validity of the proposed confidence band for ordinary smooth measurement error densities, showing that the proposed confidence band contains the true regression function with probability approaching the nominal coverage probability. 
To the best of our knowledge, this is the first paper to derive asymptotically valid uniform confidence bands for nonparametric EIV regression. 
We also suggest a two-step SIMEX method for coverage-probability optimal
bandwidth choice.
Simulation studies verify the finite sample performance of the proposed confidence band. 
Finally, we discuss extensions of our results to specification testing, cases with additional regressors without measurement errors, and confidence bands for conditional distribution functions. 

There are a couple of directions for future research.
First, our bootstrap procedure exploits only the first order, and improvements may be possible by exploiting higher orders.
For a heuristic investigation, we ran additional simulations of a method that bootstraps $s_n(x)$ and $\widehat\varphi_U(t)$, in addition to simulations of our method. 
These additional simulations turned out exhibit slight improvements for some data generating models.
A theoretical investigation of such improvements deserves a topic of future research.
Second, as mentioned in Remark \ref{remark:simex}, the data-driven SIMEX bandwidth selection procedure \citep{ DeHaJa15} has not been formally proven to be consistent with theoretical requirements of inference methods even under the simpler settings covered in the existing literature.
A theoretical investigation of this practical issue also deserves another topic of future research.

\appendix

\section{Proofs}
\label{sec: proof}

\subsection{Technical tools}
In this section, we collect technical tools that will be used in the proofs of Theorems \ref{thm: Gaussian approximation} and \ref{thm: validity of MB}. 
The proofs rely on modern empirical process theory. For a probability measure $Q$ on a measurable space $(S,\mS)$ and a  class of measurable functions $\mF$ on $S$ such that $\mF \subset L^{2}(Q)$, let $N(\mF,\| \cdot \|_{Q,2},\delta)$ denote the $\delta$-covering number for $\mF$ with respect to the $L^{2}(Q)$-seminorm $\| \cdot \|_{Q,2}$.
The class $\mF$ is said to be pointwise measurable if there exists a countable subclass $\mG \subset \mF$ such that for every $f \in \mF$ there exists a sequence $g_{m} \in \mG$ with $g_{m} \to f$ pointwise. 
A function $F: S \to [0,\infty)$ is said to be an envelope for $\mF$ if $F(x) \geq \sup_{f \in \mF} |f(x)|$ for all $x \in S$. 
See Section 2.1 in \cite{vaWe96} for details.

\begin{lemma}[A useful maximal inequality]
\label{lem: maximal inequality}
Let $X,X_{1},\dots,X_{n}$ be i.i.d. random variables taking values in a measurable space $(S,\mathcal{S})$, and let $\mF$ be a pointwise measurable class of (measurable) real-valued functions on $S$ with measurable envelope $F$. Suppose that there exist constants $A \geq e$ and $V \geq 1$ such that
\[
\sup_{Q} N(\mF,\| \cdot \|_{Q,2},\delta \| F \|_{Q,2}) \leq (A/\delta)^{V}, \ 0 < \forall \delta \leq 1,
\]
where $\sup_{Q}$ is taken over all finitely discrete distributions on $S$.
Furthermore, suppose that $0 < \Ep[F^{2}(X)] < \infty$, and let $\sigma^{2} > 0$ be any positive constant such that $\sup_{f \in \mF} \Ep[f^{2}(X)] \leq \sigma^{2} \leq \Ep[F^{2}(X)]$. 
Define $B= \sqrt{\E[ \max_{1 \leq j \leq n} F^{2}(X_{j}) ]}$.  Then
\begin{align*}
&\E \left [ \left \|  \frac{1}{\sqrt{n}} \sum_{j=1}^{n} \{ f(X_{j}) - \Ep[ f(X) ] \} \right \|_{\mF} \right ] \\
&\quad \leq C \left [  \sqrt{V\sigma^{2} \log \left ( \frac{A \sqrt{\Ep[F^{2}(X)]}}{\sigma} \right ) } + \frac{VB}{\sqrt{n}} \log \left ( \frac{A \sqrt{\Ep[F^{2}(X)]}}{\sigma} \right ) \right ],
\end{align*}
 where $C > 0$ is a universal constant. 
\end{lemma}
\begin{proof}
See Corollary 5.1 in \cite{ChChKa14a}.
\end{proof}

\begin{lemma}[An auxiliary maximal inequality]
\label{lem: moment inequality}
Let $\zeta_{1},\dots,\zeta_{n}$ be random variables such that $\Ep [|\zeta_{j}|^{r}] < \infty$ for all $j=1,\dots,n$ for some $r \geq 1$. Then 
\[
\Ep\left[ \max_{1 \leq j \leq n} | \zeta_{j} | \right ] \leq n^{1/r} \max_{1 \leq j \leq n} ( \Ep [|\zeta_{j}|^{r}] )^{1/r}. 
\]
\end{lemma}
\begin{proof}
This inequality is well known, and follows from Jensen's inequality. Indeed, $\Ep[ \max_{1 \leq j \leq n} |\zeta_{j}|] \leq (\Ep[\max_{1 \leq j \leq n}|\zeta_{j}|^{r}])^{1/r} \leq (\sum_{j=1}^{n} \Ep [|\zeta_{j}|^{r}])^{1/r} \leq n^{1/r} \max_{1 \leq j \leq n} ( \Ep [|\zeta_{j}|^{r}] )^{1/r}$. 
\end{proof}

The following \textit{anti-concentration} inequality for the supremum of a Gaussian process will play a crucial role in the proofs of Theorems \ref{thm: Gaussian approximation} and \ref{thm: validity of MB}. 

\begin{lemma}[Anti-concentration for the supremum of a Gaussian process]
\label{lem: anti-concentration}
Let $T$ be a nonempty set, and let  $X=(X_{t} : t  \in T)$ be a tight Gaussian random variable in $\ell^{\infty}(T)$ with mean zero and $\Ep[ X_{t}^{2}] = 1$ for all $t \in T$. Then for any $h > 0$, 
\[
\sup_{x \in \R} \Pr \{ | \| X \|_{T} - x | \leq h \} \leq 4 h ( 1+ \Ep[ \| X \|_{T} ]).
\]
\end{lemma}
\begin{proof}
See Corollary 2.1 in \cite{ChChKa14b}; see also Theorem 3 in \cite{ChChKa15}. 
\end{proof}

\subsection{Proof of Theorem \ref{thm: Gaussian approximation}}
In what follows, we always assume Assumption \ref{as: mean}.
We first prove some preliminary lemmas.
Recall that $A_{n}(x) = \Ep [\{ Y -g(x) \} K_{n}((x-W)/h_{n})]$ and $s_{n}^{2}(x) = \Var (\{ Y -g(x) \} K_{n}((x-W)/h_{n}))$. Observe that $\| K_{n} \|_{\R} = O(h_{n}^{-\alpha})$ under our assumption. 
In what follows, the notation $\lesssim$ signifies that the left hand side is bounded by the right hand side up to a positive constant independent of $n$ and $x$.

\begin{lemma}
\label{lem: moment bound}
The following bounds hold: (i) $\| A_{n} \|_{I} = O(h_{n}^{\beta+1})$. (ii) For sufficiently large $n$, $\inf_{x \in I} s_{n}^{2}(x) \gtrsim h_{n}^{-2\alpha+1}$. (iii) For $\ell=0,1,2$, we have $\sup_{x \in \R}\Ep[ |YK_{n}((x-W)/h_{n})|^{2+\ell}] = O(h_{n}^{-(2+\ell)\alpha+1})$. 
\end{lemma}

\begin{proof}
(i). 
Since $\Ep[ Y e^{itW} ] = \Ep[ \{ g(X) + U \} e^{it(X+\U)} ] = \psi_{X}(t) \varphi_{\U}(t)$, 
\[
\Ep[ Y K_{n}((x-W)/h_{n}) ] = \frac{h_{n}}{2\pi} \int_{\R} e^{-itx} \psi_{X}(t) \varphi_{K}(th_{n}) dt.
\] 
Since $\psi_{X}(\cdot)$ and $\varphi_{K} (\cdot h_{n})$ are the Fourier transforms of $gf_{X}$ and $h_{n}^{-1} K(\cdot/h_{n})$, respectively, the Fourier inversion formula yields that
\begin{align*}
\frac{h_{n}}{2\pi} \int_{\R} e^{-itx} \psi_{X}(t) \varphi_{K}(th_{n}) dt 
= \int_{\R} g(w) f_{X}(w) K((x-w)/h_{n}) dw.
\end{align*}
Note that the far left and right hand sides are continuous in $x$, and so the equality holds for all $x \in \R$.  Likewise, we have $\Ep[ K_{n}((x-W)/h_{n}) ] = \int_{\R} f_{X}(w) K((x-w)/h_{n})dw$ for all $x \in \R$, so that
\begin{align*}
A_{n}(x) 
= h_{n} \int_{\R} \{ g(x-h_{n}w) - g (x) \} f_{X}(x-h_{n}w) K(w) dw. 
\end{align*}
By the Taylor expansion, for any $x,w \in \R$,
\begin{align*}
&\{ g(x-h_{n}w) - g (x) \} f_{X}(x-h_{n}w) 
 = \sum_{j=0}^{k-1} \frac{(g f_{X})^{(j)}(x) - g(x) f_{X}^{(j)}(x)}{j!} (-h_{n}w)^{j} \\
&\qquad +\frac{(g f_{X})^{(k)}(x-\theta h_{n}w) - g(x) f_{X}^{(k)}(x-\theta h_{n}w)}{k!} (-h_{n}w)^{k},
\end{align*}
for some $\theta \in [0,1]$. 
Since $\int_{\R} w^{j} K(w) dw = 0$ for $j=1,\dots,k$ and $f_{X}, gf_{X} \in \Sigma (\beta,B)$, we have 
\begin{align*}
|A_{n}(x)| &= h_{n} \Bigg | \int_{\R} \Bigg [ \{ g(x-h_{n}w) - g (x) \} f_{X}(x-h_{n}w) \\
&\qquad - \frac{(g f_{X})^{(k)}(x) - g(x) f_{X}^{(k)}(x)}{k!} (-h_{n}w)^{k} \Bigg] K(w) dw \Bigg | \\
&\leq \frac{(1+\| g \|_{I})Bh_{n}^{\beta+1}}{k!} \int_{\R} |w|^{\beta} |K(w)| dw. 
\end{align*}
This shows that $\| A_{n} \|_{I} = O(h_{n}^{\beta+1})$. 

(ii). 
Since $A_{n}(x) = \Ep[\{ Y-g(x) \} K_{n}((x-W)/h_{n})] =O(h_{n}^{\beta+1})$ uniformly in $x \in I$, it suffices to show that $\inf_{x \in I} \Ep[ \{Y-g(x) \}^{2} K_{n}^{2}((x-W)/h_{n})] \gtrsim (1-o(1)) h_{n}^{-2\alpha+1}$.
We first verify that 
\[
\Ep[ Y \mid W=w] f_{W}(w) = \left ((gf_{X})*f_{\U} \right ) (w).
\]
To this end, it is enough to verify that the Fourier transforms of both sides agree. The Fourier transform of the right hand is $\psi_{X}(t) \varphi_{U}(t)$. On the other hand, the Fourier transform of the left hand side is 
\[
\begin{split}
\Ep[ \Ep[Y \mid W] e^{itW} ] &= \Ep [ \Ep[\underbrace{\Ep[Y \mid X,U]}_{=g(X)}\mid W] e^{itW}] = \Ep[ \Ep[g(X) \mid W] e^{itW}] = \E[ g(X) e^{itW} ] \\
&= \E[g(X) e^{it (X+U)}] = \E[g(X) e^{itX}] \E[e^{itU}] = \psi_{X}(t) \varphi_{U}(t). 
\end{split}
\]

Define
\begin{align*}
V(x,w) &= \Ep[ \{ Y-g(x) \}^{2} \mid W=w] f_{W}(w) \\
&= (\Ep[Y^{2} \mid W=w] + g^{2}(x)) f_{W}(w) -2g(x) \left ((gf_{X})*f_{\varepsilon} \right ) (w).
\end{align*}
The function $(gf_{X})*f_{\varepsilon}$ is bounded and continuous by boundedness of $gf_{X}$. 
Since $\Ep[Y^{2} \mid W=\cdot \,], \ f_{W}$, and $(gf_{X})*f_{\varepsilon}$ are bounded and continuous on $\R$, and $g$ is bounded and continuous on $I$, we have that 
the function $(x,w) \mapsto V(x,w)$ is bounded and continuous on $I \times \R$. In particular, since $V(x,x) > 0$ for all $x \in I$ under our assumption, we have that $\inf_{x \in I} V(x,x) > 0$. 

Now, observe that
\begin{align*}
\Ep[ \{Y-g(x) \}^{2} K_{n}^{2}((x-W)/h_{n})] &= \int_{\R} V(x,w) K_{n}^{2}((x-w)/h_{n}) dw \\
& =h_{n} \int_{\R} V(x,x-h_{n}w) K_{n}^{2}(w) dw. 
\end{align*}
Furthermore, we have that
\[
\int_{\R} K_{n}^{2}(w) dw = \frac{1}{2\pi} \int_{\R} \frac{|\varphi_{K}(t)|^{2}}{|\varphi_{\varepsilon}(t/h_{n})|^{2}} dt \sim h_{n}^{-2\alpha}
\]
by Plancherel's theorem. Hence, it suffices to show that 
\begin{equation}
\sup_{x \in I} \left | h_{n}^{2\alpha} \int_{\R}  \{ V(x,x-h_{n}w) - V(x,x) \} K_{n}^{2}(w) dw \right | \to 0.
\label{eq: moment continuity}
\end{equation}
From the proof of Lemma 3 in \cite{KaSa16},
we have that $h_{n}^{2\alpha} K_{n}^{2}(x) \lesssim \min \{ 1,x^{-2} \}$. 
By the definition of $V(x,w)$, for any $\rho > 0$, there exists sufficiently small $\delta >0$ such that $|V(x,x+w) - V(x,x)| \leq \rho$ for all $x \in I$ whenever $|w| \leq \delta$. Therefore,
\begin{align*}
&\sup_{x \in I} \int_{\R} |V(x,x-h_{n}w) - V(x,x) | h_{n}^{2\alpha}K_{n}^{2}(w) dw \\
&\quad \lesssim \rho \int_{|w| \leq \delta/h_{n}} \min \{1,w^{-2} \} dw + 2\| V \|_{I \times \R} \int_{|w| > \delta/h_{n}} w^{-2} dw  \lesssim \rho + o(1). 
\end{align*}

(iii). Pick any $\ell=0,1,2$. Since $\| K_{n} \|_{\R} \lesssim h_{n}^{-\alpha}$, we have 
\begin{align*}
\Ep[ |YK_{n}((x-W)/h_{n})|^{2+\ell}] &= h_{n} \int_{\R} \Ep[ |Y|^{2+\ell} \mid W=x-h_{n}w] |K_{n} (w)|^{2+\ell} f_{W}(x-h_{n}w)dw \\
&\lesssim h_{n}^{-\ell \alpha+1} \int_{\R} V_{\ell}(x-h_{n}w)K_{n}^{2}(w) dw \\
&\leq h_{n}^{-\ell \alpha + 1} \| V_{\ell} \|_{\R} \int_{\R} K_{n}^{2}(w) dw \lesssim h_{n}^{-(2+\ell)\alpha+1},
\end{align*}
where $V_{\ell}(w) = \Ep[ |Y|^{2+\ell} \mid W=w]f_{W}(w)$.  This completes the proof.
\end{proof}

\begin{lemma}
\label{lem: empirical chf}
$\| \hat{\varphi}_{\U} - \varphi_{\U} \|_{[-h_{n}^{-1},h_{n}^{-1}]} = O_{\Pr} \{ m^{-1/2} \log (1/h_{n}) \}$. 
\end{lemma}

\begin{proof}
See Lemma 4 in \cite{KaSa16}; see also Theorem 4.1 in \cite{NeRe09}. 
\end{proof}

Consider the following classes of functions
\begin{equation}
\begin{cases}
&\mF_{n}^{(1)} = \{ (y,w) \mapsto yK_{n}((x-w)/h_{n}) : x \in \R \}, \\
&\mF_{n}^{(2)} = \left \{ (y,w) \mapsto \frac{1}{s_{n}(x)} \{ y-g(x) \} K_{n}((x-w)/h_{n}) : x \in I \right \},  \\
&\mF_{n}^{(3)} = \{ (y,w) \mapsto \{ y-g(x) \} K_{n}^{2}((x-w)/h_{n}) : x \in I \}, \\
&\mF_{n}^{(4)} = \left \{ (y,w) \mapsto \frac{1}{s_{n}^{2}(x)} \{ y-g(x) \}^{2} K_{n}^{2}((x-w)/h_{n}) : x \in I \right \}.
\end{cases}
\label{eq: function class}
\end{equation}
In view of the fact that $\| K_{n} \|_{\R} \lesssim h_{n}^{-\alpha}$ and $\inf_{x \in I} s_{n}(x) \gtrsim h_{n}^{-\alpha+1/2}$, choose constants $D_{1},D_{2} > 0$ (independent of $n$) such that $\| K_{n} \|_{\R} \leq D_{1}h_{n}^{-\alpha}$ and $\| 1/s_{n} \|_{I} \leq D_{2}h_{n}^{\alpha-1/2}$. Let 
$F_{n}^{(1)}(y,w) = D_{1}|y| h_{n}^{-\alpha}, \ F_{n}^{(2)}(y,w) = D_{1}D_{2}(|y|+\| g \|_{I})/\sqrt{h_{n}}, \
F_{n}^{(3)}(y,w) = D_{1}(|y|+\| g \|_{I})h_{n}^{-2\alpha}$, and $F_{n}^{(4)} (y,w) = \{ F_{n}^{(2)}(y,w) \}^{2}$.
Note that  $F_{n}^{(\ell)}$ is an envelope function for $\mF_{n}^{(\ell)}$ for each $\ell=1,\dots,4$. 

\begin{lemma}
\label{lem: entropy}
There exist constants $A,v \geq e$ independent of $n$ such that 
\begin{equation}
\sup_{Q} N(\mF_{n}^{(\ell)},\| \cdot \|_{Q,2}, \delta \| F_{n}^{(\ell)} \|_{Q,2}) \leq (A/\delta)^{v}, \ 0 < \forall \delta \leq 1, \ell=1,\dots,4
\label{eq: entropy bound}
\end{equation}
where $\sup_{Q}$ is taken over all finitely discrete distributions on $\R^{2}$. 
\end{lemma}

\begin{proof}
Consider the following classes of functions: $\mathcal{K}_{n} = \{ w \mapsto K_{n}((x-w)/h_{n}) : x \in \R \}$ and $\mathcal{K}_{n}^{2} = \{ f^{2} : f \in \mathcal{K}_{n} \}$.
Lemma 1 in \cite{KaSa16} and Corollary A.1 in \cite{ChChKa14a} yield that there exist constants $A_{1},v_{1} \geq e$ independent of $n$ such that $\sup_{Q} N(\mathcal{K}_{n},\| \cdot \|_{Q,2}, D_{1}h_{n}^{-\alpha} \delta) \leq (A_{1}/\delta)^{v_{1}}$ and $\sup_{Q} N(\mathcal{K}^{2}_{n},\| \cdot \|_{Q,2}, D_{1}^{2}h_{n}^{-2\alpha} \delta) \leq (A_{1}/\delta)^{v_{1}}$ for all $0 < \delta \leq 1$.

In what follows, we only prove (\ref{eq: entropy bound}) for $\ell=2$; the proofs for the other cases are completely analogous given the above bounds on the covering numbers for $\mathcal{K}_{n}$ and $\mathcal{K}_{n}^{2}$. Let $\mathcal{H}_{n} = \{ y \mapsto \{ y-g(x) \}/s_{n}(x) : x \in I \}$, and observe that, since $\| 1/s_{n} \|_{I} \leq D_{2}h_{n}^{\alpha-1/2}$, there exist constants $A_{2},v_{2} \geq e$ independent of $n$ such that $\sup_{Q} N(\mathcal{H}_{n},\| \cdot \|_{Q,2}, \delta \| H_{n} \|_{Q,2}) \leq (A_{2}/\delta)^{v_{2}}$ for all $0 < \delta \leq 1$, where $H_{n}(y) = D_{2}(|y|+\| g \|_{I})h_{n}^{\alpha-1/2}$ is an envelope function for $\mathcal{H}_{n}$. This can be verified by a direct calculation, or observing that $\mathcal{H}_{n} \ (\subset \{ y \mapsto ay + b : a> 0, b \in \R \})$ is a VC subgraph class with VC index at most $4$ \citep[cf.][Lemma 2.6.15]{vaWe96}, and applying Theorem 2.6.7 in \cite{vaWe96}. 
Let $\mathcal{H}_{n}\mathcal{K}_{n} := \{ (y,w) \mapsto f_{1}(y)f_{2}(w) : f_{1} \in \mathcal{H}_{n}, f_{2} \in \mathcal{K}_{n} \}  \supset \mF_{n}^{(2)}$, and note that $H_{n}(y) D_{1}h_{n}^{-\alpha} = F_{n}^{(2)}(y,w)$. From Corollary A.1 in \cite{ChChKa14a}, there exist constants $A_{3},v_{3} \geq e$ independent of $n$ such that $\sup_{Q} N(\mathcal{H}_{n}\mathcal{K}_{n}, \| \cdot \|_{Q,2}, \delta \| F_{n}^{(2)} \|_{Q,2}) \leq (A_{3}/\delta)^{v_{3}}$ for all $0 < \delta \leq 1$. Now, the desired result follows from the observation that $N(\mF_{n}^{(2)}, \| \cdot \|_{Q,2}, 2\delta) \leq N(\mathcal{H}_{n}\mathcal{K}_{n}, \| \cdot \|_{Q,2}, \delta)$ for all $\delta > 0$.
\end{proof}

\begin{lemma}
\label{lem: rate}
We have 
\begin{align*}
&\| \hat{f}_{X}^{*} (\cdot)- \Ep[ \hat{f}_{X}^{*}(\cdot) ]  \|_{\R} = O_{\Pr} \{ h_{n}^{-\alpha} (nh_{n})^{-1/2} \sqrt{\log (1/h_{n})} \}, \\
&\|\Ep[ \hat{f}_{X}^{*}(\cdot) ] - f_{X} (\cdot) \|_{\R} = O(h_{n}^{\beta}) = o\{ h_{n}^{-\alpha} (nh_{n}\log (1/h_{n}))^{-1/2} \},
\end{align*}
and $\| \hat{\mu}^{*} (\cdot) - \Ep[\hat{\mu}^{*}(\cdot) ] \|_{\R} = O_{\Pr} \{ h_{n}^{-\alpha} (nh_{n})^{-1/2} \sqrt{\log (1/h_{n})} \}$.
\end{lemma}

\begin{proof}
The first two results are implicit in the proofs of Corollaries 1 and 2 in \cite{KaSa16}. 
To prove the last result, we shall apply Lemma \ref{lem: maximal inequality} to the class of functions $\mF_{n}^{(1)}$. From Lemma \ref{lem: moment bound}-(iii), we have that $\sup_{x \in \R}\Ep[  Y^{2}K^{2}_{n}((x-W)/h_{n}) ] =O(h_{n}^{-2\alpha+1})$. 
In view of the covering number bound for $\mF_{n}^{(1)}$ given in Lemma \ref{lem: entropy}, 
we may apply Lemma \ref{lem: maximal inequality} to $\mF_{n}^{(1)}$ to conclude that
\begin{align*}
&(nh_{n}) \Ep[ \| \hat{\mu}^{*}(\cdot) - \Ep[ \hat{\mu}^{*}(\cdot) ] \|_{\R}] = \Ep \left [ \left \| {\textstyle \sum}_{j=1}^{n} \{ f(Y_{j},W_{j}) - \Ep[f(Y,W)] \} \right \|_{\mF^{(1)}} \right ] \\
&\quad \lesssim h_{n}^{-\alpha} \sqrt{nh_{n}\log (1/h_{n})} + h_{n}^{-\alpha}\sqrt{\Ep[ \max_{1 \leq j \leq n} Y_{j}^{2}]} \log (1/h_{n}). 
\end{align*}
From Lemma \ref{lem: moment inequality}, we have $\Ep[ \max_{1 \leq j \leq n} Y_{j}^{2}] = O(n^{1/2})$, so that we have 
\[
(nh_{n}) \Ep[ \| \hat{\mu}^{*} (\cdot)- \Ep[ \hat{\mu}^{*}(\cdot) ] \|_{\R} ] \lesssim h_{n}^{-\alpha} \sqrt{nh_{n}\log (1/h_{n})} + h_{n}^{-\alpha}n^{1/4} \log (1/h_{n}),
\]
and the right hand side is $\lesssim h_{n}^{-\alpha} \sqrt{nh_{n}\log (1/h_{n})}$
from the first condition in (\ref{eq: bandwidth}). This completes the proof. 
\end{proof}

We are now in position to prove Theorem \ref{thm: Gaussian approximation}.

\begin{proof}[Proof of Theorem \ref{thm: Gaussian approximation}]
We divide the proof into two steps. 

\textbf{Step 1}. 
Let $r_{n}= h_{n}^{-\alpha}\{ nh_{n} \log (1/h_{n}) \}^{-1/2}$. We first prove that 
\[
\hat{g}(x) - g(x) = \frac{1}{f_{X}(x)} \frac{1}{nh_{n}} \sum_{j=1}^{n} \left [ \{ Y_{j}-g(x) \} K_{n}((x-W_{j})/h_{n}) - A_{n}(x) \right ] + o_{\Pr}(r_{n})
\]
uniformly in $x \in I$.

To this end, we shall show that $\| \hat{\mu} - \hat{\mu}^{*} \|_{\R} = o_{\Pr} (r_{n})$. First, observe from Lemma \ref{lem: empirical chf} that 
\[
\inf_{|t| \leq h_{n}^{-1}} | \hat{\varphi}_{\U} (t)| \geq \inf_{|t| \leq h_{n}^{-1}} | \varphi_{\U} (t)| -  O_{\Pr}\{ m^{-1/2} \log (1/h_{n})\} \gtrsim  (1-o_{\Pr}(1))h_{n}^{\alpha}.
\]
Let $\psi_{YW} (t) = \Ep[ Y e^{itW}] = \Ep[ \{ g(X) + \epsilon \} e^{it(X+\U)}] = \psi_{X}(t) \varphi_{\U}(t)$, and let $\hat{\psi}_{YW}(t)=n^{-1} \sum_{j=1}^{n} Y_{j} e^{itW_{j}}$. 
Decomposing the integral over $\R$ into $\int_{\{ \psi_{X} \neq 0 \}}$ and $\int_{\{ \psi_{X} = 0 \}}$, we obtain the following identity for  $\hat{\mu}(x)-\hat{\mu}^{*}(x)$: 
\begin{align*}
&\hat{\mu}(x) - \hat{\mu}^{*}(x) \\
&=\frac{1}{2\pi} \int_{\{ \psi_{X} \neq 0 \}} e^{-itx} \varphi_{K}(th_{n})\left \{ \frac{\hat{\psi}_{YW}(t)}{\psi_{YW}(t)} - 1 \right \} \left \{ \frac{\varphi_{\U}(t)}{\hat{\varphi}_{\U}(t)}  - 1 \right \} \psi_{X}(t) dt \\
&\quad +\frac{1}{2\pi} \int_{\{ \psi_{X} = 0 \}} e^{-itx} \frac{\varphi_{K}(th_{n})}{\varphi_{\U}(t)} \hat{\psi}_{YW}(t) \left \{ \frac{\varphi_{\U}(t)}{\hat{\varphi}_{\U}(t)} - 1 \right \} dt \\
&\quad +\frac{1}{2\pi} \int_{\R} e^{-itx} \varphi_{K}(th_{n})\left \{ \frac{\varphi_{\U}(t)}{\hat{\varphi}_{\U}(t)}  - 1 \right \} \psi_{X}(t) dt.
\end{align*}
Hence the Cauchy-Schwarz inequality yields that  
\begin{align}
&|\hat{\mu}(x) - \hat{\mu}^{*}(x)|^{2} \notag \\
&\lesssim  \left \{ \int_{\{ \psi_{X} \neq 0 \} \cap [-h_{n}^{-1},h_{n}^{-1}]} \left | \frac{\hat{\psi}_{YW}(t)}{\psi_{YW}(t)} - 1 \right |^{2} |\psi_{X}(t)|^{2} dt \right \} \left \{ \int_{-h_{n}^{-1}}^{h_{n}^{-1}} \left | \frac{\varphi_{\U}(t)}{\hat{\varphi}_{\U}(t)}  - 1 \right |^{2} dt \right \} \notag \\
&\quad +h_{n}^{-2\alpha} \left \{ \int_{\{ \psi_{X} = 0 \} \cap [-h_{n}^{-1},h_{n}^{-1}]} | \hat{\psi}_{YW}(t) |^{2} dt \right \} \left \{ \int_{-h_{n}^{-1}}^{h_{n}^{-1}} \left | \frac{\varphi_{\U}(t)}{\hat{\varphi}_{\U}(t)}  - 1 \right |^{2} dt \right \} \notag \\
&\quad + \int_{-h_{n}^{-1}}^{h_{n}^{-1}} \left |  \frac{\varphi_{\U}(t)}{\hat{\varphi}_{\U}(t)}  - 1 \right |^{2}| \psi_{X}(t)| dt. \label{eq: expansion}
\end{align}
We shall bound each term on the right hand side. Observe that 
\[
\int_{-h_{n}^{-1}}^{h_{n}^{-1}} \left | \frac{\varphi_{\U}(t)}{\hat{\varphi}_{\U}(t)}  - 1 \right |^{2} dt \leq O_{\Pr}(h_{n}^{-2\alpha}) \int_{-h_{n}^{-1}}^{h_{n}^{-1}} | \varphi_{\U}(t) - \hat{\varphi}_{\U}(t) |^{2} dt 
\]
and the integral on the right hand side is $O_{\Pr} \{ (mh_{n})^{-1} \}$ since
$\int_{-h_{n}^{-1}}^{h_{n}^{-1}}\Ep[ | \varphi_{\U}(t) - \hat{\varphi}_{\U}(t) |^{2} ]dt \leq m^{-1} \int_{-h_{n}^{-1}}^{h_{n}^{-1}} dt = 2(mh_{n})^{-1}$. 
Likewise, using the fact that $\psi_{X}$ is integrable, we have that the last term on the right hand side of (\ref{eq: expansion}) is $O_{\Pr}(h_{n}^{-2\alpha}m^{-1})$. 
For any $t \in \R$ with $\psi_{X}(t) \neq 0$, we have $\Ep [| \hat{\psi}_{YW}(t)/\psi_{YW}(t) - 1|^{2} ] \leq \Ep[Y^{2}]/\{n | \psi_{YW}(t) |^{2}\}$, 
so that 
\begin{align*}
&\Ep \left [ \int_{\{ \psi_{X} \neq 0 \} \cap [-h_{n}^{-1},h_{n}^{-1}]} \left | \frac{\hat{\psi}_{YW}(t)}{\psi_{YW}(t)} - 1 \right |^{2} |\psi_{X}(t)|^{2} dt \right ] \\
&\quad \lesssim n^{-1} \int_{-h_{n}^{-1}}^{h_{n}^{-1}} \frac{1}{| \varphi_{\U}(t)|^{2}} dt \lesssim h_{n}^{-2\alpha} (nh_{n})^{-1}. 
\end{align*}
Finally, for any $t \in \R$ with $\psi_{X}(t)=0$, we have $\psi_{YW}(t) = 0$, so that 
\[
\Ep\left [\int_{\{ \psi_{X} = 0 \} \cap [-h_{n}^{-1},h_{n}^{-1}]} | \hat{\psi}_{YW}(t) |^{2} dt \right ]  \lesssim (nh_{n})^{-1}. 
\]
Therefore, we have $\| \hat{\mu} - \hat{\mu}^{*} \|_{\R}^{2} = O_{\Pr}(h_{n}^{-4\alpha-2} n^{-1}m^{-1} + h_{n}^{-2\alpha} m^{-1} ) = o_{\Pr} (r_{n}^{2})$.

From Step 2 in the proof of Theorem 1 of \cite{KaSa16}, it follows that $\| \hat{f}_{X} - \hat{f}_{X}^{*} \|_{\R} = o_{\Pr}(r_{n})$, which in particular implies that $\| \hat{f}_{X} - f_{X} \|_{I} \leq \| \hat{f}_{X} - \hat{f}_{X}^{*} \|_{I} + \| \hat{f}_{X}^{*} - f_{X} \|_{I} = o_{P}(1)$ so that $\| 1/\hat{f}_{X} \|_{I} = O_{\Pr}(1)$. Furthermore, $\| \hat{\mu}^{*} \|_{I} \leq \| \Ep[ \hat{\mu}^{*} (\cdot ) ] \|_{I} + \| \hat{\mu}^{*}(\cdot) -\Ep[\hat{\mu}^{*}(\cdot) ] \|_{I} \lesssim \int_{\R}|\psi_{X}(t)|dt+ o_{\Pr}(1) =  O_{\Pr}(1)$. Therefore, 
\begin{align*}
\| \hat{g} - \hat{g}^{*} \|_{I} &\leq \| 1/\hat{f}_{X} \|_{I}\| \hat{\mu}  - \hat{\mu}^{*} \|_{I} + \| \hat{\mu}^{*} \|_{I} \| 1/\hat{f}_{X} - 1/\hat{f}_{X}^{*} \|_{I} \\
&\leq o_{\Pr} (r_{n}) + O_{\Pr}(1)\| \hat{f}_{X}^{*} - \hat{f}_{X} \|_{I} = o_{\Pr}(r_{n}).
\end{align*}

Now, observe that $\hat{g}^{*}(x) - g(x) = \frac{1}{\hat{f}_{X}^{*}(x)} \frac{1}{nh_{n}} \sum_{j=1}^{n} \{ Y_{j}-g(x) \}K_{n}((x-W_{j})/h_{n})$. 
Since $\| A_{n} \|_{I} = O(h_{n}^{\beta+1}) = o (h_{n}r_{n})$, we have 
\[
\hat{g}^{*}(x) - g(x) = \frac{1}{\hat{f}^{*}_{X}(x)} \underbrace{\frac{1}{nh_{n}} \sum_{j=1}^{n} \left [ \{ Y_{j}-g(x) \} K_{n}((x-W_{j})/h_{n}) - A_{n}(x) \right ]}_{=: (\star)} + o_{\Pr}(r_{n})
\]
uniformly in $x \in I$. Since $(\star)=\hat{\mu}^{*}(x) - \Ep[\hat{\mu}^{*}(x)] - g(x) \{ \hat{f}_{X}^{*}(x) - \Ep[\hat{f}_{X}^{*}(x)] \} 
=O_{\Pr} \{ h_{n}^{-\alpha}(nh_{n})^{-1/2} \sqrt{\log (1/h_{n})} \}$ uniformly in $x \in I$, and 
\[
\| 1/\hat{f}_{X}^{*} - 1/f_{X} \|_{I} \leq O_{\Pr}(1) \| \hat{f}_{X}^{*} - f_{X} \|_{I} = O_{\Pr} \{ h_{n}^{-\alpha}(nh_{n})^{-1/2} \sqrt{\log (1/h_{n})} \},
\]
we conclude that
\[
\hat{g}^{*}(x) - g(x) = \frac{1}{f_{X}(x)nh_{n}} \sum_{j=1}^{n} \left [ \{ Y_{j}-g(x) \} K_{n}((x-W_{j})/h_{n}) - A_{n}(x) \right ] + o_{\Pr}(r_{n})
\]
uniformly in $x \in I$. This leads to the desired result of Step 1. Furthermore, the derivation so far yields that
\begin{equation}
\| \hat{g} - g \|_{I} = O_{\Pr} \{ h_{n}^{-\alpha}(nh_{n})^{-1/2} \sqrt{\log (1/h_{n})} \}.
\label{eq: rate for g}
\end{equation}

\textbf{Step 2}. Recall the process $\mathsf{Z}_{n}^{*}$ defined in (\ref{eq: Zstar}). By Step 1 together with the fact that $\inf_{x \in I} s_{n}(x) \gtrsim h_{n}^{-\alpha+1/2}$, we have 
\begin{align*}
\hat{\mathsf{Z}}_{n}(x) = f_{X}(x)\sqrt{n}h_{n}(\hat{g}(x)-g(x))/s_{n}(x) 
= \mathsf{Z}_{n}^{*}(x) +o_{\Pr} \{ (\log  (1/h_{n}))^{-1/2} \}
\end{align*}
uniformly in $x \in I$. Recall the class of functions $\mF_{n}^{(2)}$ defined in (\ref{eq: function class}), and consider the empirical process indexed by $\mF_{n}^{(2)}$: 
$\nu_{n}(f) =\frac{1}{\sqrt{n}}\sum_{j=1}^{n} \{ f(Y_{j},W_{j}) - \Ep[ f(Y,W) ] \}$ for $f \in \mF_{n}^{(2)}$.
We apply Theorem 2.1 in \cite{ChChKa16} to approximate $\| \nu_{n} \|_{\mF_{n}^{(2)}} = \| \mathsf{Z}_{n}^{*} \|_{I}$ by the supremum of a Gaussian process. To this end, we shall verify the conditions in \cite{ChChKa16}. 

First, from the covering number bound for $\mF_{n}^{(2)}$ given in Lemma \ref{lem: entropy} and finiteness of the second moment of $F_{n}^{(2)}(Y,W)$, there exists a tight Gaussian process $G_{n}$ in $\ell^{\infty} (\mF_{n}^{(2)})$ with mean zero and the same covariance function as $\{ \nu_{n}(f) : f \in \mF_{n}^{(2)} \}$. Extend $\nu_{n}$ linearly  to $\mF_{n}^{(2)} \cup (-\mF_{n}^{(2)}) = \{ f,-f : f \in \mF_{n}^{(2)} \}$, and observe that $\| \nu_{n} \|_{\mF_{n}^{(2)}} = \sup_{f \in \mF_{n}^{(2)} \cup (-\mF_{n}^{(2)})} \nu_{n}(f)$.
Note that from Theorem 3.7.28 in \cite{GiNi16}, $G_{n}$ extends to the linear hull of $\mF_{n}^{(2)}$ in such a way that $G_{n}$ has linear sample paths, so that $\| G_{n} \|_{\mF_{n}^{(2)}} = \sup_{f \in \mF_{n}^{(2)} \cup (-\mF_{n}^{(2)})} G_{n}(f)$, and in addition $G_{n}$ has uniformly continuous paths on the symmetric convex hull of $\mF^{(2)}$ for the intrinsic $L^{2}$ pseudo-metric. It is not difficult to verify that the covering number of $\mF_{n}^{(2)} \cup (-\mF_{n}^{(2)})$ is at most twice that of $\mF_{n}^{(2)}$. In particular, $\{ G_{n}(f) : f \in \mF_{n}^{(2)} \cup (-\mF^{(2)}_{n}) \}$ is a tight Gaussian process in $\ell^{\infty}(\mF_{n}^{(2)} \cup (-\mF_{n}^{(2)}))$ with mean zero and the same covariance function as $\{ \nu_{n}(f) : f \in \mF_{n}^{(2)} \cup (-\mF_{n}^{(2)}) \}$. 

Next, from Lemma \ref{lem: moment bound}-(iii), $\Ep[ |Y K_{n}((x-W)/h_{n})|^{2+\ell}] \lesssim h_{n}^{-(2+\ell)\alpha+1}$ for $\ell=0,1,2$, 
so that $\sup_{f \in \mF_{n}^{(2)}} \Ep[|f(Y,W)|^{2+\ell}] \lesssim h_{n}^{-\ell/2}$ for $\ell=0,1,2$. 
Furthermore, $\Ep[ |F_{n}^{(2)}(Y,W)|^{4} ] \lesssim h_{n}^{-2} (\Ep[ Y^{4}]+\| g \|_{I}^{4}) \lesssim h_{n}^{-2}$. Therefore, applying Theorem 2.1 in \cite{ChChKa16} to $\mathcal{F}_{n}^{(2)} \cup (-\mF_{n}^{(2)})$ with $B(f) \equiv 0, q=4, A \lesssim 1, v \lesssim 1, b \lesssim h_{n}^{-1/2}, \sigma \lesssim 1$ and $\gamma \lesssim 1/\log n$, yields that there exists a random variable $V_{n}$ having the same distribution as $\| G_{n} \|_{\mF_{n}^{(2)}}$ such that
\[
\left| \| \nu_{n} \|_{\mF_{n}^{(2)}} - V_{n} \right | = O_{\Pr} \left \{ \frac{(\log n)^{5/4}}{n^{1/4} h_{n}^{1/2}} + \frac{\log n}{(nh_{n})^{1/6}} \right \} = o_{\Pr} \{ (\log (1/h_{n}))^{-1/2} \}. 
\]

Now, for $f_{n,x} (y,w) = \{y-g(x)\} K_{n}((x-w)/h_{n})/s_{n}(x)$, define
$\mathsf{Z}_{n}^{G}(x)  = G_{n} (f_{n,x})$ for $x \in I$, 
and observe that $\mathsf{Z}_{n}^{G}$ is a tight Gaussian process in $\ell^{\infty}(I)$ with mean zero and the same covariance function as $\mathsf{Z}_{n}^{*}$ such that $\| \mathsf{Z}_{n}^{G} \|_{I}$ has the same distribution as $V_{n}$. 
Since $| \| \hat{\mathsf{Z}}_{n} \|_{I} - V_{n}  | \leq | \| \hat{\mathsf{Z}}_{n} \|_{I} - \| \mathsf{Z}_{n}^{*} \|_{I} | + | \| \mathsf{Z}_{n}^{*} \|_{I} - V_{n} | = o_{\Pr} \{ (\log (1/h_{n}))^{-1/2} \}$, there exists a sequence $\Delta_{n} \downarrow 0$ such that $\Pr \{  | \| \hat{\mathsf{Z}}_{n} \|_{I} - V_{n} | > \Delta_{n} (\log (1/h_{n}))^{-1/2} \} \leq \Delta_{n}$ (which follows from the fact that the Ky Fan metric metrizes convergence in probability; see Theorem 9.2.2 in \cite{Du02}). Observe that for any $z \in \R$, 
\begin{align*}
\Pr \{ \| \hat{\mathsf{Z}}_{n} \|_{I} \leq z \} 
\leq \Pr \{ \| \mathsf{Z}_{n}^{G} \|_{I} \leq z+\Delta_{n}(\log (1/h_{n}))^{-1/2} \} +  \Delta_{n}.
\end{align*}
The anti-concentration inequality for the supremum of a Gaussian process (Lemma \ref{lem: anti-concentration}) then yields that 
\begin{align*}
&\Pr \{ \| \mathsf{Z}_{n}^{G} \|_{I} \leq z+\Delta_{n}(\log (1/h_{n}))^{-1/2} \} \\
&\quad \leq \Pr \{ \| \mathsf{Z}_{n}^{G} \|_{I} \leq z \} + 4 \Delta_{n}(\log (1/h_{n}))^{-1/2} \{ 1+\Ep[ \| \mathsf{Z}_{n}^{G} \|_{I} ] \}. 
\end{align*}
From the covering number bound for $\mF_{n}^{(2)}$ given in Lemma \ref{lem: entropy}, together with the facts that $\Ep[F_{n}^{(2)}(Y,W)^{2}] \lesssim h_{n}^{-1}$ and $\Var (f_{n,x}(Y,W)) = 1$ for all $x \in I$, Dudley's entropy integral bound (cf. van der Vaart and Wellner, 1996, Corollary 2.2.8) yields that 
\[
\Ep[ \| \mathsf{Z}_{n}^{G} \|_{I} ] = \Ep [ \| G_{n} \|_{\mF_{n}^{(2)}}] \lesssim \int_{0}^{1} \sqrt{1+\log (1/(\delta\sqrt{h_{n}}))} d\delta \lesssim \sqrt{\log (1/h_{n})},
\]
which implies that $
\Pr \{ \| \mathsf{Z}_{n}^{G} \|_{I} \leq z+\Delta_{n}(\log (1/h_{n}))^{-1/2} \} \leq \Pr \{ \| \mathsf{Z}_{n}^{G} \|_{I} \leq z \}  + o(1)$
uniformly in $z \in \R$. Likewise, we have $\Pr \{ \| \hat{\mathsf{Z}}_{n} \|_{I} \leq z \} \geq \Pr \{ \| \mathsf{Z}_{n}^{G} \|_{I} \leq z \} -o(1)$ uniformly in $z \in \R$. This completes the proof. 
\end{proof}

\subsection{Proof of Theorem \ref{thm: validity of MB}}

We first prove the following technical lemma. 
\begin{lemma}
$\| \hat{s}_{n}^{2}(\cdot)/s_{n}^{2}(\cdot) - 1 \|_{I} =o_{\Pr}\{ (\log (1/h_{n}) )^{-1} \}$. 
\end{lemma}
\begin{proof}
Observe that
\begin{align*}
&\{ Y_{j}-\hat{g}(x) \}^{2} \hat{K}_{n}^{2}((x-W_{j})/h_{n}) = \{ Y_{j} - g(x) \}^{2} K_{n}^{2}((x-W_{j})/h_{n}) \\
&\quad + \{ g(x) - \hat{g}(x) \}^{2} K_{n}^{2}((x-W_{j})/h_{n}) \\
&\quad + 2\{ g(x) - \hat{g}(x) \} \{ Y_{j}-g(x) \} K_{n}^{2}((x-W_{j})/h_{n}) \\
&\quad + \{ Y_{j}-\hat{g}(x) \}^{2} \{ \hat{K}_{n}^{2}((x-W_{j})/h_{n}) - K_{n}^{2}((x-W_{j})/h_{n}) \},
\end{align*}
so that 
\begin{align}
&\left \| \frac{1}{n} \sum_{j=1}^{n} \{ Y_{j}-\hat{g}(\cdot) \}^{2} \hat{K}_{n}^{2}((\cdot-W_{j})/h_{n}) - \frac{1}{n} \sum_{j=1}^{n} \{ Y_{j} - g(\cdot) \}^{2} K_{n}^{2}((\cdot-W_{j})/h_{n}) \right \|_{I} \notag \\
&\leq  O_{\Pr}(h_{n}^{-2\alpha}) \| \hat{g} - g \|_{I}^{2} + 2 \| \hat{g} - g \|_{I} \left \| \frac{1}{n} \sum_{j=1}^{n} \{ Y_{j}-g(\cdot) \} K_{n}^{2}((\cdot-W_{j})/h_{n}) \right \|_{I} \notag \\
&\quad + \frac{2}{n} \sum_{j=1}^{n} (Y_{j}^{2} +\| \hat{g} \|_{I}^{2})\| \hat{K}_{n}^{2} - K_{n}^{2} \|_{\R}. \label{eq: variance bound 1}
\end{align}
From (\ref{eq: rate for g}), the first term on the right hand side of (\ref{eq: variance bound 1}) is 
\[
O_{\Pr}\{ h_{n}^{-4\alpha} (nh_{n})^{-1}\log (1/h_{n}) \}.
\]
Since 
\begin{align*}
\| \hat{K}_{n} - K_{n} \|_{\R} &\lesssim \int_{\R} \left | \frac{1}{\hat{\varphi}_{\U}(t/h_{n})} - \frac{1}{\varphi_{\U}(t/h_{n})} \right | | \varphi_{K}(t) | dt \\
& \leq O_{\Pr}(h_{n}^{-2\alpha}) \int_{\R} | \hat{\varphi}_{\U}(t/h_{n}) - \varphi_{\U}(t/h_{n})| |\varphi_{K}(t)| dt \\
&= O_{\Pr}(h_{n}^{-2\alpha}m^{-1/2}), 
\end{align*}
we have that 
\[
\| \hat{K}_{n}^{2} - K_{n}^{2} \|_{\R} \leq \| \hat{K}_{n}-K_{n} \|_{\R} \| \hat{K}_{n}+K_{n} \|_{\R}=O_{\Pr} (h_{n}^{-3\alpha} m^{-1/2}),
\]
which implies that the last term on the right hand side on (\ref{eq: variance bound 1}) is \\
$O_{\Pr} (h_{n}^{-3\alpha} m^{-1/2})$. 
To bound the second term, observe first that, since $\Ep[ Y \mid W=w] f_{W}(w)= \left ( (g f_{X})*f_{\U} \right )(w)$ is bounded (in absolute value) by $\| gf_{X} \|_{\R}$, 
\begin{align*}
&\| \Ep[ \{ Y -g(\cdot) \} K_{n}^{2}((\cdot-W)/h_{n}) ] \|_{I} \\
&\quad \leq h_{n} (\| gf_{X} \|_{\R} + \| g \|_{I} \| f_{W} \|_{\R}) \int_{\R} K_{n}^{2}(w) dw \lesssim h_{n}^{-2\alpha+1}.
\end{align*}
Hence, 
\begin{align*}
&\left \| \frac{1}{n} \sum_{j=1}^{n} \{ Y_{j}-g(\cdot) \} K_{n}^{2}((\cdot-W_{j})/h_{n}) \right \|_{I} \leq \underbrace{\| \Ep[ \{ Y -g(\cdot) \} K_{n}^{2}((\cdot-W)/h_{n}) ] \|_{I}}_{=O(h_{n}^{-2\alpha+1})} \\
&\quad + \left \| \frac{1}{n} \sum_{j=1}^{n} \{ Y_{j}-g(\cdot) \} K_{n}^{2}((\cdot-W_{j})/h_{n})-\Ep[\{ Y-g(\cdot) \} K_{n}^{2}((\cdot-W)/h_{n}) ]   \right \|_{I}.
\end{align*}
The second term on the right hand side is identical to 
\[
\left \| \frac{1}{n} \sum_{j=1}^{n} \{ f(Y_{j},W_{j}) - \Ep[ f(Y,W) ] \} \right \|_{\mF_{n}^{(3)}}.
\]
In view of the covering number bound for $\mF_{n}^{(3)}$ given in Lemma \ref{lem: entropy}, together with Theorem 2.14.1 in \cite{vaWe96}, the expectation of the last term is 
\[
\lesssim n^{-1/2} \sqrt{\Ep[\{ F_{n}^{(3)}(Y,W)]\}^{2}]} \lesssim h_{n}^{-2\alpha}n^{-1/2}.
\]
Therefore, the right hand side on (\ref{eq: variance bound 1}) is 
\begin{align*}
O_{\Pr} \Bigg \{ &h_{n}^{-4\alpha} (nh_{n})^{-1} \log (1/h_{n}) +h_{n}^{-3\alpha} m^{-1/2} \\
&\quad + (nh_{n})^{-1/2} \sqrt{\log (1/h_{n})} (h_{n}^{-2\alpha+1} +  h_{n}^{-2\alpha} n^{-1/2})  \Bigg \},
\end{align*}
which is $o_{\Pr} \{ h_{n}^{-2\alpha+1} (\log (1/h_{n}))^{-1} \}$. Hence, because $\inf_{x \in I} s_{n}^{2}(x) \gtrsim h_{n}^{-2\alpha+1}$, we have 
\begin{align*}
&\frac{1}{s_{n}^{2}(x)n} \sum_{j=1}^{n} \{ Y_{j}-\hat{g}(x) \}^{2} \hat{K}_{n}^{2}((x-W_{j})/h_{n}) \\
&\quad = \frac{1}{s_{n}^{2}(x)n}\sum_{j=1}^{n} \{ Y_{j} - g(x) \}^{2} K_{n}^{2}((x-W_{j})/h_{n}) + o_{\Pr}\{ (\log (1/h_{n}))^{-1} \}
\end{align*}
uniformly in $x \in I$. Since $\| A_{n}^{2}(\cdot)/s_{n}^{2}(\cdot) \|_{I} =O(h_{n}^{2\alpha+2\beta+1})$, it remains to prove that
\begin{align*}
&\Bigg \| \frac{1}{s_{n}^{2}(\cdot)n}\sum_{j=1}^{n}\{ Y_{j} - g(\cdot) \}^{2} K_{n}^{2}((\cdot-W_{j})/h_{n}) \\
&\qquad - \Ep \left[\frac{1}{s_{n}^{2}(\cdot)} \{ Y-g(\cdot) \}^{2}K_{n}^{2}((\cdot-W))/h_{n}) \right] \Bigg \|_{I}  \\
&=\left \| \frac{1}{n} \sum_{j=1}^{n} \{ f(Y_{j},W_{j}) - \Ep[f(Y,W)] \} \right \|_{\mF_{n}^{(4)}} 
\end{align*}
is $o_{\Pr}\{ (\log (1/h_{n}))^{-1} \}$. In view of the covering number bound for $\mF_{n}^{(4)}$ given in Lemma \ref{lem: entropy}, together with Theorem 2.14.1 in \cite{vaWe96}, the expectation of the last term is 
\[
\lesssim n^{-1/2} \sqrt{\Ep[\{ F_{n}^{(4)}(Y,W)]\}^{2}]} \lesssim h_{n}^{-1}n^{-1/2} = o\{ (\log (1/h_{n}))^{-1} \}.
\]
This completes the proof. 
\end{proof}

\begin{proof}[Proof of Theorem \ref{thm: validity of MB}]
We divide the proof into several steps.

\textbf{Step 1}. Define 
\begin{align*}
\mathsf{Z}_{n}^{\xi} (x)= \frac{1}{s_{n}(x) \sqrt{n}} \sum_{j=1}^{n} \xi_{j} &\Big[ \{ Y_{j} - g(x) \} K_{n}((x-W_{j})/h_{n}) \\
&\qquad - n^{-1} {\textstyle \sum}_{j'=1}^{n} \{ Y_{j'} - g(x) \} K_{n}((x-W_{j'})/h_{n}) \Big ]
\end{align*}
for $x \in I$. We first prove that 
\[
\sup_{z \in \R} \left | \Pr \{ \| \mathsf{Z}_{n}^{\xi} \|_{I} \leq z \mid \mathcal{D}_{n} \} - \Pr \{ \| \mathsf{Z}_{n}^{G} \|_{I} \leq z \}  \right | \stackrel{\Pr}{\to} 0.
\]
To this end, we shall apply Theorem 2.2 in \cite{ChChKa16} to $\mF_{n}^{(2)} \cup (-\mF_{n}^{(2)})$. Let 
\[
\nu_{n}^{\xi} (f) = \frac{1}{\sqrt{n}} \sum_{j=1}^{n} \xi_{j} \{ f(Y_{j},W_{j}) - n^{-1} {\textstyle \sum}_{j'=1}^{n} f(Y_{j'},W_{j'}) \}, \ f \in \mF_{n}^{(2)}.
\]
Then applying Theorem 2.2 in \cite{ChChKa16} to $\mF_{n}^{(2)} \cup (-\mF_{n}^{(2)})$ with $B(f) \equiv 0, q=4,A \lesssim 1, v \lesssim 1, b \lesssim h_{n}^{-1/2}, \sigma \lesssim 1$ and $\gamma \lesssim 1/\log n$, yields that there exists a random variable $V_{n}^{\xi}$ of which the conditional distribution given $\mathcal{D}_{n}$ is identical to the distribution of $\| G_{n} \|_{\mF_{n}^{(2)}} (= \| \mathsf{Z}_{n}^{G} \|_{I})$, and such that 
\[
\left |  \| \nu_{n}^{\xi} \|_{\mF_{n}^{(2)}}- V_{n}^{\xi} \right | = O_{\Pr} \left \{ \frac{(\log n)^{9/4}}{n^{1/4}h_{n}^{1/2}} + \frac{(\log n)^{2}}{(nh_{n})^{1/4}} \right \} = o_{\Pr} \{ (\log (1/h_{n}) )^{-1/2} \},
\]
which shows that there exists a sequence $\Delta_{n} \downarrow 0$ such that
\[
\Pr \left \{ \left |  \| \nu_{n}^{\xi} \|_{\mF_{n}^{(2)}}- V_{n}^{\xi} \right | > \Delta_{n} (\log (1/h_{n}) )^{-1/2}  \mid \mathcal{D}_{n} \right \} \stackrel{\Pr}{\to} 0.
\]
Since $\| \nu_{n}^{\xi} \|_{\mF_{n}^{(2)}} = \| \mathsf{Z}_{n}^{\xi} \|_{I}$, we have 
\begin{align*}
\Pr \{ \| \mathsf{Z}_{n}^{\xi} \|_{I} \leq z \mid \mathcal{D}_{n} \} &\leq \Pr \{ V_{n}^{\xi} \leq z + \Delta_{n} (\log (1/h_{n}) )^{-1/2}  \mid \mathcal{D}_{n} \} + o_{\Pr} (1) \\
&=\Pr \{ \| \mathsf{Z}_{n}^{G} \|_{I}  \leq z + \Delta_{n} (\log (1/h_{n}) )^{-1/2}  \} + o_{\Pr} (1)
\end{align*}
uniformly in $z \in \R$, and the anti-concentration inequality for the supremum of a Gaussian process (Lemma \ref{lem: anti-concentration})  yields that
\[
\Pr \{ \| \mathsf{Z}_{n}^{G} \|_{I}  \leq z + \Delta_{n} (\log (1/h_{n}) )^{-1/2}  \}  \leq \Pr \{ \| \mathsf{Z}_{n}^{G} \|_{I} \leq z \} + o(1)
\]
uniformly in $z \in \R$. Likewise, we have $\Pr \{ \| \mathsf{Z}_{n}^{\xi} \|_{I} \leq z \mid \mathcal{D}_{n} \} \geq \Pr \{ \| \mathsf{Z}_{n}^{G} \|_{I} \leq z \} - o_{\Pr}(1)$
uniformly in $z \in \R$. 

\textbf{Step 2}. In view of the proof of Step 1, in order to prove the result (\ref{eq: validity of MB}), it is enough to prove that $\| \hat{\mathsf{Z}}_{n}^{\xi} - \mathsf{Z}_{n}^{\xi} \|_{I} = o_{\Pr} \{ (\log (1/h_{n}))^{-1/2} \}$. 
To this end, define
\[
\tilde{\mathsf{Z}}_{n}^{\xi} (x) = \frac{1}{s_{n}(x)\sqrt{n}} \sum_{j=1}^{n} \xi_{j} \{ Y_{j}-\hat{g}(x) \} \hat{K}_{n}((x-W_{j})/h_{n})
\]
for $x \in I$, and we first prove that 
\begin{equation}
\| \tilde{\mathsf{Z}}_{n}^{\xi} - \mathsf{Z}_{n}^{\xi} \|_{I} = o_{\Pr} \{ (\log (1/h_{n}))^{-1/2} \}.
\label{eq: intermediate}
\end{equation}
We begin with noting that 
\begin{align*}
&\frac{1}{n} \sum_{j=1}^{n} \{ Y_{j} - g(x) \}K_{n}((x-W_{j})/h_{n}) \\
&=h_{n} \{\hat{\mu}^{*}(x) - \Ep [ \hat{\mu}^{*}(x) ] \} - h_{n}g(x) \{ \hat{f}_{X}^{*} (x) - \Ep[ \hat{f}_{X}^{*}(x) ]  \} + A_{n}(x) \\
&=O_{\Pr} \{ h_{n}^{-\alpha+1} (nh_{n})^{-1/2}\sqrt{\log (1/h_{n})} \}
\end{align*}
uniformly in $x \in I$, so that it suffices to verify that 
\begin{align*}
&\Bigg \| \frac{1}{s_{n}(\cdot)\sqrt{n}}  \Bigg [\sum_{j=1}^{n} \xi_{j} \{Y_{j}-\hat{g}(\cdot) \} \hat{K}_{n}((\cdot-W_{j})/h_{n}) \\
&\quad -  \sum_{j=1}^{n} \xi_{j} \{Y_{j}-g(\cdot) \} \ K_{n}((\cdot-W_{j})/h_{n})\Bigg ] \Bigg \|_{I}
\end{align*}
is $o_{\Pr} \{ (\log (1/h_{n}))^{-1/2} \}$. Since $\| 1/s_{n} \|_{I} \lesssim h_{n}^{\alpha-1/2}$, the last term is 
\begin{align*}
&\lesssim h_{n}^{\alpha-1/2}n^{-1/2} \Bigg \{ \left \| \sum_{j=1}^{n} \xi_{j} Y_{j} \{ \hat{K}_{n}((\cdot-W_{j})/h_{n}) - K_{n}((\cdot-W_{j})/h_{n}) \} \right \|_{I} \\
&\quad + \| \hat{g} - g \|_{I} \left \| \sum_{j=1}^{n} \xi_{j}  \hat{K}_{n}((\cdot-W_{j})/h_{n}) \right \|_{I} \\
&\quad + \| g \|_{I} \left \| \sum_{j=1}^{n} \xi_{j}\{ \hat{K}_{n}((\cdot-W_{j})/h_{n}) - K_{n}((x-W_{j})/h_{n}) \} \right \|_{I} \Bigg \} \\
&=: h_{n}^{\alpha-1/2}n^{-1/2} \{ I_{n}+II_{n}+III_{n} \}.
\end{align*}
Step 2 in the proof of Theorem 2 in \cite{KaSa16} shows that $h_{n}^{\alpha-1/2}n^{-1/2}III_{n} = o_{\Pr} \{ (\log (1/h_{n}))^{-1/2} \}$. 
For the second term $II_{n}$, observe that 
\begin{align*}
II_{n} &\lesssim \| \hat{g} - g \|_{I} \int_{\R} \left | \sum_{j=1}^{n} \xi_{j} e^{itW_{j}/h_{n}} \right | \left | \frac{\varphi_{K}(t)}{\hat{\varphi}_{\U} (t/h_{n})} \right | dt \\
&\leq O_{\Pr}( h_{n}^{-\alpha}) \| \hat{g} - g \|_{I} \int_{-1}^{1} \left | \sum_{j=1}^{n} \xi_{j} e^{itY_{j}/h_{n}} \right | dt\\
& =O_{\Pr} \{h_{n}^{-2\alpha-1/2} \sqrt{\log (1/h_{n})} \},
\end{align*}
so that $h_{n}^{\alpha-1/2}n^{-1/2} II_{n} = O_{\Pr} \{ h_{n}^{-\alpha-1}n^{-1/2} \sqrt{\log (1/h_{n})} \}$, and the right hand side is $o_{\Pr} \{ (\log (1/h_{n}))^{-1/2} \}$ under our assumption.
For the first term $I_{n}$, observe that 
\begin{align*}
I_{n} &\lesssim \int_{\R} \left | \sum_{j=1}^{n} \xi_{j}Y_{j} e^{itW_{j}/h_{n}} \right | \left | \frac{1}{\hat{\varphi}_{\U}(t/h_{n})} - \frac{1}{\varphi_{\U}(t/h_{n})} \right | | \varphi_{K}(t) | dt \\
&\leq \left \{ \int_{-1}^{1} \left | \sum_{j=1}^{n} \xi_{j}Y_{j} e^{itW_{j}/h_{n}} \right |^{2}dt \right \}^{1/2} \left \{ \int_{-1}^{1} \left | \frac{1}{\hat{\varphi}_{\U}(t/h_{n})} - \frac{1}{\varphi_{\U}(t/h_{n})} \right |^{2} dt \right \}^{1/2} \\
&=O_{\Pr} (n^{1/2} h_{n}^{-2\alpha} m^{-1/2}),
\end{align*}
so that $h_{n}^{\alpha-1/2}n^{-1/2} I_{n} = o_{\Pr} \{ (\log (1/h_{n}))^{-1} \}$. Hence we have proved (\ref{eq: intermediate}). 

Note that the result of Step 1 and the fact that $\Ep[ \| \mathsf{Z}_{n}^{G} \|_{I}] = O(\sqrt{\log (1/h_{n})})$ imply that $\| \mathsf{Z}_{n}^{\xi} \|_{I} = O_{\Pr}(\sqrt{\log (1/h_{n})})$, which in turn implies that $\| \tilde{\mathsf{Z}}_{n}^{\xi} \|_{I} = O_{\Pr}(\sqrt{\log (1/h_{n})})$. Hence 
\[
\| \hat{\mathsf{Z}}_{n}^{\xi} - \tilde{\mathsf{Z}}_{n}^{\xi} \|_{I}  \leq \| s_{n}(\cdot)/\hat{s}_{n}(\cdot) - 1 \|_{I} \| \tilde{\mathsf{Z}}_{n}^{\xi} \|_{I} = o_{\Pr} \{ (\log (1/h_{n}))^{-1/2} \},
\]
which leads to (\ref{eq: validity of MB}). 

\textbf{Step 3}. We shall prove the last two assertions of the theorem. Observe that 
\begin{align*}
&\underbrace{\frac{\hat{f}_{X}(x) \sqrt{n}h_{n}(\hat{g}(x)-g(x))}{\hat{s}_{n}(x)}}_{=:\hat{\mathsf{Z}}_{n}^{\dagger}(x)} - \hat{\mathsf{Z}}_{n}(x) \\
&= \{ \hat{f}_{X}(x) - f_{X}(x) \}\frac{ \sqrt{n}h_{n}(\hat{g}(x)-g(x))}{\hat{s}_{n}(x)}  + \left \{ \frac{s_{n}(x)}{\hat{s}_{n}(x)} - 1 \right \} \hat{\mathsf{Z}}_{n}(x),
\end{align*}
and the right hand side is $o_{\Pr} \{ (\log (1/h_{n}))^{-1/2} \}$ uniformly in $x \in I$. To see this,
since $\| \hat{g} - g \|_{I} = O_{\Pr} \{ h_{n}^{-\alpha} (nh_{n})^{-1/2}\sqrt{\log (1/h_{n})}\}$ and $\| \hat{f}_{X} - f_{X} \|_{I} = O_{\Pr}  \{h_{n}^{-\alpha} (nh_{n})^{-1/2}\sqrt{\log (1/h_{n})} \}$ (which follows from Corollary 1 in \cite{KaSa16}), the right hand side on the above displayed equation is 
\begin{align*}
&O_{\Pr} \{ h_{n}^{-\alpha} (nh_{n})^{-1/2}\sqrt{\log (1/h_{n})}\} \times O_{\Pr}(\sqrt{\log (1/h_{n})}) \\
&\quad + o_{\Pr}\{ (\log (1/h_{n}))^{-1} \} \times O_{\Pr} (\sqrt{\log (1/h_{n})}) \\
&= o_{\Pr} \{ (\log (1/h_{n}))^{-1/2} \}
\end{align*}
uniformly in $x \in I$. 
Now, Theorem \ref{thm: Gaussian approximation}  and the anti-concentration inequality for the supremum of a Gaussian process (Lemma \ref{lem: anti-concentration}) yield that 
\[
\sup_{z \in \R} \left | \Pr \{ \| \hat{\mathsf{Z}}_{n}^{\dagger} \|_{I} \leq z \} - \Pr \{ \| \mathsf{Z}_{n}^{G} \|_{I} \leq z \} \right | \to 0.
\]
We are to show that $\Pr \{ \| \hat{\mathsf{Z}}_{n}^{\dagger}\|_{I} \leq \hat{c}_{n}(1-\tau) \} \to 1-\tau$.
From the result (\ref{eq: validity of MB}), there exists a sequence $\Delta_{n} \downarrow 0$ such that with probability greater than $1-\Delta_{n}$, 
\begin{equation}
\sup_{z \in \R} \left | \Pr \{ \| \hat{\mathsf{Z}}^{\xi} \|_{I} \leq z \mid \mathcal{D}_{n} \} - \Pr \{ \| \mathsf{Z}_{n}^{G} \|_{I} \leq z \} \right | \leq \Delta_{n},
\label{eq: event}
\end{equation}
and let $\mathcal{E}_{n}$ be the event that (\ref{eq: event}) holds. Taking $\Delta_{n} \downarrow 0$ more slowly if necessary, we have that $\sup_{z \in \R}  | \Pr \{ \| \hat{\mathsf{Z}}_{n}^{\dagger} \|_{I} \leq z \} - \Pr \{ \| \mathsf{Z}_{n}^{G} \|_{I} \leq z \} | \leq \Delta_{n}$. 
Recall that $c_{n}^{G}(1-\tau)$ is the $(1-\tau)$-quantile of $\| \mathsf{Z}_{n}^{G} \|_{I}$, and observe that on the event $\mathcal{E}_{n}$, 
\[
\Pr \{ \| \hat{\mathsf{Z}}_{n}^{\xi} \|_{I} \leq c_{n}^{G}(1-\tau+\Delta_{n}) \} \geq \Pr \{ \| \mathsf{Z}_{n}^{G} \|_{I} \leq c_{n}^{G}(1-\tau+\Delta_{n}) \} - \Delta_{n} = 1-\tau,
\]
where the last equality holds since the distribution function of $\| \mathsf{Z}_{n}^{G} \|_{I}$ is continuous (which follows from Lemma \ref{lem: anti-concentration}). Hence on the event $\mathcal{E}_{n}$, it holds that $\hat{c}_{n}(1-\tau) \leq c_{n}^{G}(1-\tau+\Delta_{n})$, so that 
\begin{align*}
&\Pr \{ \| \hat{\mathsf{Z}}_{n}^{\dagger} \|_{I} \leq \hat{c}_{n}(1-\tau) \} \leq \Pr \{ \| \hat{\mathsf{Z}}_{n}^{\dagger} \|_{I} \leq c_{n}^{G}(1-\tau+\Delta_{n}) \} + \Delta_{n} \\
&\quad \leq \Pr \{ \| \mathsf{Z}_{n}^{G} \|_{I} \leq c_{n}^{G}(1-\tau+\Delta_{n}) \} + 2\Delta_{n} = 1-\tau + 3\Delta_{n}. 
\end{align*}
Likewise, we have $\Pr \{ \| \hat{\mathsf{Z}}_{n}^{\dagger} \|_{I} \leq \hat{c}_{n}(1-\tau) \} \geq 1-\tau-3\Delta_{n}$, which shows that $\Pr \{ \| \hat{\mathsf{Z}}_{n}^{\dagger}\|_{I} \leq \hat{c}_{n}(1-\tau) \} \to 1-\tau$ and thus (\ref{eq: validity of MB 2}) holds. 

Finally, the Borell-Sudakov-Tsirelson inequality \citep[][Lemma A.2.2]{vaWe96} yields that 
\[
c_{n}^{G}(1-\tau+\Delta_{n}) \leq \Ep[ \| \mathsf{Z}_{n}^{G} \|_{I}] + \sqrt{2\log (1/(\tau-\Delta_{n}))} \lesssim \sqrt{\log (1/h_{n})},
\]
which implies that $\hat{c}_{n}(1-\tau) = O_{\Pr}(\sqrt{\log (1/h_{n})})$. Furthermore,
\[
\sup_{x \in I} \hat{s}_{n}(x) \leq \sup_{x \in I} s_{n}(x) \cdot \sup_{x \in I} \frac{\hat{s}_{n}(x)}{s_{n}(x)} = O_{\Pr}(h_{n}^{-\alpha+1/2}).
\]
Therefore, the supremum width of the band $\hat{\mathcal{C}}_{1-\tau}$ is
\[
2\sup_{x \in I} \frac{\hat{s}_{n}(x)}{\sqrt{n}h_{n}} \hat{c}_{n}(1-\tau) = O_{\Pr}\left \{ h_{n}^{-\alpha}(nh_{n})^{-1/2}\sqrt{\log (1/h_{n})} \right \}.
\]
This completes the proof. 
\end{proof}

\subsection{Proof of Theorem \ref{thm: validity of MB part 2}}

Let $\hat{\varphi}_{WZ}(t) = n^{-1}\sum_{j=1}^{n} e^{i(t_{1}W_{j} + t_{-1}^{T}Z_{j})}$ and $\hat{\psi}_{YWZ} (t) = n^{-1} \sum_{j=1}^{n} Y_{j} e^{i(t_{1}W_{j}+t_{-1}^{T}Z_{j})}$. 
We first point out that $\hat{\mu}(x,z)$ and $\hat{f}_{XZ}(x,z)$ can be expressed as 
\[
\begin{split}
\hat{\mu}(x,z) = \frac{1}{(2\pi)^{d+1}} \int_{\R^{d+1}} e^{-i(t_{1}x + t_{-1}^{T}z)} \hat{\psi}_{YWZ}(t) \frac{\varphi_{K}(t_{1}h_{n}) \varphi_{L}(t_{-1}h_{n})}{\hat{\varphi}_{U}(t_{1})} dt, \\
\hat{f}_{XZ}(x,z) = \frac{1}{(2\pi)^{d+1}} \int_{\R^{d+1}} e^{-i(t_{1}x + t_{-1}^{T}z)}\hat{\varphi}_{WZ}(t) \frac{\varphi_{K}(t_{1}h_{n}) \varphi_{L}(t_{-1}h_{n})}{\hat{\varphi}_{U}(t_{1})} dt. 
\end{split}
\]
Given these expressions, the proof of the theorem is almost analogous to the proofs of Theorems \ref{thm: Gaussian approximation} and \ref{thm: validity of MB}. Essentially we only have to care about the fact the Jacobian of the change of variables $(x,w) \to (h_{n}x,h_{n}z)$ is $h_{n}^{d+1}$. 
For instance, the conclusions of Lemma \ref{lem: moment bound} change to $\| A_{n} \|_{I \times J} = O(h_{n}^{\beta+d+1})$, $\inf_{(x,z) \in I \times J} s_{n}^{2}(x,z) \gtrsim h_{n}^{-2\alpha+d+1}$ with $s_{n}^{2}(x,z) = \Var (\{ Y - g(x,z) \}K_{n}((x-W)/h_{n})L((z-Z)/h_{n}))$, and $\sup_{(x,z) \in \R \times \R^{d}} \Ep[|Y K_{n}((x-W)/h_{n})L((z-Z)/h_{n})|^{2+\ell}] \lesssim h_{n}^{-(2+\ell) \alpha + d+1}$ for $\ell=0,1,2$. 
Let $\hat{\mu}^{*}(x,z)$ and $\hat{f}_{XZ}^{*}(x,z)$ be defined by replacing $\hat{K}_{n}$ in $\hat{\mu}(x,z)$ and $\hat{f}_{XZ}(x,z)$, respectively, to $K_{n}$. Then the conclusions of Lemma \ref{lem: rate} change to 
\begin{align*}
&\| \hat{f}_{XZ}^{*} - \Ep[ \hat{f}_{XZ}^{*} ]  \|_{\R \times \R^{d}} = O_{\Pr} \{ h_{n}^{-\alpha} (nh_{n}^{d+1})^{-1/2} \sqrt{\log (1/h_{n})} \}, \\
&\|\Ep[ \hat{f}_{XZ}^{*} ] - f_{XZ}  \|_{\R \times \R^{d}} = O(h_{n}^{\beta}) = o\{ h_{n}^{-\alpha} (nh_{n}^{d+1}\log (1/h_{n}))^{-1/2} \},
\end{align*}
and $\| \hat{\mu}^{*} - \Ep[\hat{\mu}^{*} ] \|_{\R \times \R^{d}} = O_{\Pr} \{ h_{n}^{-\alpha} (nh_{n}^{d+1})^{-1/2} \sqrt{\log (1/h_{n})} \}$. Using such estimates, we can show under our assumption that
\[
\hat{g}(x,z)-g(x,z) =  \frac{1}{f_{XZ}(x,z) nh_{n}^{d+1}} \sum_{j=1}^{n} \left [ \{ Y_{j} - g(x,z) \} K_{n}((x-W_{j})/h_{n}) L((z-Z_{j})/h_{n}) - A_{n}(x,z) \right] + o_{\Pr}(r_{n})
\]
with $A_{n}(x,z) = \Ep[\{ Y - g(x,z) \} K_{n}((x-W)/h_{n}) L((z-Z)/h_{n})]$ and $r_{n} = h_{n}^{-\alpha} \{ nh_{n}^{d+1} \log (1/h_{n}) \}^{-1/2}$, where $o_{\Pr}(r_{n})$ is uniform in $(x,z) \in I \times J$. 
Hence, arguing as in Step 2 of the proof of Theorem \ref{thm: Gaussian approximation}, for the process 
\[
\hat{\mathsf{Z}}_{n}(x,z) = f_{XZ}(x,z) \sqrt{n} h_{n}(\hat{g}(x,z) - g(x,z))/s_{n}(x,z),
\]
we obtain the Gaussian approximation: there exists a tight Gaussian process $\{ \mathsf{Z}_{n}^{G}(x,z) ; (x,z) \in I \times J \}$ in $\ell^{\infty}(I \times J)$ with mean zero and covariance function
\[
\begin{split}
&\Ep[\mathsf{Z}_{n}^{G}(x,z) \mathsf{Z}_{n}^{G}(\overline{x},\overline{z})] \\
&=\frac{\Cov \left (\{ Y - g(x,z) \} K_{n}((x-W)/h_{n}) L((z-Z)/h_{n}), \{ Y - g(\overline{x},\overline{z}) \} K_{n}((\overline{x}-W)/h_{n}) L((\overline{z}-Z)/h_{n}) \right)}{s_{n}(x,z) s_{n}(\overline{x},\overline{z})}
\end{split}
\]
 such that 
\[
\sup_{y \in \R} \left| \Pr \{ \| \hat{\mathsf{Z}}_{n} \|_{I \times J} \le y \} - \Pr \{ \| \mathsf{Z}_{n}^{G} \|_{I \times J} \le y \} \right | \to 0.
\]
Likewise, arguing as in the proof of Theorem \ref{thm: validity of MB}, we can show that 
\[
\sup_{y \in \R} \left | \Pr \left \{ \| \hat{\mathsf{Z}}_{n}^{\xi} \|_{I \times J} \le y \mid \mathcal{D}_{n} \right \} - \Pr \{ \| \mathsf{Z}_{n}^{G} \|_{I \times J} \le y \} \right | \stackrel{\Pr}{\to} 0.
\]
This leads to the validity of the confidence band $\hat{\mathcal{C}}_{1-\tau}$. 
The supremum width of the band follows from the fact that $\sup_{(x,z) \in I \times J} \hat{s}_{n}(x,z) = O_{\Pr}(h_{n}^{-\alpha+(d+1)/2})$ and the $(1-\tau)$-quantile of $\| \mathsf{Z}_{n}^{G} \|_{I \times J}$ is $O(\sqrt{\log (1/h_{n})})$. \qed

\subsection{Proof of Theorem \ref{thm: cdf}}
\label{sec: validity of MB}

The proof is completely analogous to those of Theorems \ref{thm: Gaussian approximation} and  \ref{thm: validity of MB}, given the facts that $g(y,x) = \Ep[ 1(Y \leq y) \mid X=x]$ and the function class $\{ 1(\cdot \leq y) : y \in J \}$ is a VC class. Hence, we omit the details for brevity. \qed

\subsection{Validity of the empirical bootstrap confidence band (\ref{eq: EB band})}

In this section, we show that under the same assumption as Theorem \ref{thm: validity of MB}, the empirical bootstrap confidence band (\ref{eq: EB band}) is asymptotically valid, i.e.,
\[
\Pr \{ g(x) \in \hat{\mathcal{C}}^{EB}_{1-\tau}(x) \ \forall x \in I \} =1-\tau + o(1), 
\]
and the supremum width of the band contracts at the rate $O_{\Pr}\{ h_{n}^{-\alpha} (nh_{n})^{-1/2} \sqrt{\log (1/h_{n})} \}$. The proof is analogous to that of Theorem \ref{thm: validity of MB} and so we only point out required modifications. 
In view of the proof of Theorem \ref{thm: validity of MB}, it suffices to prove that 
\[
\sup_{z \in \R} \left | \Pr \left \{ \| \hat{\mathsf{Z}}_{n}^{EB} \|_{I} \le z \mid \mathcal{D}_{n} \right \} - \Pr \{ \| \mathsf{Z}_{n}^{G} \|_{I} \le z \} \right | \stackrel{\Pr}{\to} 0.
\]
As in Step 1 of the proof of Theorem \ref{thm: validity of MB}, consider first the infeasible empirical bootstrap process
\[
\mathsf{Z}_{n}^{EB}(x) = \frac{1}{s_{n}(x)\sqrt{n}} \sum_{j=1}^{n} \left [ \{ Y_{j}^{*} - g(x) \} K_{n}((x-W_{j}^{*})/h_{n}) - n^{-1} \sum_{j'=1}^{n} \{ Y_{j'}  - g(x) \} K_{n}((x-W_{j'})/h_{n}) \right ],
\]
which can be rewritten as 
\[
\mathsf{Z}_{n}^{EB}(x) = \frac{1}{s_{n}(x)\sqrt{n}}  \sum_{j=1}^{n} (M_{nj}-1) \{ Y_{j} - g(x) \} K_{n}((x-W_{j})/h_{n}), 
\]
where $M_{nj}$ is the number of times that $(Y_{j},W_{j})$ is ``redrawn'' in the bootstrap sample, and $(M_{n1},\dots,M_{nn})$ is multinomially distributed with parameters $n$ and (probabilities) $1/n,\dots,1/n$: cf. Chapter 3.6 of \cite{vaWe96}. 
Then, instead of  Theorem 2.2 in \cite{ChChKa16}, we can apply Theorem 2.3 in \cite{ChChKa16} to deduce that 
\[
\sup_{z \in \R} \left | \Pr \left \{ \| \mathsf{Z}_{n}^{EB} \|_{I} \le z \mid \mathcal{D}_{n} \right \} - \Pr \{ \| \mathsf{Z}_{n}^{G} \|_{I} \le z \} \right | \stackrel{\Pr}{\to} 0.
\]
The rest is to verify that $\| \hat{\mathsf{Z}}_{n}^{EB} - \mathsf{Z}_{n}^{EB} \|_{I} = o_{\Pr}\{ (\log (1/h_{n}))^{-1/2} \}$, but this is quite analogous to Step 2 of the proof of Theorem \ref{thm: validity of MB} given the expression
\[
\hat{\mathsf{Z}}_{n}^{EB} (x) =  \frac{1}{\hat{s}_{n}(x)\sqrt{n}}  \sum_{j=1}^{n} (M_{nj}-1) \{ Y_{j} - g(x) \} \hat{K}_{n}((x-W_{j})/h_{n}).
\]
(Recall that $\sum_{j=1}^{n} \{ Y_{j} - g(x) \} \hat{K}_{n}((x-W_{j})/h_{n}) = 0$.) Hence, we omit the details. \qed





\section{Additional simulation results}
\label{sec: additional simulation}

The main text of the paper presents only a subset of simulation results.
In this section, we present the remaining simulation results which are not presented in the main text.
Table \ref{tab:simulation_results_additional} in this appendix show simulation results for (A) $g(x) = x$, (B) $g(x)=x^2$, and (C) $g(x)=\cos(x)$.
Each table contains results for each of Model 1 and Model 2, for each of the two sample sizes, $n=200$ and $400$, and for each of the two error variance ratios, $1/4$ and $1/3$. 
Simulated coverage probabilities are reported for each of the two nominal coverage probabilities, $0.900$ and $0.950$.

\begin{table}
	\centering	
		\scalebox{1}{
		\begin{tabular}{ccccccccc}
		\hline\hline
		\multicolumn{3}{l}{(A) Regression: $g(x) = x$} && \multicolumn{2}{c}{Error Variance} && \multicolumn{2}{c}{Error Variance}\\
			& Nominal & Sample && \multicolumn{2}{c}{=1/4 (25\%)} && \multicolumn{2}{c}{=1/3 (33\%)}\\
		\cline{5-6}\cline{8-9}
			Model & Probability & Size ($n$) && Coverage & Length && Coverage & Length\\
		\hline
			1 & 0.900            
			 &   200  && 0.898 & 1.996 && 0.899 & 1.887\\
			&&   400  && 0.932 & 1.503 && 0.927 & 1.339\\
		\cline{2-9}
			1 & 0.950
			 &   200  && 0.932 & 2.215 && 0.939 & 2.092\\
			&&   400  && 0.967 & 1.661 && 0.955 & 1.489\\
		\hline
			2 & 0.900            
			 &   200  && 0.919 & 1.946 && 0.891 & 1.823\\
			&&   400  && 0.934 & 1.364 && 0.900 & 1.290\\
		\cline{2-9}
			2 & 0.950
			 &   200  && 0.951 & 2.156 && 0.929 & 2.021\\
			&&   400  && 0.961 & 1.516 && 0.947 & 1.436\\
		\hline\hline
		\\
		\hline\hline
		\multicolumn{3}{l}{(B) Regression: $g(x) = x^2$} && \multicolumn{2}{c}{Error Variance} && \multicolumn{2}{c}{Error Variance}\\
			& Nominal & Sample && \multicolumn{2}{c}{=1/4 (25\%)} && \multicolumn{2}{c}{=1/3 (33\%)}\\
		\cline{5-6}\cline{8-9}
			Model & Probability & Size ($n$) && Coverage & Length && Coverage & Length\\
		\hline
			1 & 0.900            
			 &   200  && 0.881 & 3.601 && 0.897 & 2.950\\
			&&   400  && 0.918 & 3.127 && 0.919 & 2.540\\
		\cline{2-9}
			1 & 0.950
			 &   200  && 0.919 & 3.976 && 0.936 & 3.277\\
			&&   400  && 0.954 & 3.464 && 0.954 & 2.828\\

		\hline
			2 & 0.900            
			 &   200  && 0.915 & 3.721 && 0.919 & 3.061\\
			&&   400  && 0.922 & 2.827 && 0.904 & 2.289\\
		\cline{2-9}
			2 & 0.950
			 &   200  && 0.947 & 4.134 && 0.946 & 3.406\\
			&&   400  && 0.956 & 3.128 && 0.943 & 2.551\\
		\hline\hline
		\\
		\hline\hline
		\multicolumn{3}{l}{(C) Regression: $g(x) = \cos(x)$} && \multicolumn{2}{c}{Error Variance} && \multicolumn{2}{c}{Error Variance}\\
			& Nominal & Sample && \multicolumn{2}{c}{=1/4 (25\%)} && \multicolumn{2}{c}{=1/3 (33\%)}\\
		\cline{5-6}\cline{8-9}
			Model & Probability & Size ($n$) && Coverage & Length && Coverage & Length\\
		\hline
			1 & 0.900            
			 &   200  && 0.882 & 1.588 && 0.891 & 1.470\\
			&&   400  && 0.907 & 1.242 && 0.886 & 1.088\\
		\cline{2-9}
			1 & 0.950
			 &   200  && 0.926 & 1.768 && 0.936 & 1.631\\
			&&   400  && 0.940 & 1.367 && 0.932 & 1.207\\
		\hline
			2 & 0.900            
			 &   200  && 0.901 & 1.587 && 0.888 & 1.451\\
			&&   400  && 0.897 & 1.169 && 0.887 & 1.020\\
		\cline{2-9}
			2 & 0.950
			 &   200  && 0.945 & 1.765 && 0.944 & 1.612\\
			&&   400  && 0.933 & 1.299 && 0.933 & 1.137\\
		\hline\hline
		\end{tabular}
		}
\medskip
\caption{{\small Simulated uniform coverage probabilities of (A) $g(x) = x$, (B) $g(x)=x^2$, and (C) $g(x)=\cos(x)$ by estimated confidence bands in $I=[-\sigma_X,\sigma_X]$ under normally distributed $X$ and Laplace distributed $\U$. Also reported are the medians of the average band lengths on $I$. The simulated probabilities and lengths are computed for each of the two nominal coverage probabilities, 90\% and 95\%, based on 1,000 Monte Carlo iterations.}}
	\label{tab:simulation_results_additional}
\end{table}

Figure \ref{fig:sim_model2_x3_sinx} in this appendix illustrates realizations of estimates and confidence bands for the functions $g(x)=x^3$ and $g(x)=\sin(x)$ in Model 2 for error variance ratios of $1/4$ and $1/3$.
Figure \ref{fig:sim_model1_x1_x2} in this  appendix illustrates realizations of estimates and confidence bands for the functions $g(x)=x$ and $g(x)=x^2$ in Model 1 for error variance ratios of $1/4$ and $1/3$.
Similarly, Figure \ref{fig:sim_model2_x1_x2} in this  appendix illustrates realizations of estimates and confidence bands for the functions $g(x)=x$ and $g(x)=x^2$ in Model 2 for error variance ratios of $1/4$ and $1/3$.
Finally, Figure \ref{fig:sim_model1_model2_cosx} in this  appendix illustrates realizations of estimates and confidence bands for the function $g(x)=\cos(x)$ in Model 1 and Model 2 for error variance ratios of $1/4$ and $1/3$.

\begin{figure}
	\centering
		\begin{tabular}{ccc}
		&
		$n=200$&
		$n=400$\\
		\begin{minipage}{0.2\textwidth}$g(x)=x^3$\bigskip\\Model 2\bigskip\\EV=1/4 (25\%)\end{minipage} &
		\begin{minipage}{0.3\textwidth}\includegraphics[width=1\textwidth]{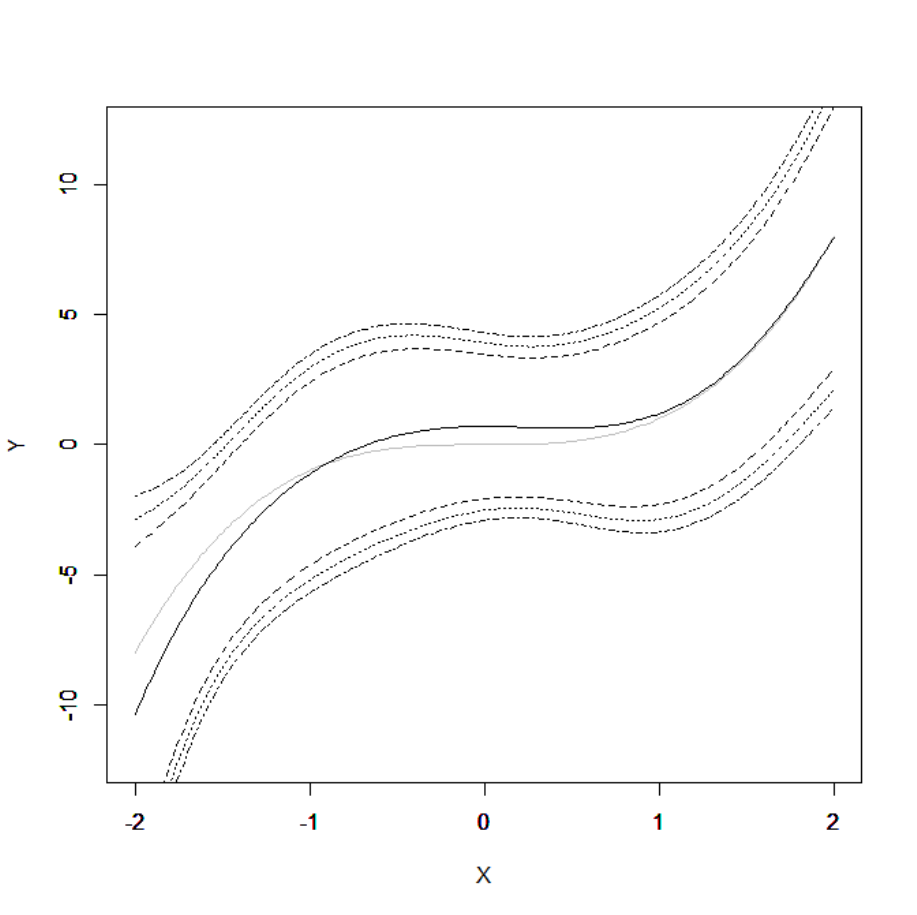}\end{minipage} &
		\begin{minipage}{0.3\textwidth}\includegraphics[width=1\textwidth]{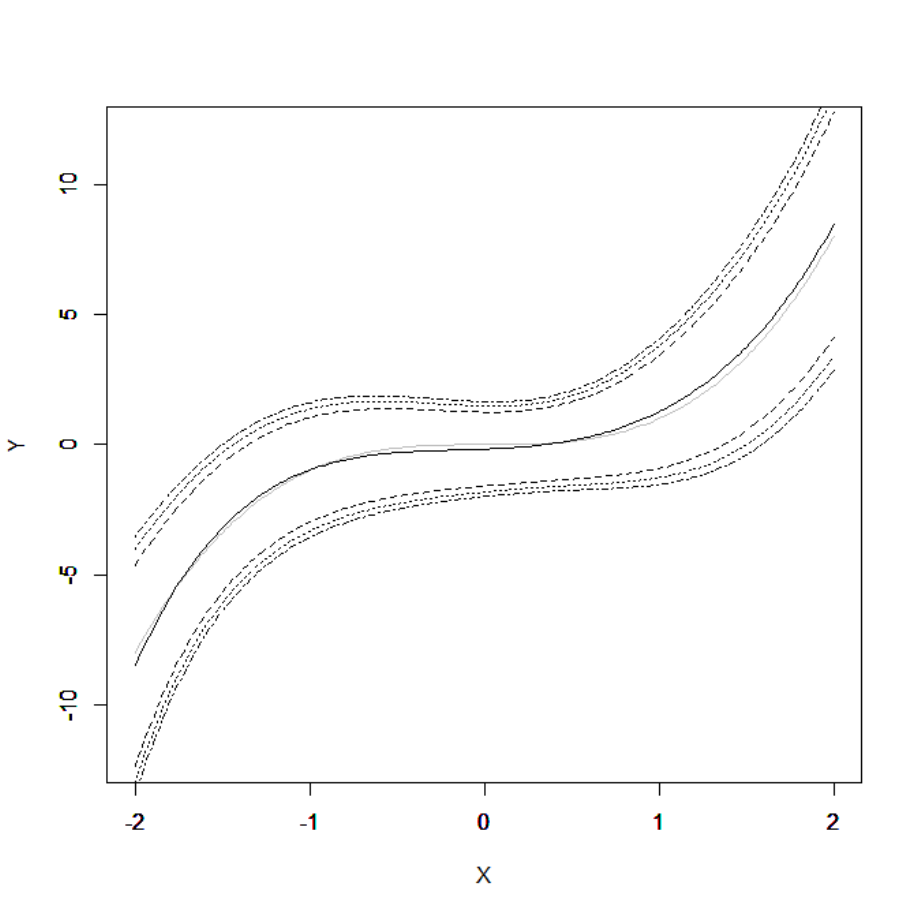}\end{minipage} \\
		\begin{minipage}{0.2\textwidth}$g(x)=x^3$\bigskip\\Model 2\bigskip\\EV=1/3 (33\%)\end{minipage} &
		\begin{minipage}{0.3\textwidth}\includegraphics[width=1\textwidth]{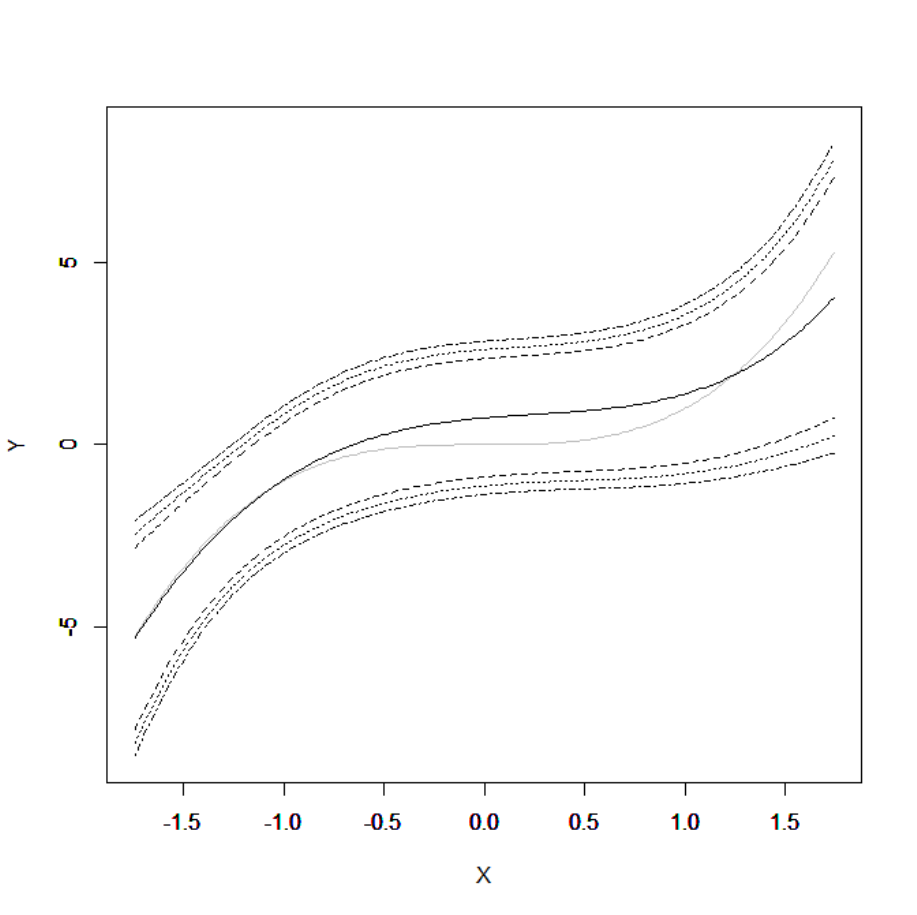}\end{minipage} &
		\begin{minipage}{0.3\textwidth}\includegraphics[width=1\textwidth]{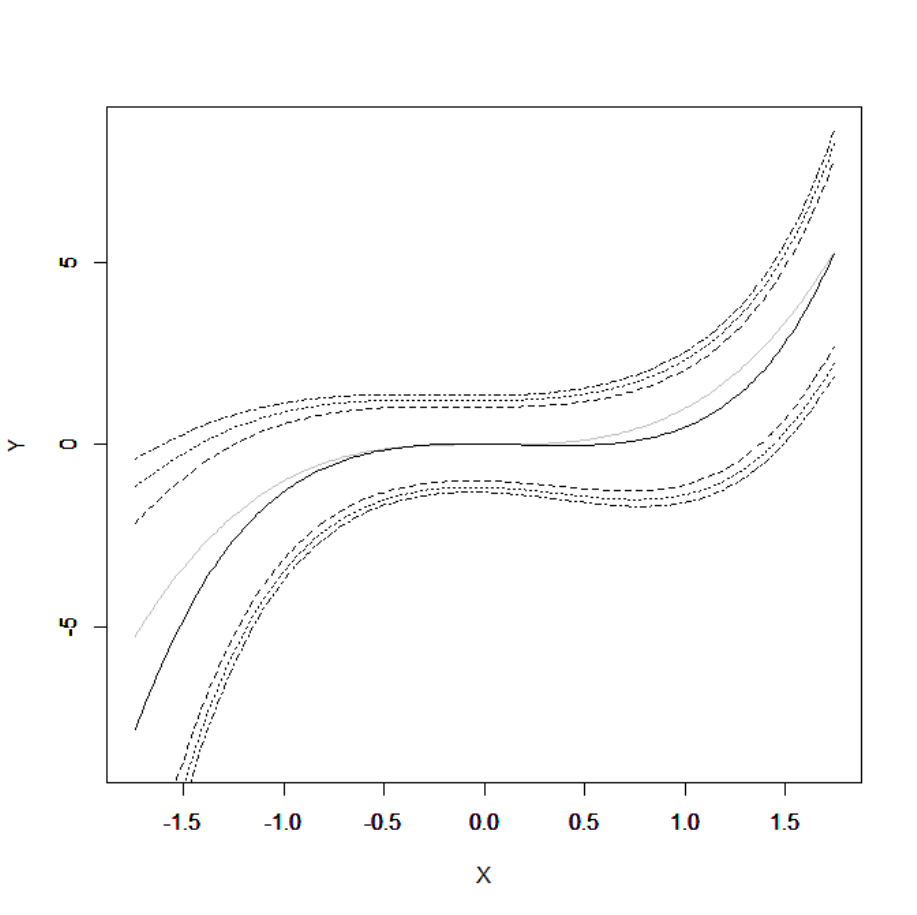}\end{minipage} \\
		\begin{minipage}{0.2\textwidth}$g(x)=\sin(x)$\bigskip\\Model 2\bigskip\\EV=1/4 (25\%)\end{minipage} &
		\begin{minipage}{0.3\textwidth}\includegraphics[width=1\textwidth]{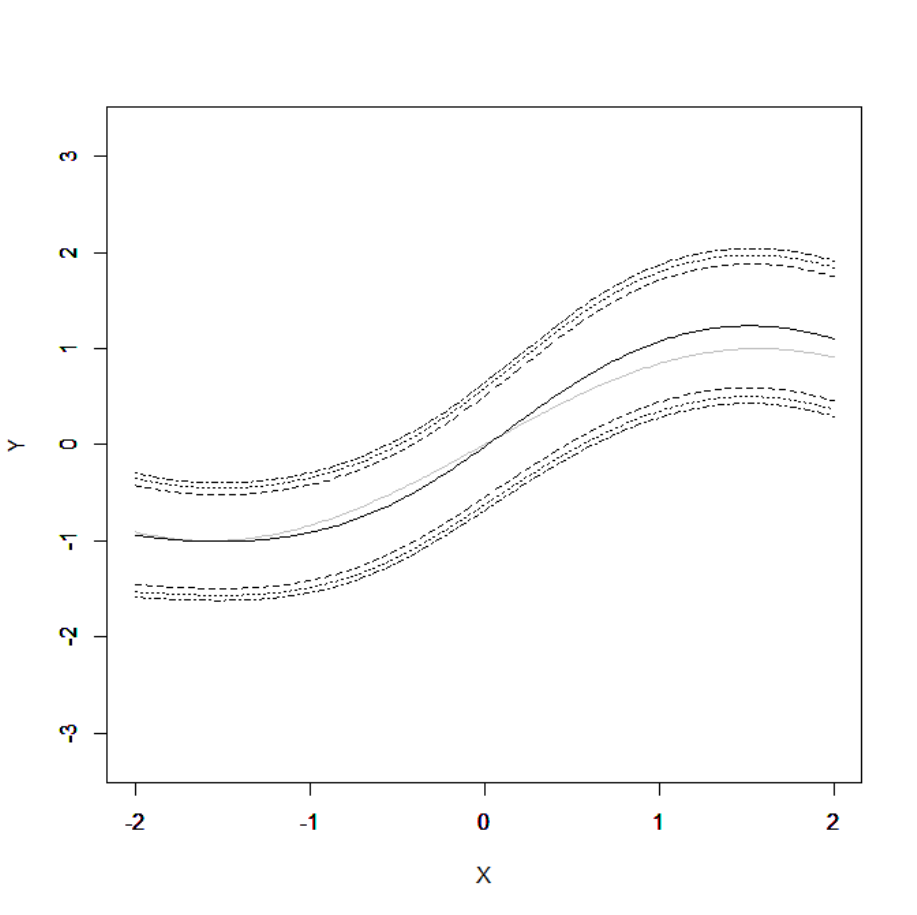}\end{minipage} &
		\begin{minipage}{0.3\textwidth}\includegraphics[width=1\textwidth]{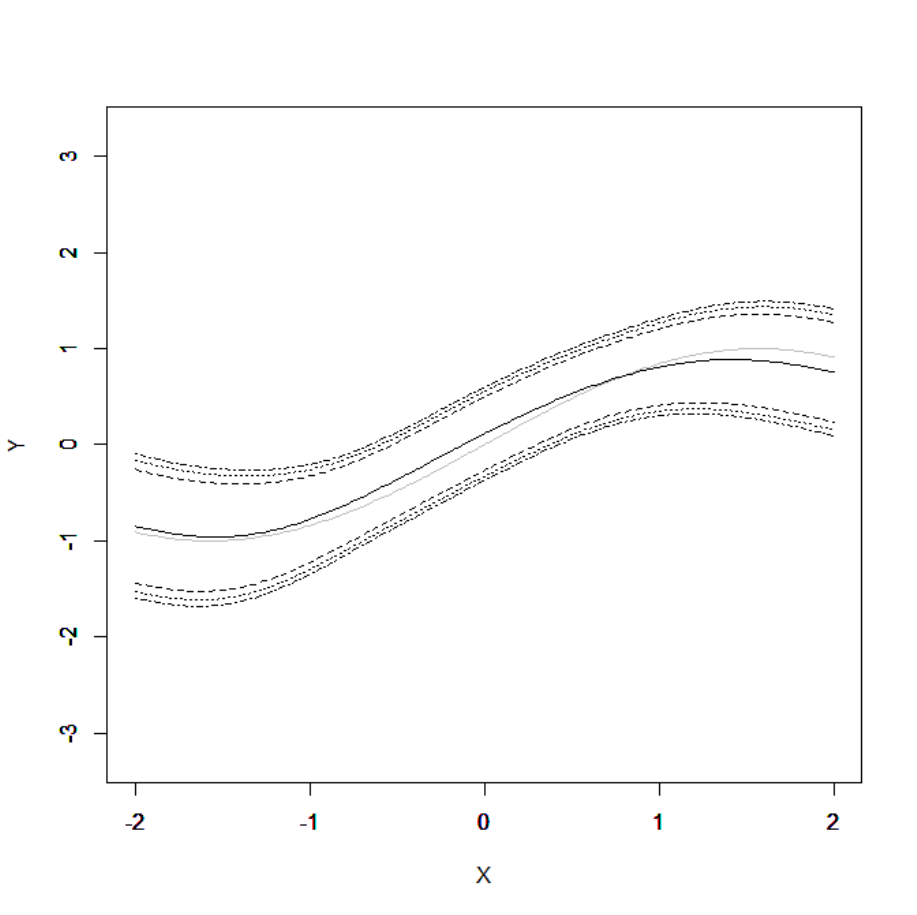}\end{minipage} \\
		\begin{minipage}{0.2\textwidth}$g(x)=\sin(x)$\bigskip\\Model 2\bigskip\\EV=1/3 (33\%)\end{minipage} &
		\begin{minipage}{0.3\textwidth}\includegraphics[width=1\textwidth]{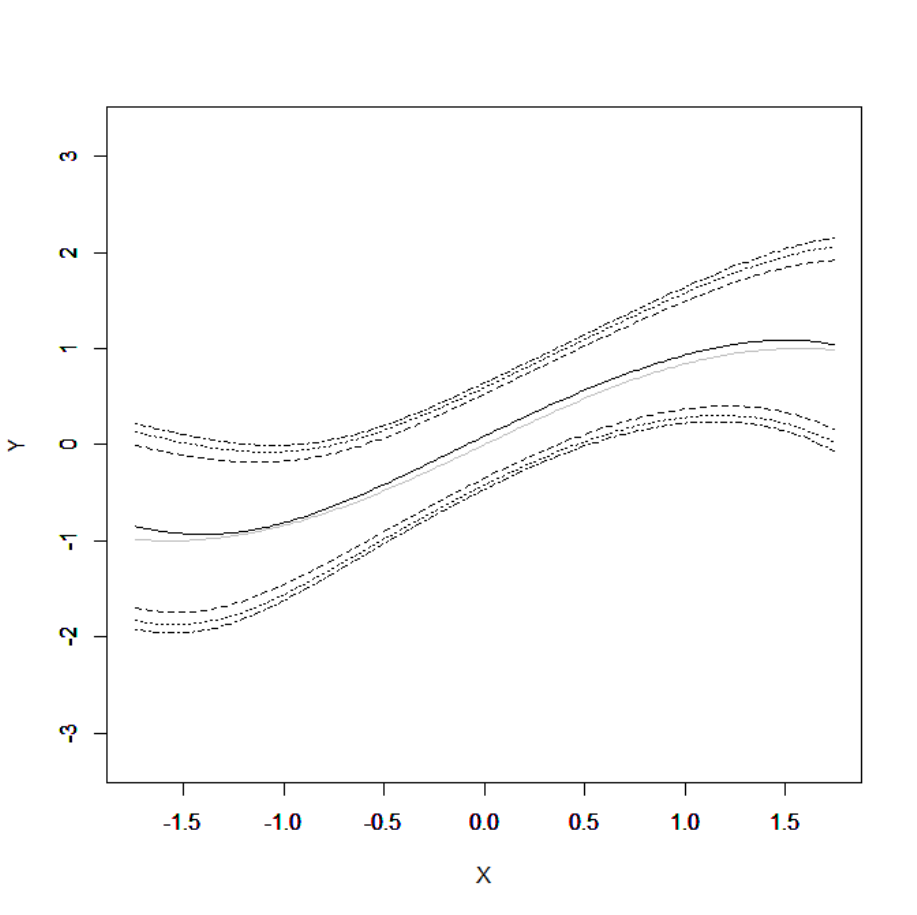}\end{minipage} &
		\begin{minipage}{0.3\textwidth}\includegraphics[width=1\textwidth]{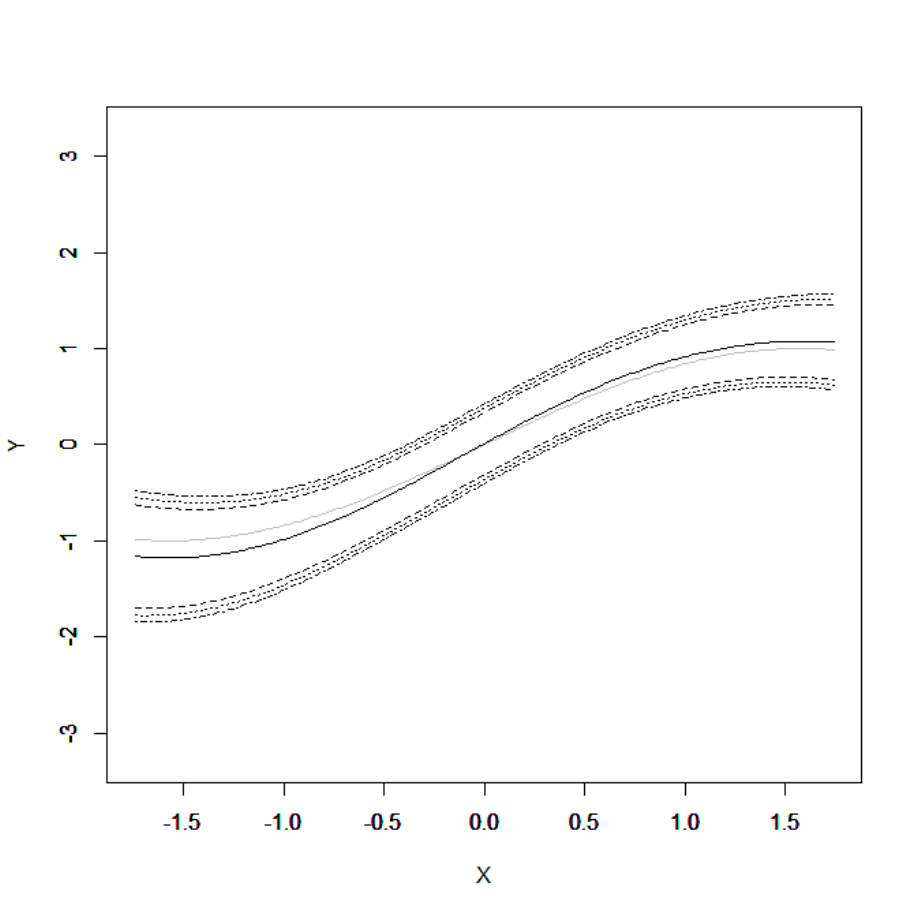}\end{minipage}
		\end{tabular}
	\caption{{\small Confidence bands for $g(x)=\sin(x)$ and $g(x)=x^3$ in Model 2 for error variance ratios of $1/4$ and $1/3$. Gray curves indicate the true function, black solid curves indicate estimates, and dashed curves indicate the 80\%, 90\%, and 95\% confidence bands.}}
	\label{fig:sim_model2_x3_sinx}
\end{figure}

\begin{figure}
	\centering
		\begin{tabular}{ccc}
		&
		$n=200$&
		$n=400$\\
		\begin{minipage}{0.2\textwidth}$g(x)=x$\bigskip\\Model 1\bigskip\\EV=1/4 (25\%)\end{minipage} &
		\begin{minipage}{0.3\textwidth}\includegraphics[width=1\textwidth]{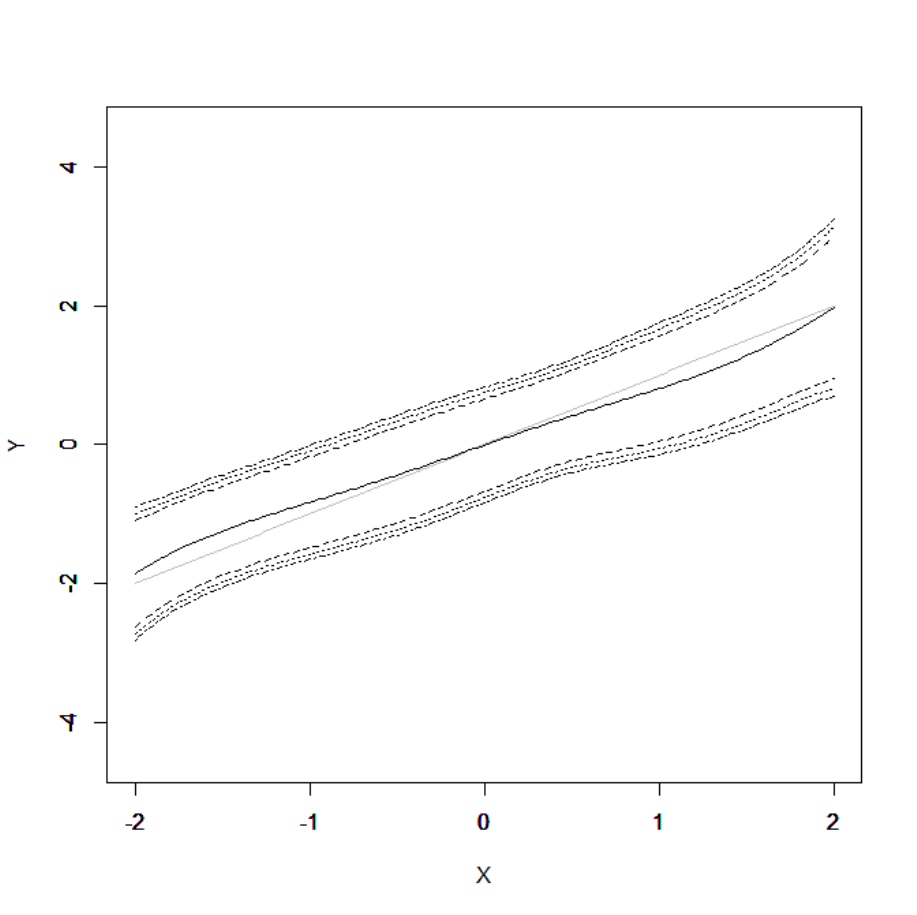}\end{minipage} &
		\begin{minipage}{0.3\textwidth}\includegraphics[width=1\textwidth]{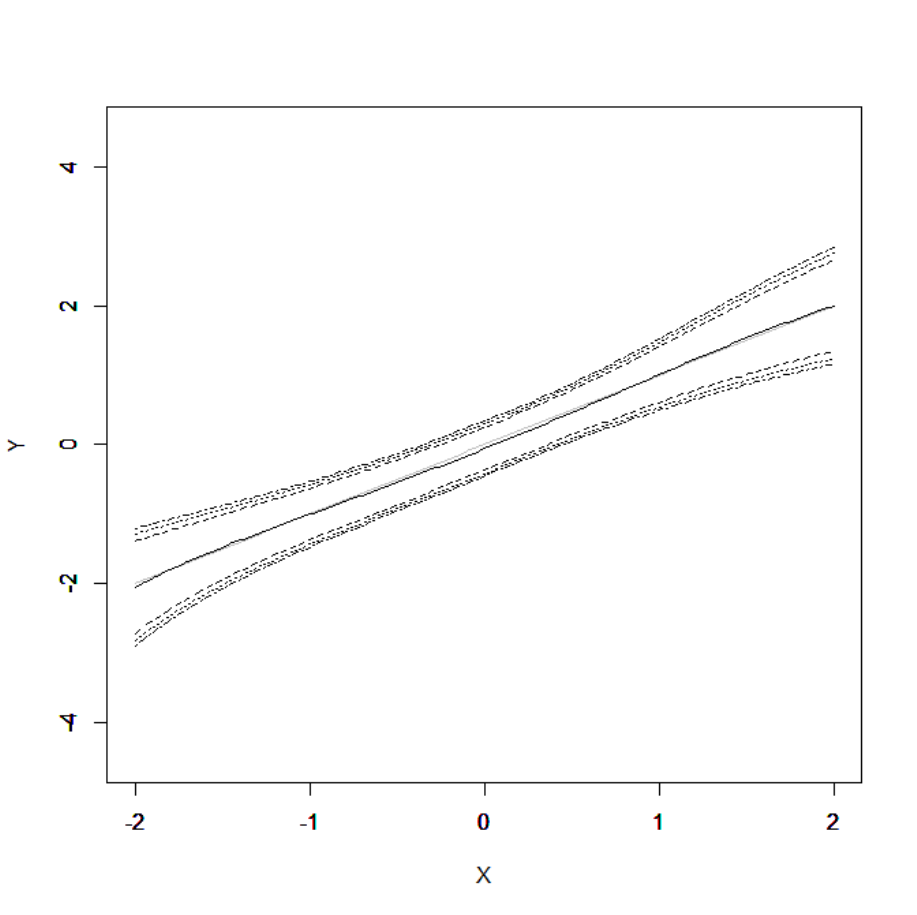}\end{minipage} \\
		\begin{minipage}{0.2\textwidth}$g(x)=x$\bigskip\\Model 1\bigskip\\EV=1/3 (33\%)\end{minipage} &
		\begin{minipage}{0.3\textwidth}\includegraphics[width=1\textwidth]{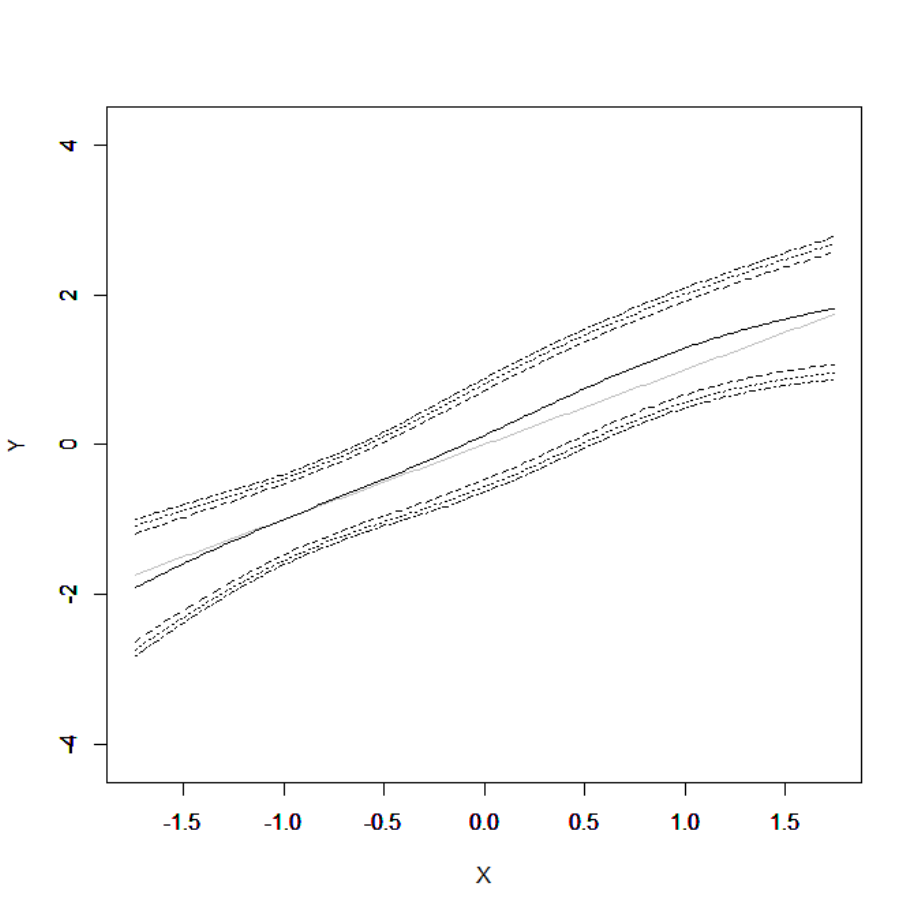}\end{minipage} &
		\begin{minipage}{0.3\textwidth}\includegraphics[width=1\textwidth]{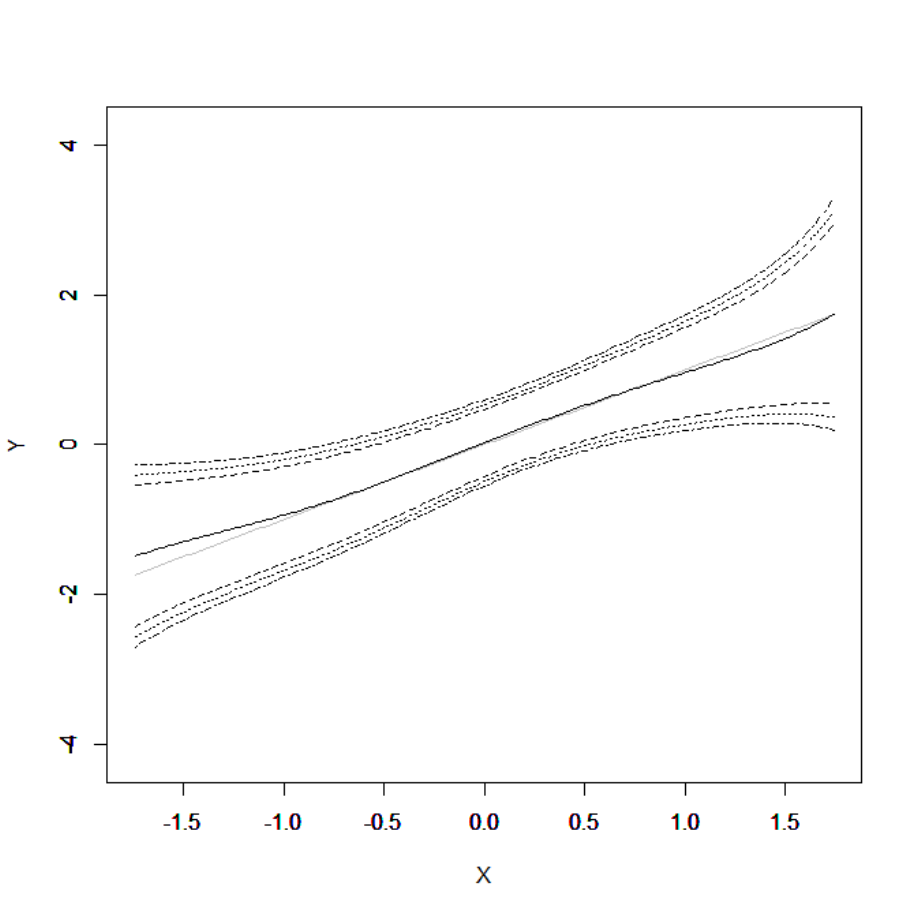}\end{minipage} \\
		\begin{minipage}{0.2\textwidth}$g(x)=x^2$\bigskip\\Model 1\bigskip\\EV=1/4 (25\%)\end{minipage} &
		\begin{minipage}{0.3\textwidth}\includegraphics[width=1\textwidth]{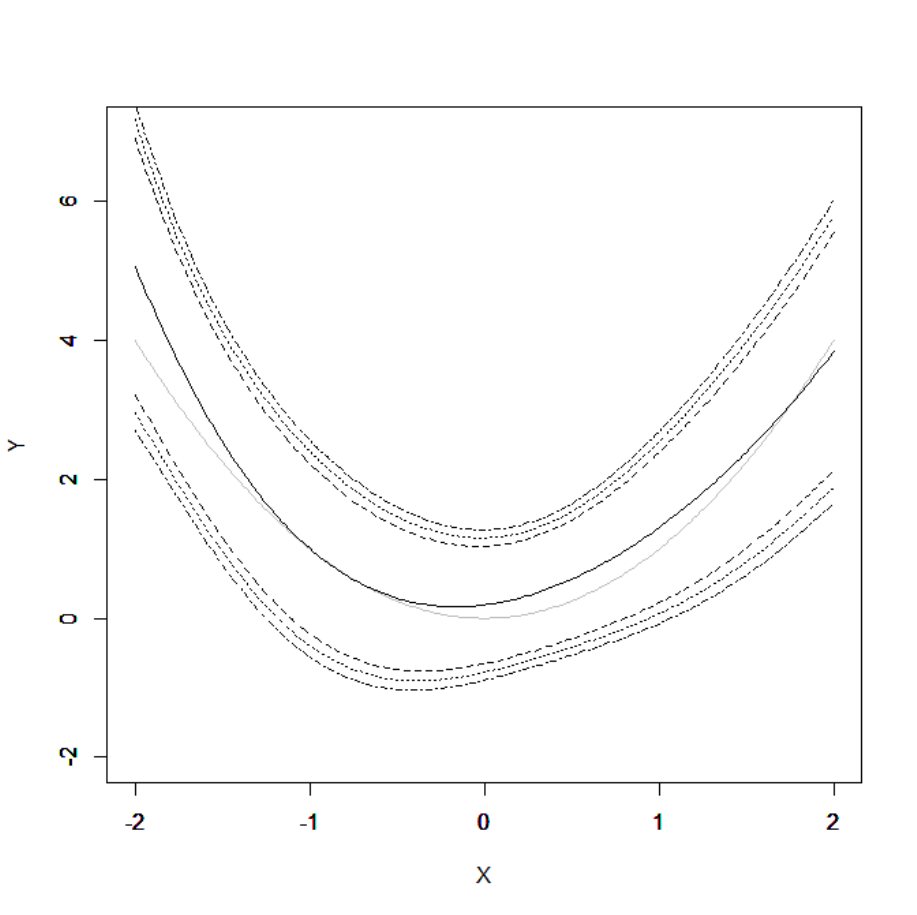}\end{minipage} &
		\begin{minipage}{0.3\textwidth}\includegraphics[width=1\textwidth]{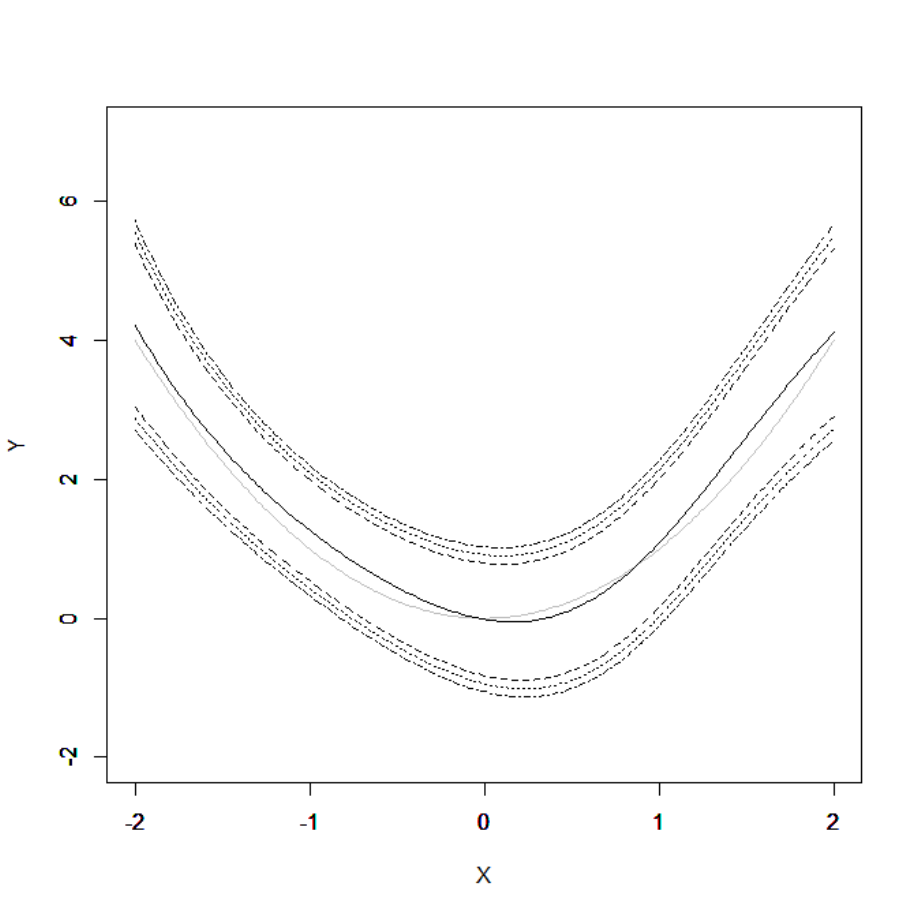}\end{minipage} \\
		\begin{minipage}{0.2\textwidth}$g(x)=x^2$\bigskip\\Model 1\bigskip\\EV=1/3 (33\%)\end{minipage} &
		\begin{minipage}{0.3\textwidth}\includegraphics[width=1\textwidth]{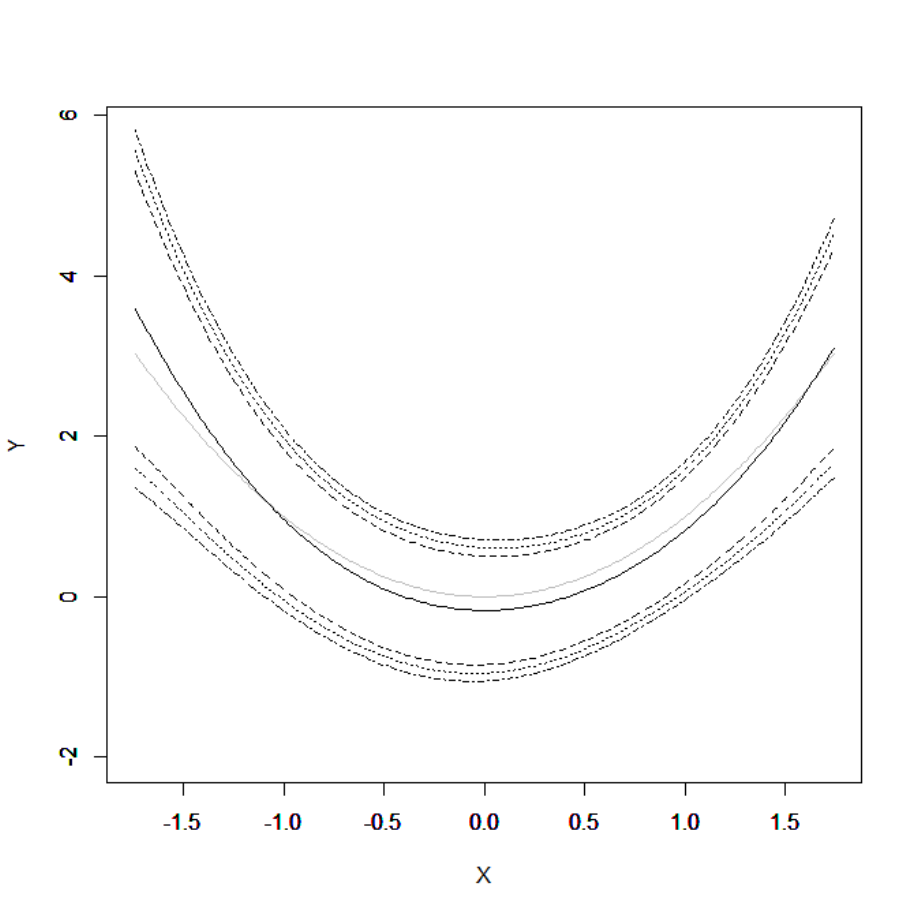}\end{minipage} &
		\begin{minipage}{0.3\textwidth}\includegraphics[width=1\textwidth]{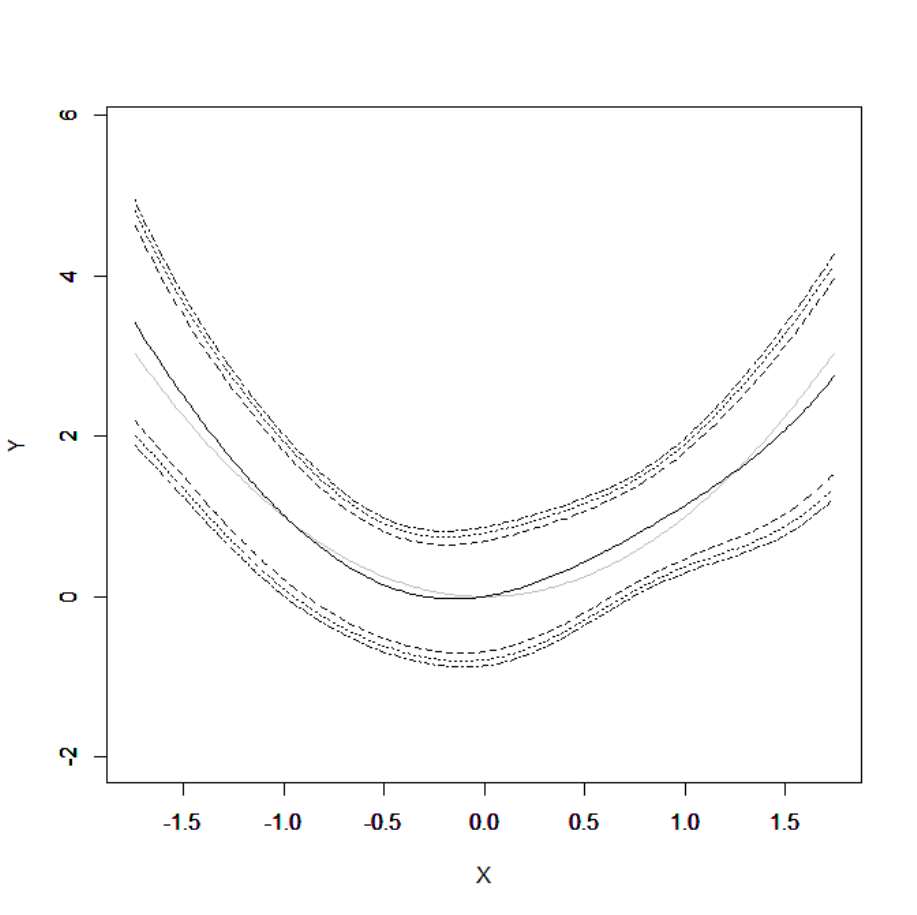}\end{minipage}
		\end{tabular}
	\caption{{\small Confidence bands for $g(x)=x$ and $g(x)=x^2$ in Model 1 for error variance ratios of $1/4$ and $1/3$. Gray curves indicate the true function, black solid curves indicate estimates, and dashed curves indicate the 80\%, 90\%, and 95\% confidence bands.}}
	\label{fig:sim_model1_x1_x2}
\end{figure}

\begin{figure}
	\centering
		\begin{tabular}{ccc}
		&
		$n=200$&
		$n=400$\\
		\begin{minipage}{0.2\textwidth}$g(x)=x$\bigskip\\Model 2\bigskip\\EV=1/4 (25\%)\end{minipage} &
		\begin{minipage}{0.3\textwidth}\includegraphics[width=1\textwidth]{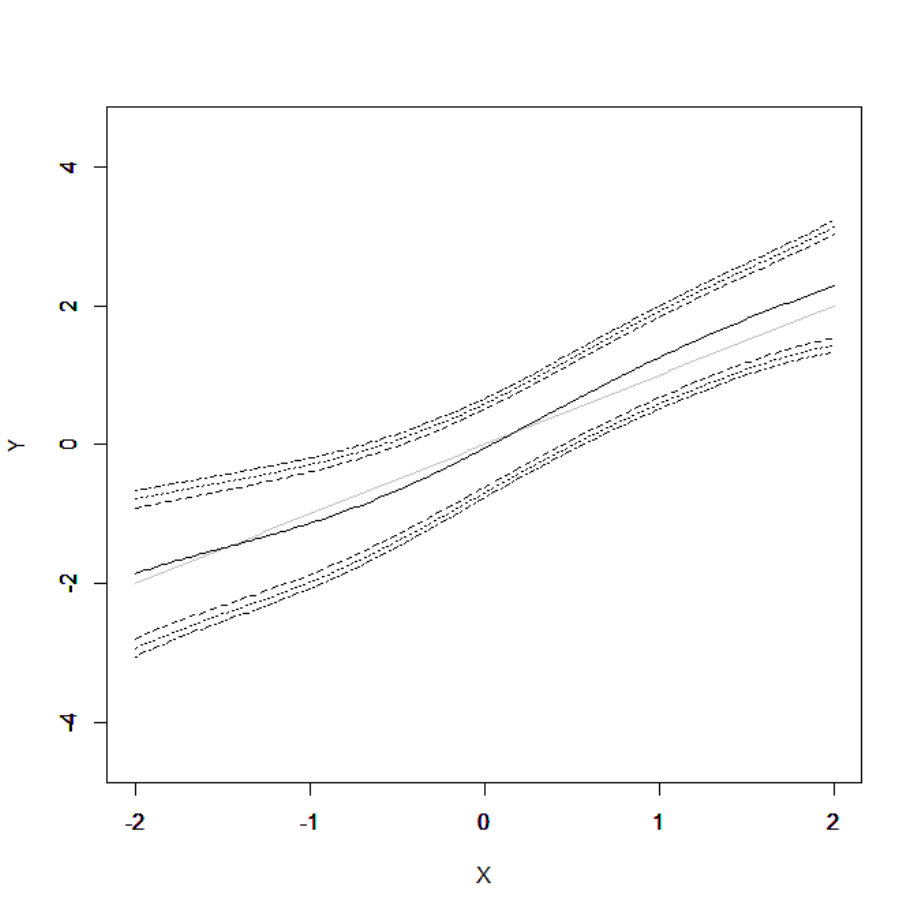}\end{minipage} &
		\begin{minipage}{0.3\textwidth}\includegraphics[width=1\textwidth]{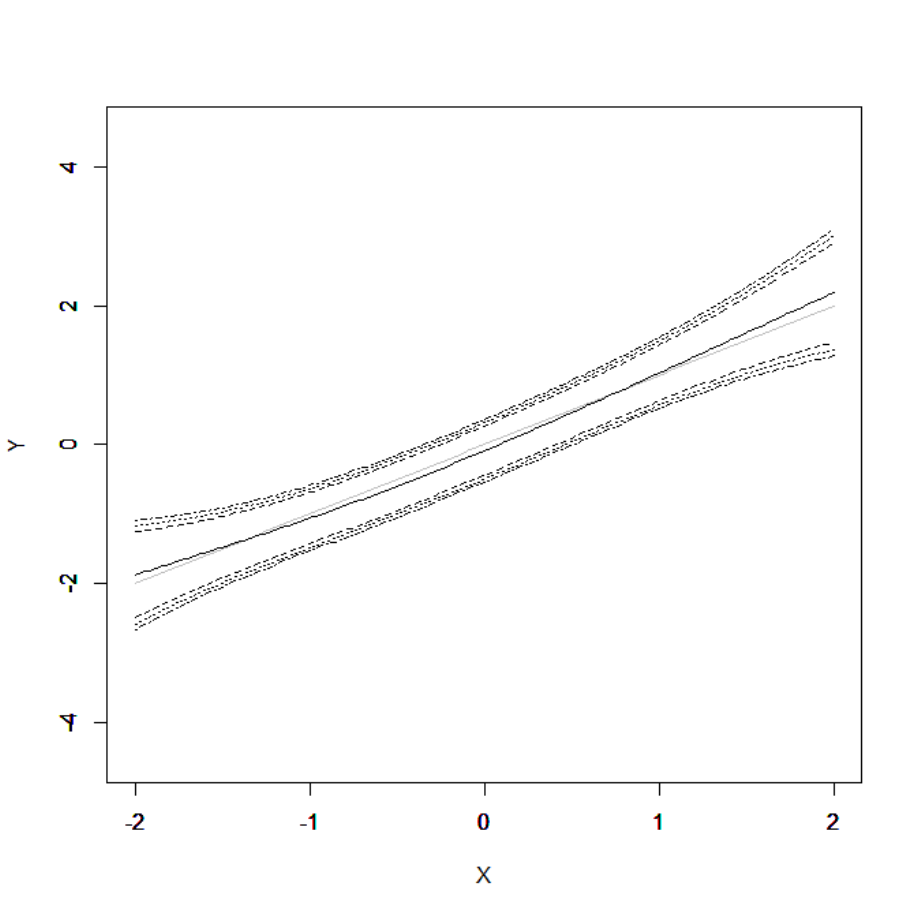}\end{minipage} \\
		\begin{minipage}{0.2\textwidth}$g(x)=x$\bigskip\\Model 2\bigskip\\EV=1/3 (33\%)\end{minipage} &
		\begin{minipage}{0.3\textwidth}\includegraphics[width=1\textwidth]{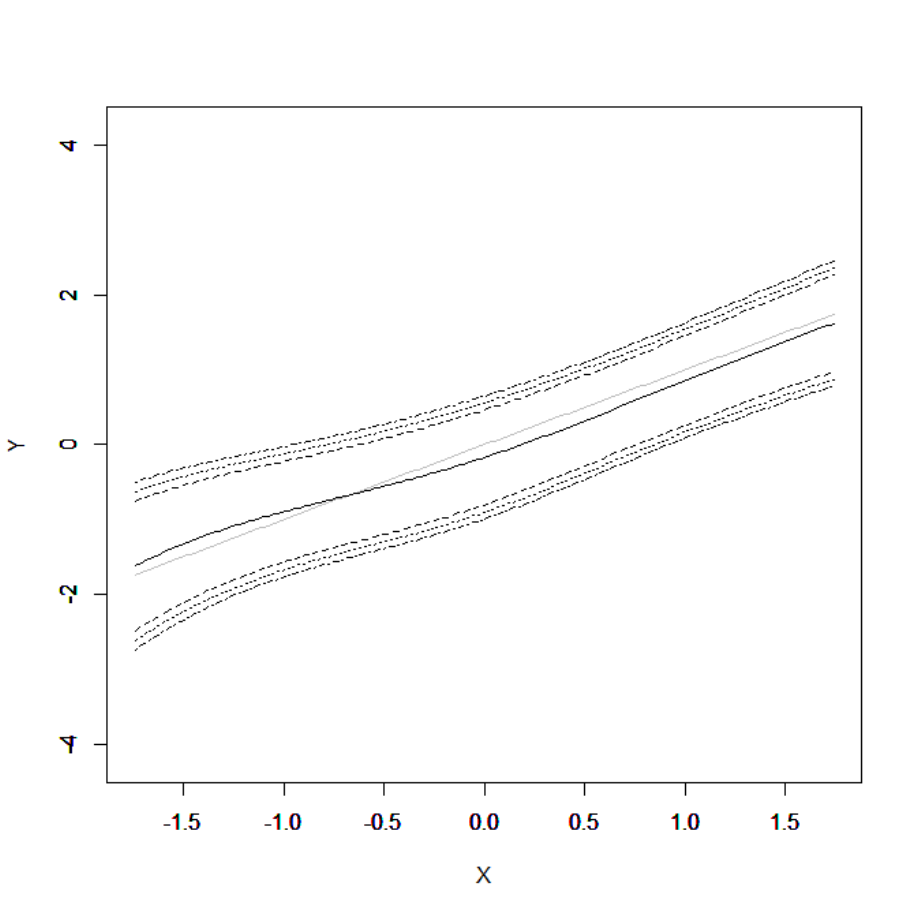}\end{minipage} &
		\begin{minipage}{0.3\textwidth}\includegraphics[width=1\textwidth]{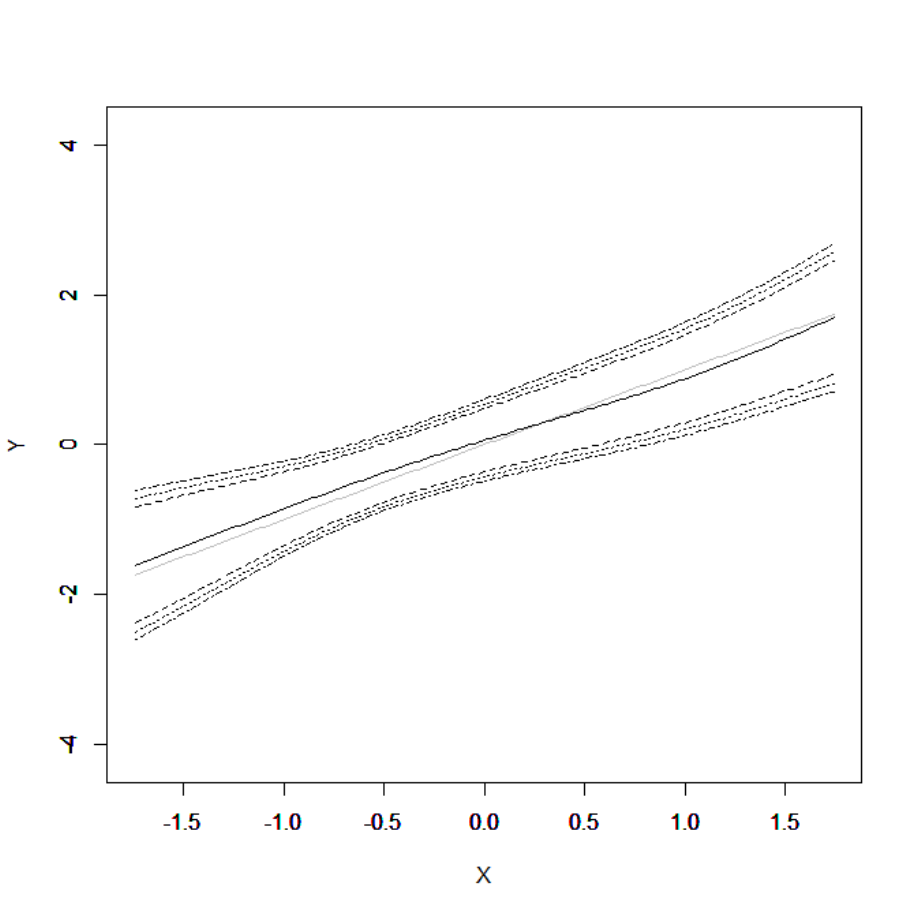}\end{minipage} \\
		\begin{minipage}{0.2\textwidth}$g(x)=x^2$\bigskip\\Model 2\bigskip\\EV=1/4 (25\%)\end{minipage} &
		\begin{minipage}{0.3\textwidth}\includegraphics[width=1\textwidth]{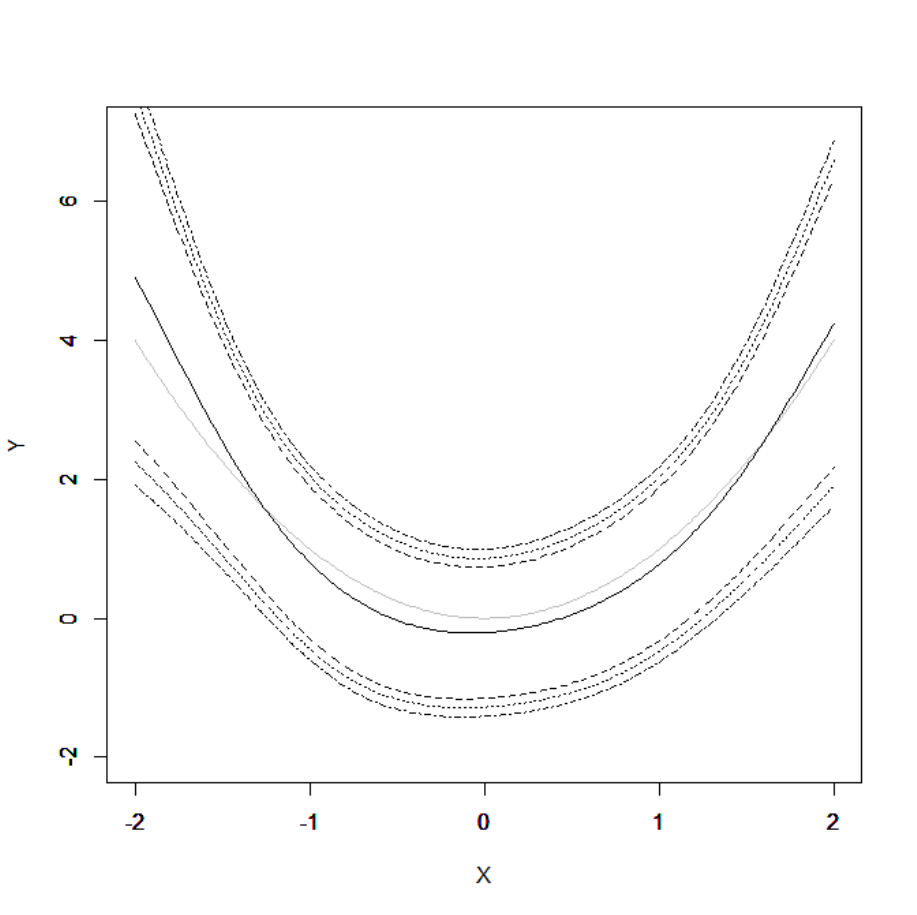}\end{minipage} &
		\begin{minipage}{0.3\textwidth}\includegraphics[width=1\textwidth]{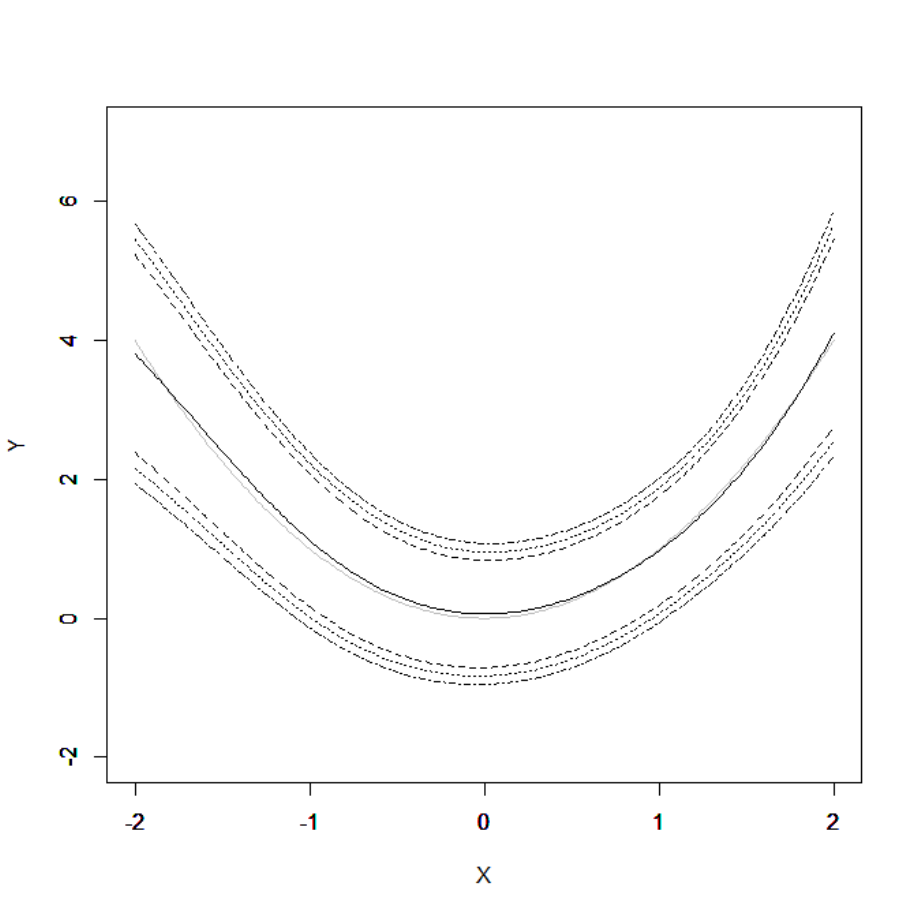}\end{minipage} \\
		\begin{minipage}{0.2\textwidth}$g(x)=x^2$\bigskip\\Model 2\bigskip\\EV=1/3 (33\%)\end{minipage} &
		\begin{minipage}{0.3\textwidth}\includegraphics[width=1\textwidth]{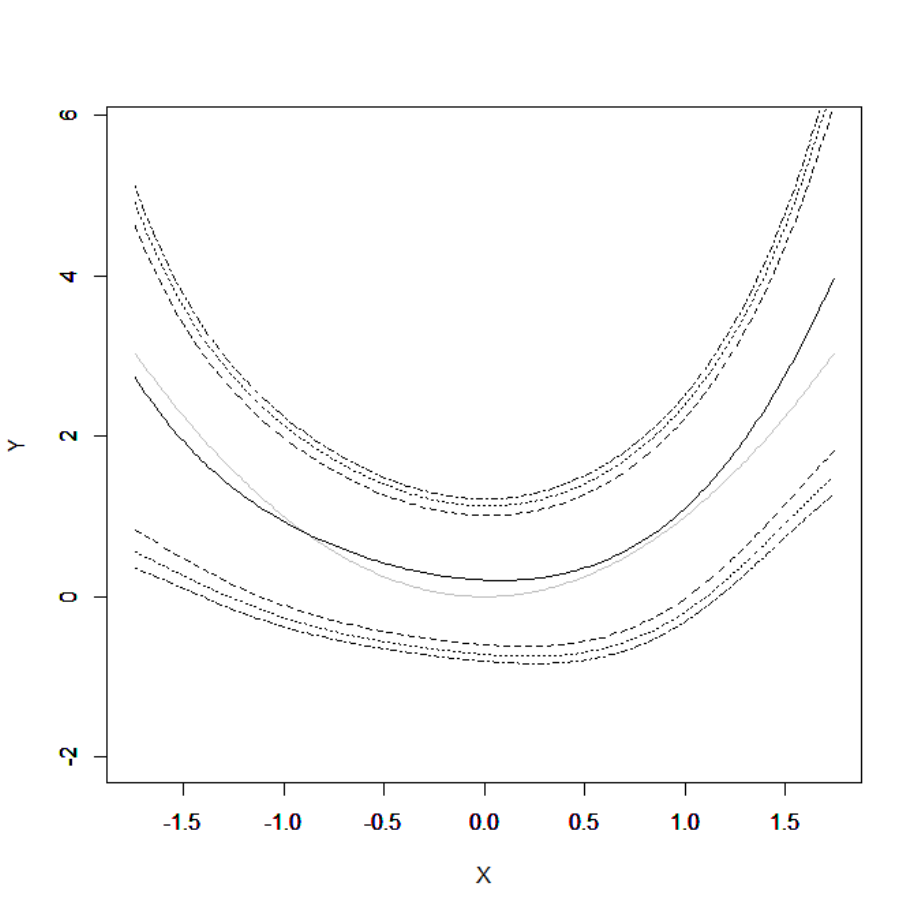}\end{minipage} &
		\begin{minipage}{0.3\textwidth}\includegraphics[width=1\textwidth]{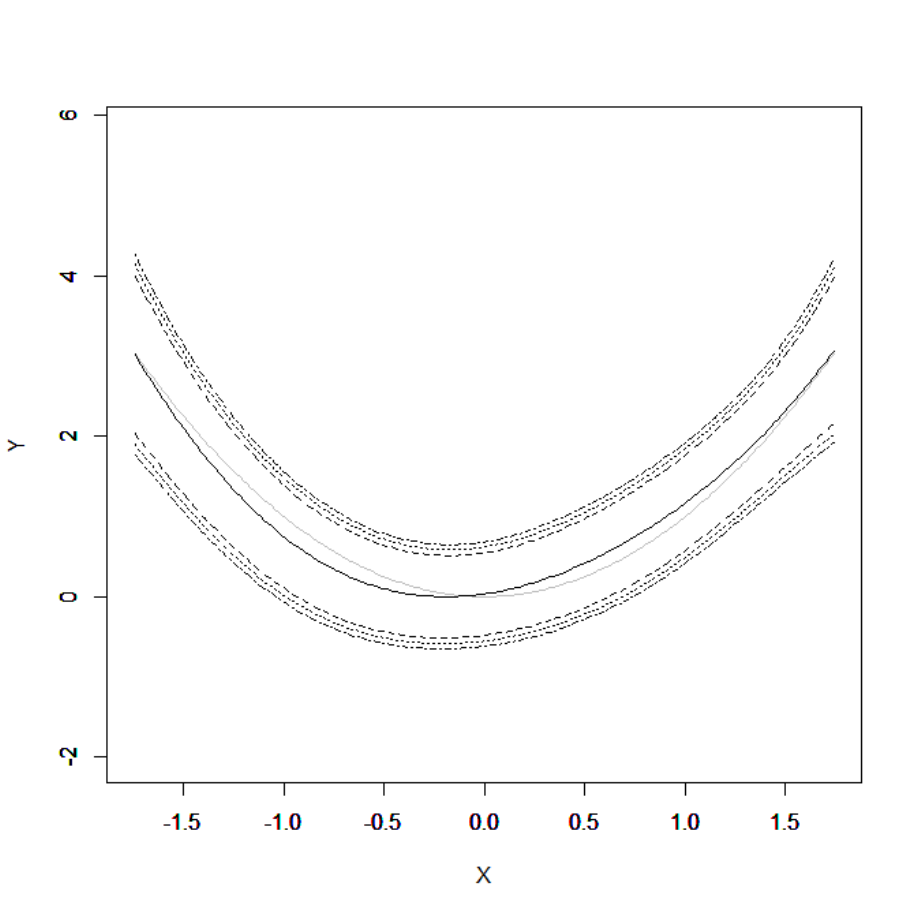}\end{minipage}
		\end{tabular}
	\caption{{\small Confidence bands for $g(x)=x$ and $g(x)=x^2$ in Model 2 for error variance ratios of $1/4$ and $1/3$. Gray curves indicate the true function, black solid curves indicate estimates, and dashed curves indicate the 80\%, 90\%, and 95\% confidence bands.}}
	\label{fig:sim_model2_x1_x2}
\end{figure}

\begin{figure}
	\centering
		\begin{tabular}{ccc}
		&
		$n=200$&
		$n=400$\\
		\begin{minipage}{0.2\textwidth}$g(x)=\cos(x)$\bigskip\\Model 1\bigskip\\EV=1/4 (25\%)\end{minipage} &
		\begin{minipage}{0.3\textwidth}\includegraphics[width=1\textwidth]{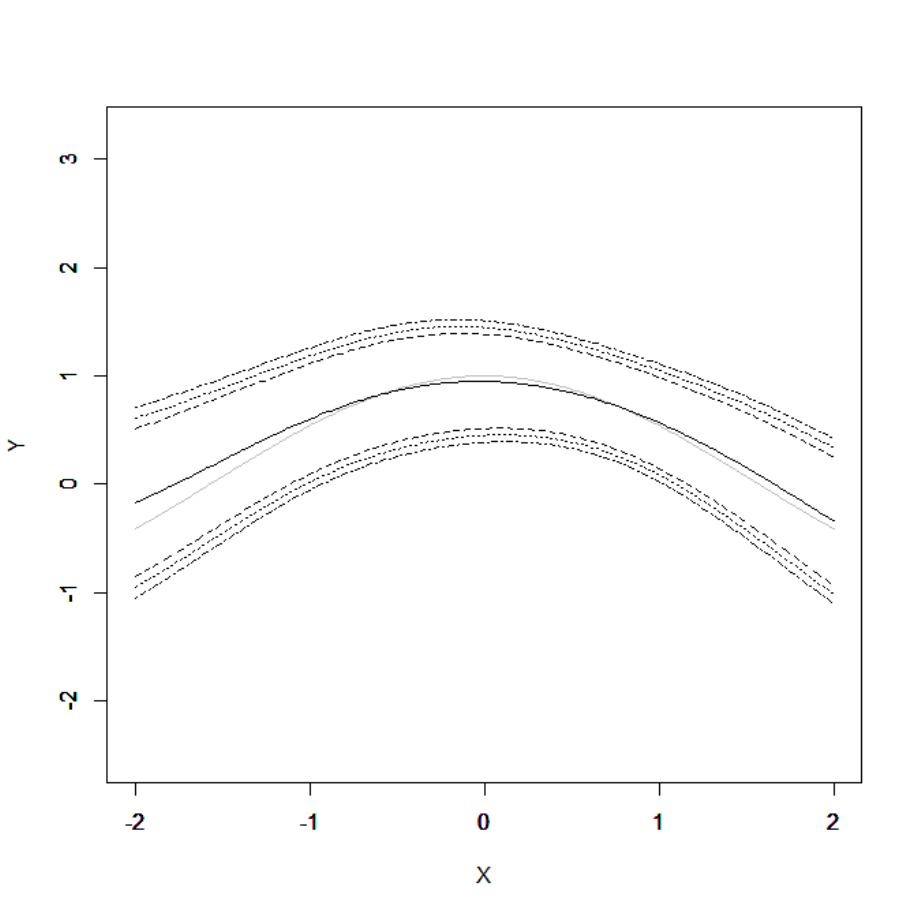}\end{minipage} &
		\begin{minipage}{0.3\textwidth}\includegraphics[width=1\textwidth]{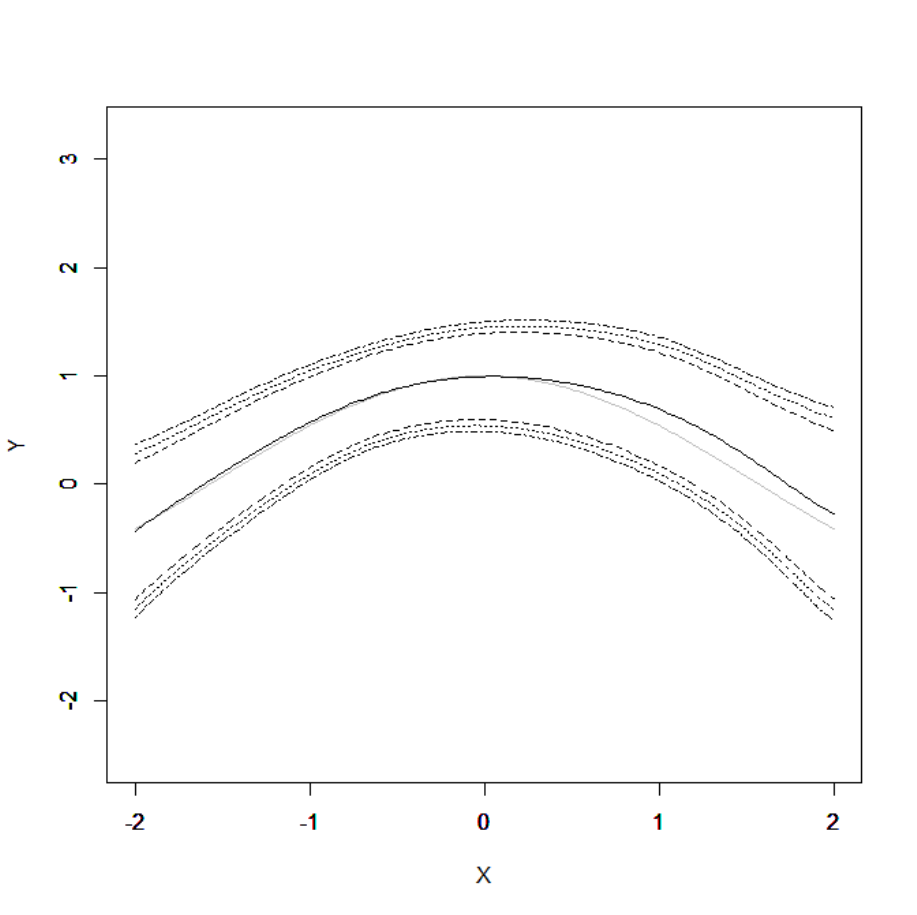}\end{minipage} \\
		\begin{minipage}{0.2\textwidth}$g(x)=\cos(x)$\bigskip\\Model 1\bigskip\\EV=1/3 (33\%)\end{minipage} &
		\begin{minipage}{0.3\textwidth}\includegraphics[width=1\textwidth]{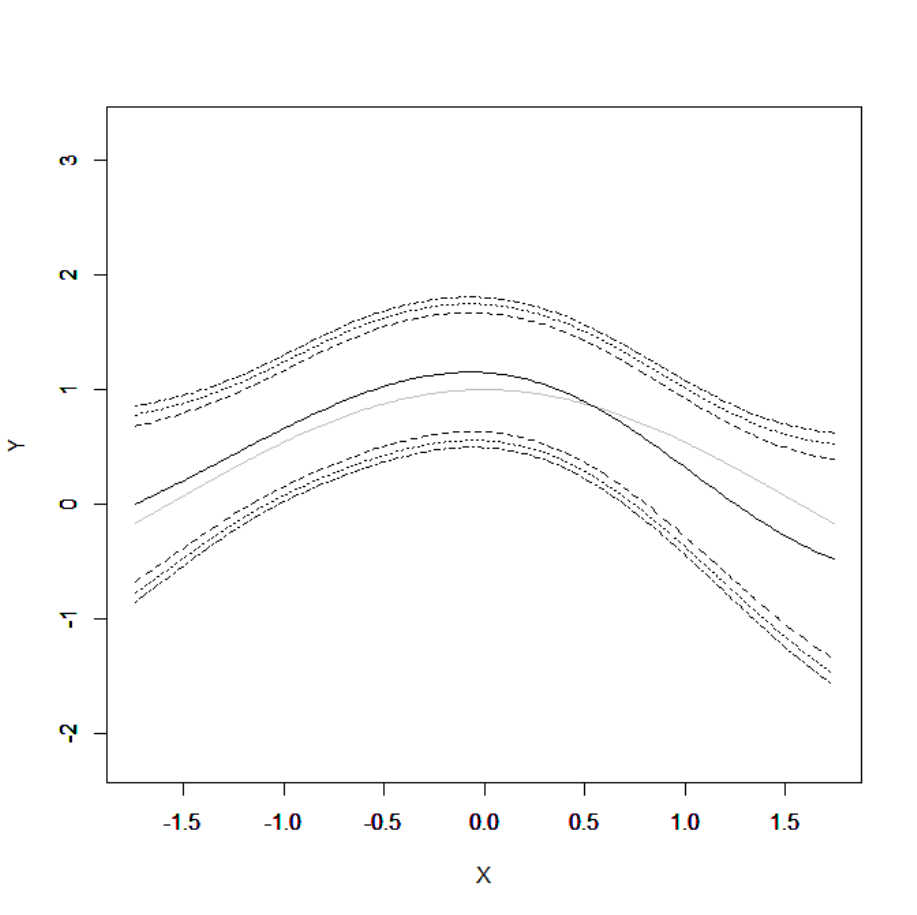}\end{minipage} &
		\begin{minipage}{0.3\textwidth}\includegraphics[width=1\textwidth]{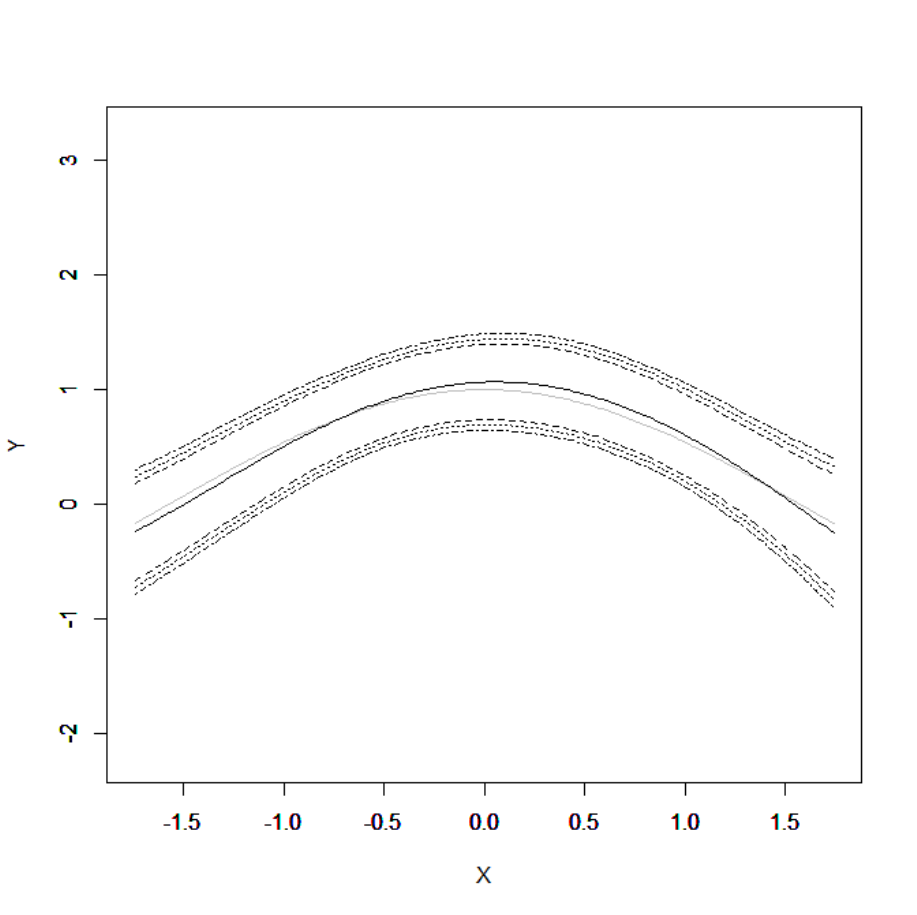}\end{minipage} \\
		\begin{minipage}{0.2\textwidth}$g(x)=\cos(x)$\bigskip\\Model 2\bigskip\\EV=1/4 (25\%)\end{minipage} &
		\begin{minipage}{0.3\textwidth}\includegraphics[width=1\textwidth]{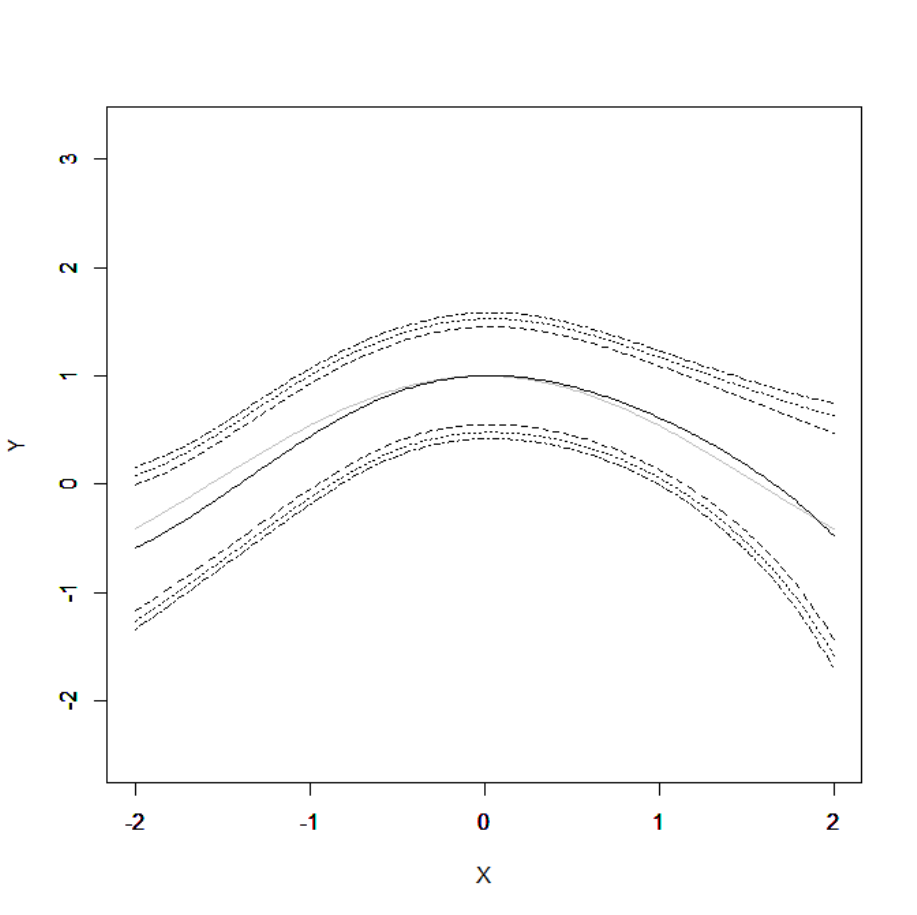}\end{minipage} &
		\begin{minipage}{0.3\textwidth}\includegraphics[width=1\textwidth]{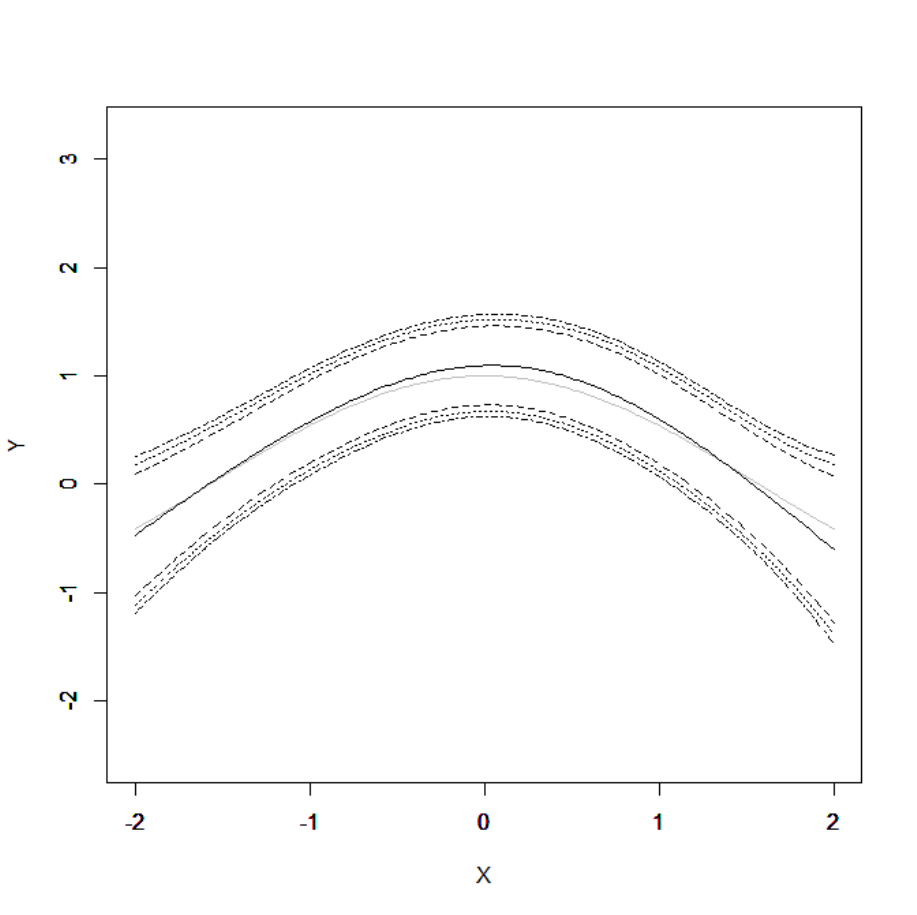}\end{minipage} \\
		\begin{minipage}{0.2\textwidth}$g(x)=\cos(x)$\bigskip\\Model 2\bigskip\\EV=1/3 (33\%)\end{minipage} &
		\begin{minipage}{0.3\textwidth}\includegraphics[width=1\textwidth]{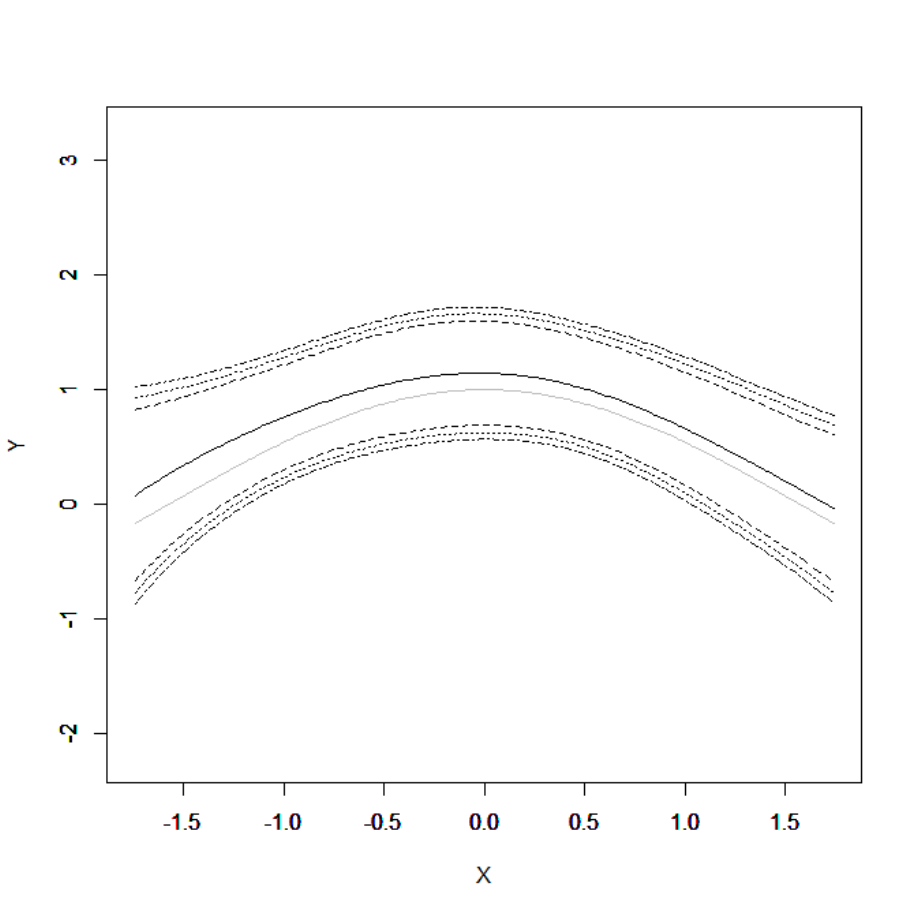}\end{minipage} &
		\begin{minipage}{0.3\textwidth}\includegraphics[width=1\textwidth]{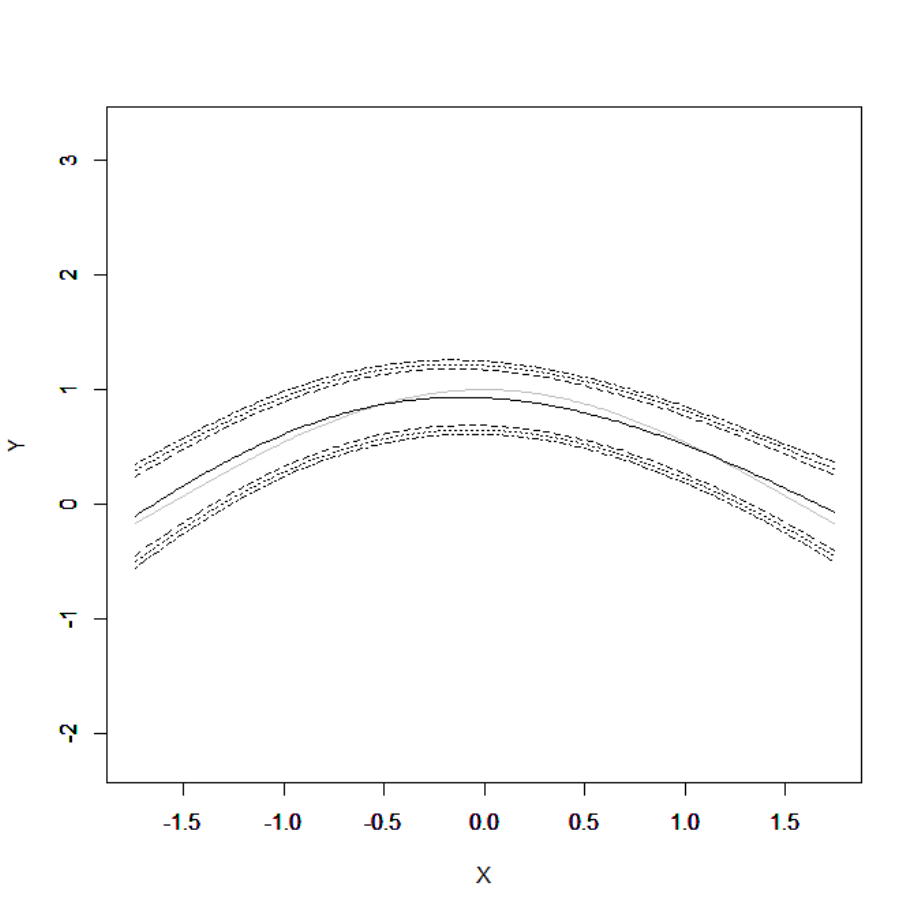}\end{minipage}
		\end{tabular}
	\caption{{\small Confidence bands for $g(x)=\cos(x)$ in Model 1 and Model 2 for error variance ratios of $1/4$ and $1/3$. Gray curves indicate the true function, black solid curves indicate estimates, and dashed curves indicate the 80\%, 90\%, and 95\% confidence bands.}}
	\label{fig:sim_model1_model2_cosx}
\end{figure}

\section{Additional details of real data analysis}
\label{sec: additional details}

\subsection{The background of the empirical study}

According to Centers for Disease Control and Prevention (CDC) of the US Department of Health and Human Services, more than one-third (36.5\%) of US adults have obesity (defined by body mass index or BMI $>30$) in the period between 2011 and 2014 \citep{OgCaFrFl15}.
The estimated annual medical cost of obesity in the United States was 147 billion 2008 U.S. dollars, with the medical costs for people who are obese being \$1,429 higher than those of normal weight \citep{FiTrCoDi09}.
While there is an extensive body of literature on cost estimation of obesity, it is a limitation that commonly used data sets contain only self-reported body measures, and hence the values of BMI generated from them are prone to biases \citep{BoBrMa02}.
More recently, \cite{CaMe12} use the instrumental variable approach to address this issue in cost estimation of obesity.
In the main text of this paper as well as in the following subsection, we employ our data combination approach to treat the self-reporting errors, and draw confidence bands for nonparametric regressions of medical costs on BMI.
We focus on costs measured by medical expenditures.
With this said, we note that there are also indirect costs of obesity which we do not account for, e.g., the costs of obesity are known to be passed on to obese workers with employer-sponsored health insurance in the form of lower cash wages and labor market discrimination against obese job seekers by insurance-providing employers \citep{BhBu09} -- see also \cite{Ca04}.
Details of the two data sets which we combine are as follows.

The National Health and Nutrition Examination Survey (NHANES) of CDC contains data of survey responses, medical examination results, and laboratory test results.
The survey responses include demographic characteristics, such as gender and age.
In addition to the demographic characteristics, the survey responses also contain self-reported body measures and self-reported health conditions.
Among the self reported body measures are height in inches and weight in pounds.
These two variables allow us to construct the BMI in lbs/in$^2$ as a generated variable.
We convert this unit into the metric unit (kg/m$^2$).
The NHANES also contains medical examination results, including clinically measured BMI in kg/m$^2$.
We treat the BMI constructed from the self-reported body measures as $W_j$, and the clinically measured BMI as $X_j$.
From the NHANES as a validation data set of size $m$, we can compute $\eta_j = W_j - X_j$ for each $j=1,\dots,m$.

The Panel Survey of Income Dynamics (PSID) is a longitudinal panel survey of American families conducted by the Survey Research Center at the University of Michigan.
This data set contains a long list of variables including demographic characteristics, socio-economic attributes, expenses, and health conditions, among others.
In particular, the PSID contains self-reported body measures of the household head, including height in inches and weight in pounds.
These two variables allow us to construct the body mass index (BMI) in lbs/in$^2$  as a generated variable.
Again, we convert this unit into the metric unit (kg/m$^2$).
The PSID also contains medical and prescription expenses.
We treat the BMI constructed from the self-reported body measures as $W_j$, and the medical and prescription expenses as $Y_j$.
We note that the information contained in the PSID are mostly at the household level, as opposed to the individual level, and thus $Y_j$ indicates the total medical and prescription expenses of household $j$.
To focus on the individual medical and prescription expenses rather than household expenses, we only consider the sub-sample of the households of single men with no dependent family, for which the total medical and prescription expenses of the household equal to the individual medical and prescription expenses of the household head.
Hence, the reported regression results concern these selected subpopulations.

After deleting observations with missing fields from the NHANES 2009-2010, we obtain the following sample sizes of these four subsamples: (a) $m=407$, (b) $m=435$, (c) $m=407$, and (d) $m=431$.
After deleting observations with missing fields from the PSID 2009 for prescription expenses as the dependent variable $Y$, we obtain the following sample sizes of these four subsamples: (a) $n=528$, (b) $n=243$, (c) $n=247$, and (d) $n=106$.

\subsection{Data sources}

The data sets for NHANES are available at:
	\begin{center}
	https://wwwn.cdc.gov/nchs/nhanes/Default.aspx
	\end{center}
Clinical measurements are constructed from ``Examination'' data, while self reports constructed from from ``Questionnaire'' data.
Unique keys are available for combining these two pieces of data sets.
The data set for PSID is available at:
	\begin{center}
	https://simba.isr.umich.edu/default.aspx
	\end{center}
We are also happy to provide cleaned data files for immediate use for the sample used for our paper upon request.

\subsection{Additional empirical results}

In the main text of the paper, we present empirical results of nonparametric regressions of prescription expenses on BMI.
We also ran  nonparametric regressions of total medical expenses on BMI.
After deleting observations with missing fields from the PSID 2009 for total medical expenses as the dependent variable $Y$, we obtain the following sample sizes of the four subsamples: (a) $n=413$, (b) $n=181$, (c) $n=180$, and (d) $n=64$.

Figure \ref{fig:application_medical} displays estimates and confidence bands for total medical expenses in 2009 US dollars as the dependent variable.
The estimates are indicated by solid black curves.
The areas shaded by gray-scaled colors indicate 80\%, 90\%, and 95\% confidence bands.
The four parts of the figure represent (a) men aged from 20 to 34, (b) men aged from 35 to 49, (c) men aged from 50 to 64, and (d) men aged 65 or above.
If we look at the 90\% confidence band for the group (c) of men aged from 50 to 64, annual average total medical expenses are
approximately \$0--\$15,561 if BMI $=20$,
approximately \$1,235--\$35,885 if BMI $=25$, and
approximately \$8,036--\$34,465 if BMI $=30$.

\begin{figure}
	\centering
	\begin{tabular}{cc}
		(a) Men Aged from 20 to 34 & (b) Men Aged from 35 to 49 \\
		\includegraphics[width=0.32\textwidth]{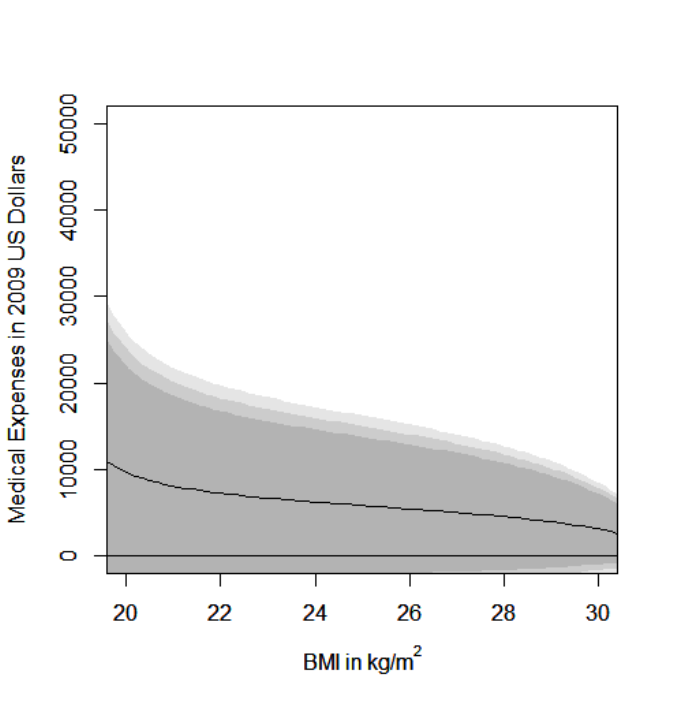} &
		\includegraphics[width=0.32\textwidth]{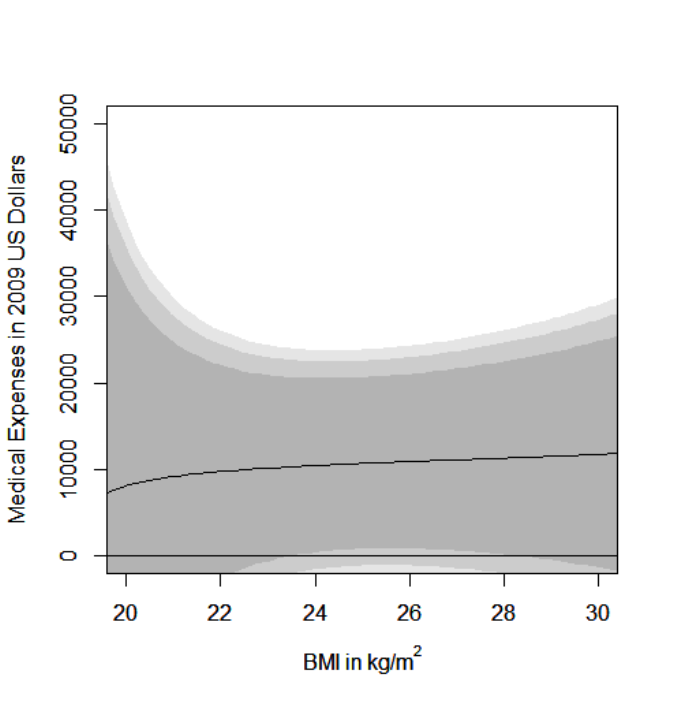} \\
		(c) Men Aged from 50 to 64 & (d) Men Aged 65 or Above \\
		\includegraphics[width=0.32\textwidth]{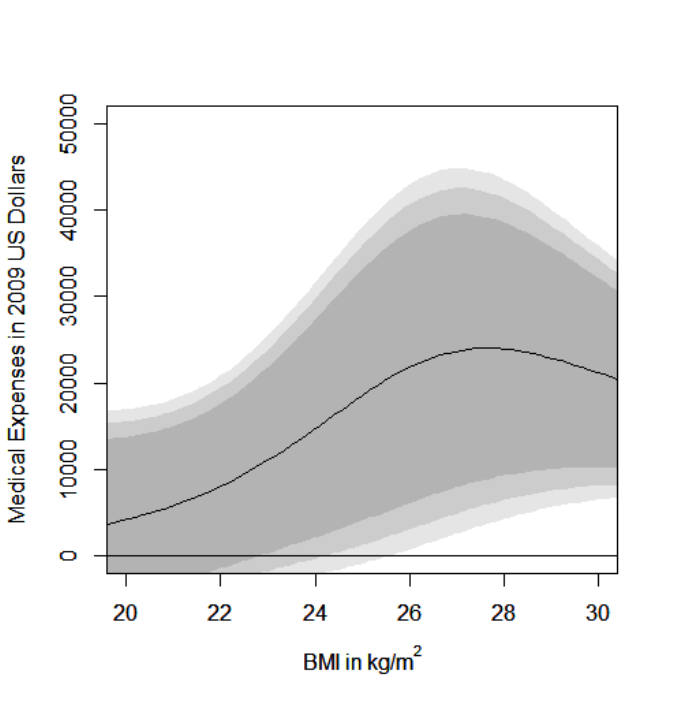} &
		\includegraphics[width=0.32\textwidth]{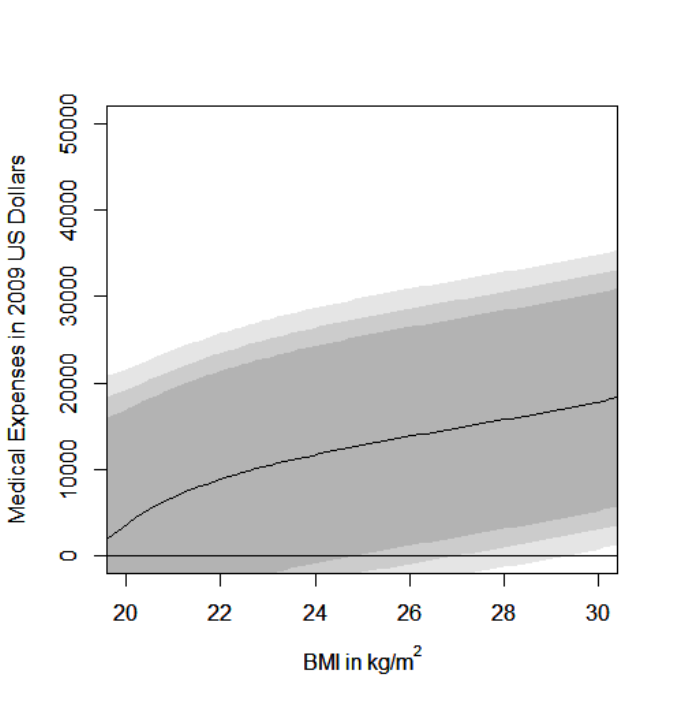} \\
	\end{tabular}
	\caption{{\small Estimates and confidence bands for the nonparametric regression of medical expenses on BMI for (a) men aged from 20 to 34, (b) men aged from 35 to 49, (c) men aged from 50 to 64, and (d) men aged 65 or above. The horizontal axes measure the BMI in kg/m$^2$. The vertical axes measure the medical expenses in 2009 US dollars. The estimates are indicated by solid black curves. The areas shaded by gray-scaled colors indicate 80\%, 90\%, and 95\% confidence bands.}}
	\label{fig:application_medical}
\end{figure}

\section*{Acknowledgments}
We would like to thank Tatsushi Oka and Holger Dette for useful comments and suggestions. 
We also would like to thank the Editor Yacine Ait-Sahalia, an AE, and two anonymous referees for their constructive comments that helped improve the quality of the paper.

\end{document}